\documentclass[11pt]{amsart}
\usepackage[utf8]{inputenc}

\usepackage{xcolor}
\usepackage[babel=true]{csquotes}
\usepackage{amssymb}
\usepackage{amsmath}
\usepackage{graphicx}
\usepackage{amsthm}
\usepackage{amscd}
\usepackage[all]{xy}               
\usepackage{epsfig,exscale}
\usepackage{tikz-cd}

\newcommand\red[1]{\textcolor{red}{#1}}

\newcommand\orange[1]{\textcolor{orange}{#1}}

 \font \eightrm=cmr8

 \newcommand{\nc}{\newcommand}

\newtheorem{thm}{Theorem}
\newtheorem{exam}{Example}

\newtheorem{cor}[thm]{Corollary}

\newtheorem{lem}[thm]{Lemma}
\newtheorem{prop}[thm]{Proposition}
\newtheorem{defn}{Definition}
\newtheorem{rmk}[thm]{Remark}

\def\proofname{Proof}

\newcommand{\cC}{\mathcal{C}} 
\newcommand{\cD}{\mathcal {D}}
\newcommand{\cE}{\mathcal{E}} 
\newcommand{\cF}{\mathcal{F}}
\newcommand{\cG}{\mathcal{G}}

\newcommand{\cL}{\mathcal{L}}
\newcommand{\cM}{\mathcal{M}}

\newcommand{\cU}{\mathcal{U}} 
\newcommand{\cV}{\mathcal{V}} 
\newcommand{\cW}{\mathcal{W}}

\newcommand{\cZ}{\mathcal{Z}}

\nc{\ignore}[1]{{}}
\nc{\mrm}[1]{{\rm #1}}
\nc{\dirlim}{\displaystyle{\lim_{\longrightarrow}}\,}
\nc{\invlim}{\displaystyle{\lim_{\longleftarrow}}\,}
\nc{\vep}{\varepsilon} \nc{\ep}{\epsilon}
\nc{\sigmat}{\widetilde\sigma}
\nc{\ostar}{\overline{*}}
\nc{\mchar}{\mrm{Char}}
\nc{\Hom}{\mrm{Hom}}
\nc{\id}{\mrm{id}}

\nc{\remark}{\noindent{\bf{Remark:}}}
\nc{\remarks}{\noindent{\bf{Remarks:}}}

 \nc{\delete}[1]{}
 \nc{\grad}[1]{^{({#1})}}
 \nc{\fil}[1]{_{#1}}

\nc{\BA}{{\mathbb A}} \nc{\CC}{{\mathbb C}} \nc{\DD}{{\mathbb D}}
\nc{\EE}{{\mathbb E}} \nc{\FF}{{\mathbb F}} \nc{\GG}{{\mathbb G}}
\nc{\HH}{{\mathbb H}} \nc{\LL}{{\mathbb L}} \nc{\NN}{{\mathbb N}}
\nc{\PP}{{\mathbb P}} \nc{\QQ}{{\mathbb Q}} \nc{\RR}{{\mathbb R}}
\nc{\TT}{{\mathbb T}} \nc{\VV}{{\mathbb V}} \nc{\ZZ}{{\mathbb Z}}
\nc{\setr}{\RR} \nc{\setn}{\NN}
\nc{\Cal}[1]{{\mathcal {#1}}}
\nc{\mop}[1]{\mathop{\hbox {\rm #1} }}
\nc{\smop}[1]{\mathop{\hbox {\eightrm #1} }}
\nc{\mopl}[1]{\mathop{\hbox {\rm #1} }\limits}
\nc{\frakg}{{\mathfrak g}}
\nc{\g}[1]{{\mathfrak {#1}}}
\def \restr#1{\mathstrut_{\textstyle |}\raise-8pt\hbox{$\scriptstyle #1$}}
\def \srestr#1{\mathstrut_{\scriptstyle |}\hbox to
  -1.5pt{}\raise-4pt\hbox{$\scriptscriptstyle #1$}}
\nc{\wt}{\widetilde}
\nc{\wh}{\widehat}
\nc{\un}{\hbox{\bf 1}}
\nc{\redtext}[1]{\textcolor{red}{{\tt [[#1]]}}}
\nc{\bluetext}[1]{\textcolor{blue}{#1}}

\nc\fleche[1]{\mathop{\hbox to #1 mm{\rightarrowfill}}\limits}
\def\semi{\!\mathrel{\times}\kern -6.5 pt\joinrel\mathrel{\raise
    1.4pt\hbox{${\scriptscriptstyle |}$}}\,\,}

\nc{\gmop}[1]{\mathop{\hbox {\bf #1} }\nolimits}

\def\fleche#1{\mathop{\hbox to #1 mm{\rightarrowfill}}\limits}
\def\gfleche#1{\mathop{\hbox to #1 mm{\leftarrowfill}}\limits}
\def\inj#1{\mathop{\hbox to #1 mm{$\lhook\joinrel$\rightarrowfill}}\limits}
\def\ginj#1{\mathop{\hbox to #1 mm{\leftarrowfill$\joinrel\rhook$}}\limits}
\def\surj#1{\mathop{\hbox to #1 mm{\rightarrowfill\hskip 2pt\llap{$\rightarrow$}}}\limits}
\def\gsurj#1{\mathop{\hbox to #1 mm{\rlap{$\leftarrow$}\hskip 2pt
      \leftarrowfill}}\limits}
\def \restr#1{\mathstrut_{\textstyle |}\raise-6pt\hbox{$\scriptstyle #1$}}
\def \srestr#1{\mathstrut_{\scriptstyle |}\hbox to
-1.5pt{}\raise-4pt\hbox{$\scriptscriptstyle #1$}}

\nc{\scal}[2]{ \langle #1\, ,\, #2 \rangle}
\nc{\abs}[1]{ \left\vert  #1 \right\vert}
\nc{\norm}[1]{ \left\lVert  #1 \right\rVert}

\newcommand{\WF}[1]{\mathop{\mathrm{WF}(#1)}\nolimits}

\newcommand{\supp}[1]{{\mathrm{supp}(#1)}}
\newcommand{\rank}{\mathrm{rank}}

\newcommand{\pr}[1]{\mathop{\mathrm{pr}_{#1}}\nolimits}
\newcommand{\Gr}[1]{\mathop{\mathrm{Gr}_{#1}}\nolimits}

\begin{document}
\title[FIO on Lie groupoids]
      {Fourier Integrals operators on Lie groupoids}

\author{Jean-Marie Lescure, St\'ephane Vassout ${}^{(1)}$}
\address{}
 
\footnotetext[1]{The  authors are supported by the  Grant ANR-14-CE25-0012-01  SINGSTAR}

\date{\today
}

\begin{abstract}
As announced in \cite{LMV1}, we develop a calculus of Fourier integral $G$-operators on any Lie groupoid $G$.  For that purpose, we study convolability and invertibility of Lagrangian conic submanifolds of the symplectic groupoid $T^*G$. We also identify those Lagrangian which correspond to  equivariant families parametrized by the unit space $G^{(0)}$ of homogeneous canonical relations  in $(T^*G_x\setminus 0)\times (T^*G^x\setminus 0)$. This allows us to select a subclass of Lagrangian distributions on any Lie groupoid $G$ that deserve the name of Fourier integral $G$-operators ($G$-FIO). By construction, the class of $G$-FIO contains the class of equivariant families of ordinary Fourier integral operators on the manifolds $G_x$,   $x\in G^{(0)}$.  We then develop for $G$-FIO the first stages of the calculus in the spirit of Hormander's work. Finally, we work out an example proving the efficiency of the present approach for studying Fourier integral operators on singular manifolds.   
\end{abstract}

\maketitle

\setcounter{tocdepth}{2}
\tableofcontents

\section{Introduction}
 Lagrangian distributions constitute an important and wide class of distributions, containing for instance the classes of pseudodifferential operators and Fourier integrals operators on manifolds. To formalize properly this type of distributions on a manifold and both to analyse their singularities and develop a calculus, one is led to handle, in addition to Fourier analysis, tools from microlocal analysis and from symplectic geometry. For arbitrary manifolds, all of this is fully achieved in the series of L. H\"ormander's books \cite{Horm-classics}. 

Motivated by analysis on singular spaces, foliations and all other situations where a groupoid encodes a 
suitable pseudodifferential calculus, the purpose of the present work is the study of Lagrangian distributions as defined in \cite{Horm-classics} in the specific situation where the underlying manifold is a Lie groupoid $G$, and then, to be able to propose a class of Fourier integral operators suitable for singular spaces, foliations, etc... 
 
Pseudodifferential operators on general Lie groupoids \cite{MP,NWX} ($G$-PDO) provide a stimulating example: 
they are exaclty the $G$-operators given by equivariant $C^\infty$ families of pseudodifferential operators acting 
in the $r$ or $s$-fibers of $G$. Importantly, the space of $G$-PDOs coincides with the space of Lagrangian distributions on $G$ subordinated to  $A^*G=N^*(G^{(0)})\subset T^*G$, that is, to the dual of the normal bundle of the embedding $G^{(0)}\hookrightarrow G$. 

Another inspiration comes from the paper \cite{LMV1} in which distributions on  Lie groupoids are studied. On one hand, distributions on a Lie groupoid $G$ that provide $G$-operators are charaterized and natural sufficient conditions on their wave front sets are given.  On the other hand, still in \cite{LMV1}, the convolution product of distributions on a Lie groupoid $G$ is analyzed and sufficient conditions for that product to be defined are again given in terms of wave front sets. 
For the understanding of the present work, it is relevant to note that all the conditions mentionned above, as well as the formula for the wave front set of a convolution product of distributions, have an algebraic nature involving the symplectic groupoid $T^*G$.

This background gives a general framework into which one should develop the theory of Fourier integral operators on a groupoid, and brings in natural questions too. 

The least we can require is that a Fourier integral operator should be an element of $I(G,\Lambda)$ \cite[Section 25.1]{Horm-classics}, that is a Lagrangian distribution on the manifold $G$ subordinated to some Lagrangian  $\Lambda\subset T^*G\setminus 0$, and  should also be regular enough to produce an adjointable $G$-operator. The article \cite{LMV1} already gives a  way to fulfill this constraint: if the Lagrangian $\Lambda\subset T^*G\setminus 0$ does not intersect the kernel of source and target maps of $T^*G\rightrightarrows A^*G$, then the elements of $I(G,\Lambda)$ provide adjointable $G$-operators. Note that this is a purely algebraic condition, very simple to check in practice, which boils down to the so-called "no-zeros" condition \cite{Horm-classics,Melrose1981} in the case of the pair groupoid $G=X\times X$. We call $G$-relations the conic Lagrangian submanifolds of $T^*G\setminus 0$ fulfilling this condition, in reference to the classical term "canonical relations",
and we abbreviate the corresponding classes of Lagrangian distributions by $G$-FIO. 

Then, a first natural question arises: given a $G$-relation $\Lambda$ and a $G$-FIO $u\in I(G,\Lambda)$, we have at hand a ($C^\infty$, equivariant) family of distributions $u_x\in \cD'(G^x\times G_x)$, $x\in G^{(0)}$, and so, it is natural to ask whether these distributions are Lagrangian, that is, are ordinary Fourier integral operators on the manifold $G_x$ ? 

The answer is no in general. Actually, the distributions $u_x$ are still given by oscillatory integrals, but we provide an example where some of them are not Lagrangian distributions. This unstable behavior is fixed by imposing  that the underlying $G$-relation $\Lambda\subset T^*G\setminus 0$ has a projection in $G$ transversal to the canonical (singular) foliation of $G$. Indeed, this transversality condition implies that $\Lambda$ gives a ($C^\infty$, equivariant) family of canonical relations $\Lambda_x\in T^*G^x\times G_x$,  $x\in G^{(0)}$, and each $u_x$, being expressed as an appropriate pull-back distribution, is then an element of $I(G^x\times G_x,\Lambda_x)$.  We call family $G$-relations the  $G$-relations enjoying this transversality condition and we abbreviate by $G$-FFIO the corresponding Lagrangian distributions. Note that the transversality property characterizing family $G$-relations among  the general ones is geometric and still very simple to check in practice. 

Thus, by construction, $G$-FFIO provide $C^\infty$ equivariant families of Lagrangian distributions $u_x$, $x\in G^{(0)}$ and the next natural goal is to obtain the converse. This is achieved after proving that $C^\infty$ equivariant family of canonical relations $\Lambda_x\in T^*G^x\times G_x$,  $x\in G^{(0)}$, can be "glued" into a single family $G$-relation $\Lambda\subset T^*G$. This requires a preliminary work on families of Lagrangian submanifolds in the cotangent spaces of the fibers of an arbitrary submersion. 

To summarize the previous discussion,  $G$-FIO provide a class of distributions on $G$  desserving the name of Fourier integral operators on $G$ and among them,  we know how to characterize in a simple geometric way those which correspond to $C^\infty$ equivariant families of ordinary FIO in the $s$ or $r$ fibers of $G$.
 
The next natural point is to develop a calculus for $G$-FIO: adjointness, composition, module structure over the algebra of pseudodifferential operators, Egorov theorem and $C^*$-continuity.  We could have restricted ourselves to the sub-class of $G$-FFIO in order to export all the results available on manifolds to groupoids via the point of view of families. Actually, we have chosen to treat  $G$-FFIO as single distributions on $G$ to develop the calculus: the statements are simpler, more conceptual, and the central role of the  symplectic  groupoid $T^*G$ is enlighted. Moreover, most of the results hold for $G$-FIO and not only for $G$-FFIO.

More precisely, we prove that adjoints of $G$-FIO  are $G$-FIO, and adjunction replaces the Lagrangian by its image under the inverse map of the groupoid $T^*G\rightrightarrows A^*G$. Next, we work out a natural convolability assumption on $G$-relations in order that their convolution in $T^*G$ is again a $G$-relation. Then, when $\Lambda_1$ and $\Lambda_2$ are convolable, we prove that the convolution of any distributions $u_j\in I(G,\Lambda_j)$ (that is, the composition of the corresponding $G$-operators), is a $G$-FIO subordinated to the $G$-relation $\Lambda_1.\Lambda_2$.  We observe that the convolution of family $G$-relations is not always a family $G$-relation and then, convolution of $G$-FFIO are not always  $G$-FFIO: we explain how to strengthen the convolability assumption to fix this problem. 

The previous adjunction and composition theorems have direct applications. 
Firstly, the composition of $G$-FIO (resp. $G$-FFIO) with pseudodifferential operators are $G$-FIO (resp. $G$-FFIO), the Lagrangian being unchanged. Secondly, for any convolable $G$-relations $\Lambda_1,\Lambda_2$  whose  convolution   is contained in the unit space of $T^*G\rightrightarrows A^*G$, that is,  $\Lambda_1.\Lambda_2\subset A^*G$,  we get an obvious statement generalizing Egorov Theorem. 

The assumption made in our version of Egorov Theorem can be viewed as a weak invertibility property for $G$-relations. Actually, for any $G$-relation $\Lambda_1$,  we prove that  the existence of a convolable $G$-relation $\Lambda_2$ such that   $\Lambda_1*\Lambda_2=r_\Gamma(\Lambda_1)$ and $\Lambda_2*\Lambda_1=s_\Gamma(\Lambda_1)$ (here $s_\Gamma, r_\Gamma $ denote the source and target maps of $ T^*G\rightrightarrows A^*G$) is equivalent for $\Lambda_1$ to be a Lagrangian bissection. This is what we call invertible $G$-relations. 

It then follows that for any invertible $G$-relation $\Lambda$, the $G$-relation  $\Lambda^\star=i_\Gamma(\Lambda)$ (where $i_\Gamma$ is the inversion of $T^*G$) is an inverse and by the composition result  it also follows that  $uu^*$ is a pseudodifferential operator as soon as $u\in I(G,\Lambda)$. Hence, using known $C^*$-continuity results for pseudodifferential operators, which rely on the classical H\"ormander's trick to prove $L^2$-continuity,  we obtain $C^*$-continuity results for $G$-FIO subordinated to invertible $G$-relations. This also holds for locally invertible $G$-relations, that is, for $G$-relations onto which the source and target maps of $T^*G$ are only local diffeomophisms, also known as Lagrangian local bissections \cite{AS2006}.  

For the sake of clarity of this introduction, we  have ignored a technical point about the regularity of $G$-relations. More precisely, as sets, $G$-relations are true submanifolds, but all the  statements above are true for local $G$-relations, that is, those based on immersed submanifolds (ranges of immersions). As in the classical case, we can not avoid the introduction of immersed submanifolds since the convolution of two $G$-relations  is the image of some submanifold by a $C^\infty$ map of constant rank. Such images are not in general true submanifolds but are always images of some not necessarily injective immersion: we call them local submanifolds, following the (implicit) suggestion of \cite[Prop. C.3.3]{Horm-classics}.  Local submanifolds are countable union of true submanifolds of the same dimension and it is very convenient for our purposes to handle them in this way.

To finish, we treat the example of the groupoids arising in the case of manifolds with boundary and we compare the corresponding $G$-relations with the boundary canonical relations as defined by R. Melrose \cite{Melrose1981}.
Fourier integral operators on Lie groups are treated in \cite{NielStetk1974}. The present work recovers and significantly extends the results of \cite{NielStetk1974,Melrose1981}. The present paper  also uses  the notations and results about distributions on groupoids in \cite{LMV1}, a subject which is also treated in \cite{VY}.

The paper is organized as follows. In Section \ref{Sec:vocabulary} we introduce useful vocabulary and we  present classical facts in differential geometry in a  perhaps not completely classical way. Section \ref{sec:Lagrange-family} is devoted to families of Lagrangian (submanifolds, distributions). It contains in particular a new result about the comparison of families of Lagrangian submanifolds in the cotangent spaces of the fibers of a given submersion and Lagrangian submanifolds of the cotangent space of the total space of that  submersion. In section \ref{sec:Lag-subm-of-Gamma}, we study operations on Lagrangian submanifolds of the symplectic groupoid $T^*G$. This includes the study of their convolution, transposition and invertibility property. The notion of $G$-relation is also introduced and the relationship with equivariant families of Lagrangian is clarified. The  $G$-FIO are introduced in section \ref{G-FIO}. Similarly to $G$-relations, the parallel with equivariant families of FIO is fully analysed.
 Furthmermore, we extend the basic calculus of FIO to $G$-FIO, which includes a formula for the product of principal symbols. Section \ref{sec:b-FIO} is devoted to the comparison between the calculus we get in the case of the groupoid of the $b$-calculus and former constructions by R. Melrose \cite{Melrose1981}.

\bigskip \noindent {\bf Acknowledgments and credits.} We are grateful to Daniel Bennequin for the extremely stimulating mathematical discussion he offered to them. Also, we would like to thank Dominique Manchon for his constant encouragement. The first author thanks Claire Debord for the time and the help she has given to him at each milestone of this project.

\section{General definitions and preliminaries}\label{Sec:vocabulary}

\subsection{Local submanifolds, families of submanifolds}
 
We recall here some vocabulary of differential geometry. All manifolds are $C^\infty$, $\sigma$-compact,  with connected components of the same dimension.  All maps are $C^\infty$.  

\begin{defn}\label{defn:appendix-DG-0} Let $X$ be a manifold.
\begin{enumerate}
 \item  A local submanifold of $X$ of dimension $p$ is any subset consisting of a countable union of submanifolds of $X$ of dimension $p$. 
\item   A submanifold of $X$ contained in a local submanifold $Z$ of the same dimension is called a patch of $Z$. 
\item  A parametrization of a local submanifold $Z$  is a diffeomorphism of a manifold onto a patch of $Z$. 
\end{enumerate}
\end{defn}
A local submanifold is essentially the same thing as an immersed submanifold, the difference being that for a local submanifold, the immersion is not a priori given.

Indeed, let $Z$ be a local submanifold of $X$ and let  $(Z_j)_j$ be a countable family  of patches such that $Z=\cup_j Z_j$. Consider the manifold $Y=\sqcup_j Z_j$ and observe that the obvious immersion $f : Y\to X$ satisfies $f(Y)=Z$. Conversely, if $f : Y\to X$ is an immersion then   one can cover $Y$ by a countable family $(U_j)$ of open subsets   such that  $f: U_j\to f(U_j)$ is a diffeomorphism. Thus the image $Z=f(Y)=\cup f(U_j)$ is a local submanifold of dimension $\dim Y$. 

Actually, images of maps $f : Y\to X$ of constant rank are also local submanifolds. Indeed, by \cite[Proposition C.3.3]{Horm-classics}) there exist   countable families of local coordinate systems $(U_j)_j$ covering $Y$ and  $(V_j)_j$ covering $X$ such that $f(U_j)\subset V_j$ and $f$ is given in these coordinates by
 \begin{equation}
   f(y_1,\ldots,y_m) = (y_1,\ldots,y_p,0,\ldots,0) \quad \forall (y_1,\ldots,y_m)\in U_j.
 \end{equation}
Set $Y_j=\{ y\in U_j\ ;\ y_{p+1}=\cdots=y_m=0\}$ and consider the manifold $\widetilde{Y}=\sqcup_j Y_j$. The natural map 
\begin{equation}
   \widetilde{Y}\ni y \longmapsto f(y)
\end{equation}
is an immersion with range $Z$.

We recall that two submanifolds $Z_1,Z_2$ of $X$ have a clean intersection if $Z_1 \cap Z_2$  is a submanifold and at any point $z\in Z_1 \cap Z_2$,
\begin{equation}
 \label{eq:clean-tangent-cond}
 T_z (Z_1 \cap Z_2)= T_z Z_1 \cap T_z Z_2.
\end{equation}
The excess of the intersection is  the number
\(
 e = \mathrm{codim}( T_z Z_1 + T_z Z_2).
\)
The intersection is transversal if $e=0$ and the transversality of $Z_1,Z_2$ is denoted by $Z_1 \pitchfork Z_2$.

\begin{defn}\label{defn:appendix-DG-1}
 Let $Z_1,Z_2$ be two local submanifolds of $X$. The intersection $Z_1 \cap Z_2$ is clean (resp. transversal) if there exist covers $(Z_{1j})$ and $(Z_{2k})$ of $Z_1$ and $Z_2$ by countably many patches such that $Z_{1j}\cap Z_{2k}$ is clean with the same excess (resp. transversal) for all $j,k$ . 
\end{defn}
 
In the following definitions, we consider a surjective submersion $\pi : X \longrightarrow B$ between manifolds. The fiber of $\pi$ at the point $b$ is noted $X_b$.
\begin{defn}\label{defn:appendix-DG-2} \
\begin{enumerate}
 \item A submanifold $Z$ of $X$ is said transverse to $\pi$ if $\pi\vert_{Z} : Z\to B$ is a submersion. 
 \item A local submanifold $Z$ of $X$ is said transverse to $\pi$ if it can be covered by countably many submanifolds transverse to $\pi$. 
\end{enumerate}
\end{defn}

\begin{defn}\label{defn:appendix-DG-5} 
 A family $\cZ=(Z_b)_{b\in U}$ of subsets  $Z_b\subset X_b$, $U$ open in $B$, is  a $C^\infty$ family subordinated to $\pi$ of (resp. local) submanifolds if $Z=\cup_U Z_b$ is a (resp. local) submanifold of $X$ transverse to $\pi$.
 \end{defn}

\begin{defn}\label{defn:appendix-DG-4}
 Let $\cZ=(Z_b)_{b\in B}$ be a $C^\infty$ family of local submanifolds and set $Z=\cup_B Z_b$.
\begin{enumerate}
 \item  Patches and parametrizations of a family $\cZ$ refer to the same objects for $Z$. 
 \item A section of $\cZ$ is a $C^\infty$ locally defined section of $\pi : X\to B$  with values in a patch of $\cZ$.
\end{enumerate}
\end{defn}

\subsection{Phases, clean and non-degenerate phases}

A subset $\cU$ of $\RR^n\times \RR^N$ is conic if $(x,\theta)\in \cU$ implies $(x,t\theta)\in \cU$ for all $t>0$. A map $\chi : \cU\to \cV$ between conic subsets is homogeneous if $\chi(x,t\theta)=t\chi(x,\theta)$ for all $t>0$.

\begin{defn}\cite[p.86]{Horm-FIO1}\cite[21.1.8]{Horm-classics}
A cone bundle consists of a  surjective submersion $p:C\to X$ and  an action of $\RR_+^*$ on  $C$ which respects the fibers of $p$, such that: \\
For all $v\in C$, there exists a conic neighborhood $\cU$ of $v$ in $C$  and a homogeneous diffeomorphism $\chi:\cU\to\cV\subset \RR^n\times(\RR^N\setminus\{0\})$ onto a conic open subset  such that  the diagram 
\begin{equation}
\begin{tikzcd}
    \cU \arrow{rr}{\chi} \arrow{dr}{p}&   &   \cV
    \arrow{ld}{\pr{1}}   \\
     & U &
\end{tikzcd}
\end{equation}
commutes. 
The triple $(\cU,\cV,\chi)$  is then called a conic local trivialization of  the cone bundle around $v$.  When $X$ is a point,  $C$ is called a conic manifold. 
\end{defn}

\begin{exam} 
\begin{enumerate}
\item If $X$ is a manifold, $T^*X\setminus 0$ is a conic manifold. 

\item Let $\pi : Z \to X$ be a submersion onto $X$ and set $C= Z\times \RR^{k}\setminus 0$, $p=\pi\circ\pr{1}$. Then $C$, with the obvious $\RR_+$-action is a cone bundle over $X$. Indeed, Local trivializations $\kappa : p^{-1}(U)\overset{\simeq}{\to}U\times Y\times \RR^{k}\setminus 0$ provide conic local trivializations after composition by   
\[
  (x,y,\theta) \longmapsto (x,\vert\theta\vert.y,\theta) \in U\times(\RR^{n_Z-n_X+k}\setminus 0).
\]

\end{enumerate}
\end{exam}
\begin{defn}\cite[Def. 21.2.15]{Horm-classics},\cite[p. 154]{GS}.  Let $X$ be a manifold and $U\subset X$ an open subset. 
\begin{enumerate}
 \item A phase function over $U$  consists of a cone bundle $(p,C,U)$ and a $C^\infty$ homogeneous function 
\(\phi : C\to \RR\) without critical points.

\item Let   $\phi : C\to \RR$ be a phase function over $U$. Let us note $\phi'_{\mathrm{vert}} : C \to (\ker dp)^*$ the restriction of the differential of $\phi$ to the fibers of $p :C\to U$. We say that $\phi$ is clean if the set 
\begin{equation}
 C_\phi =\{ v\in C;\ \phi'_{\mathrm{vert}}(v)=0 \} = (\phi')^{-1}(\ker dp^\perp)
\end{equation}
is a submanifold of $C$ with tangent space given by the equation $d\phi'_{\mathrm{vert}}=0$. The excess of the clean phase $\phi$ is the number 
$e=\dim C_\phi - \dim X= \dim \ker dp - \mathrm{rk}(d\phi'_{\mathrm{vert}})$.

\item The phase function $\phi$ is non degenerate if $\phi'_{\mathrm{vert}}$ is a submersion (that is, clean and $e=0$). 
\end{enumerate}
\end{defn}
Using ${}^tdp^{-1} : (\ker dp)^\perp \to T^*X$, the \enquote{horizontal} part of $d\phi$ is then well defined on $C_\phi$ by  $\phi'_{\mathrm{hor}}(v) = {}^t(dp_v)^{-1}(\phi'(v))\in T^*_{p(v)}X$, that is 
\begin{equation}
  \phi'_{\mathrm{hor}}(v)(t) = \phi'(v)(u), \quad t\in p^*(TX)_v,\ dp(u)=t.
\end{equation}
We introduce the map
\begin{eqnarray}\label{eq:dp-graph-dphi}
T_\phi:C_\phi&\longrightarrow & T^*X   \\
 v&\longmapsto& (p(v),\phi'_{\mathrm{hor}}(v)\big)\nonumber
\end{eqnarray}
and we set 
\begin{equation}\label{eq:Lambda-phi}
 \Lambda_\phi=T_\phi(C_\phi) = \{ (p(v),\phi'_{\mathrm{hor}}(v))\ ;\  \phi'_{\mathrm{vert}}(v)=0 \}.
\end{equation}
If $\phi$ is clean, then for any $v\in C_\phi$, there exists an open conic neighborhood $V$ of $v$ into $C$ such that $T_\phi(V)$  is a $C^\infty$ conic Lagrangian submanifold of $T^*X\setminus 0$ and $T_\phi : C_\phi\cap V\longrightarrow T_\phi(V)$ is a  fibration with fibers of dimension $e$   and therefore $\Lambda_\phi$ is a conic Lagrangian  local submanifold of $T^*X\setminus 0$ (\cite{Horm-classics,GS}, see also Remark \ref{rem:clean-phase-and-symplectic-reduction} below).
If the fibers of $T_\phi$ are moreover connected and compact, then $\Lambda_\phi$ is a true submanifold  and $T_\phi : C_\phi \longrightarrow\Lambda_\phi $ is a  fibration.
On the other hand, if $\phi$ is non degenerate, we just gain that  $T_\phi$ is an immersion: $\Lambda_\phi$ is still a local submanifold  (usually called in the litterature   {\sl immersed submanifold}, self-intersections being not excluded). 

Conversely, any  conic Lagrangian submanifold $\Lambda$ of $T^*X\setminus 0$ can be locally parametrized by   non denegerate phase functions \cite{Horm-classics,GS}. This means that for any $(x,\xi)\in \Lambda$ there exist  an open conic neighborhood $W$ of $(x,\xi)$ into $T^*X$, an open conic  subset $V\subset X\times \RR^N\setminus 0$  and a non-degenerate phase function $\phi: V\to \RR$ with $\Lambda_\phi=\Lambda\cap W$.

\subsection{Lagrangian distributions on a manifold}
Unless otherwise stated, we use the definitions and notations of \cite{Horm-classics} for all the notions involved in the theory of Lagrangian distributions. 

Let $X$ be a $C^\infty$ manifold of dimension $n$, $E$ a complex vector bundle over $E$, $\Lambda$ a conic Lagrangian submanifold of $T^*X\setminus 0$ and $m\in\RR$. 
The set $I^m(X,\Lambda;E)$ consists of distributions belonging to $\cD'(X,E)$ which, modulo $C^\infty(X,E)$, are  locally finite sum of oscillatory integrals (\cite[Section 25.1]{Horm-classics}):
\begin{equation}
  u= \sum_{j\in J} (2\pi)^{-(n+2N_j)/4}\int e^{i\phi_j(x,\theta_j)} a_j(x,\theta_j)d\theta_j \mod C^\infty(X,E)
\end{equation}
where for all $j$,
\begin{itemize}
 \item $(x,\theta_j)\in \cV_j\subset U_j\times \RR^{N_j}$ with  $U_j$ a local coordinate patch of $X$ and $\cV_j$ an open conic subset;
 \item $\phi_j : \cV_j \to\RR$ is a non degenerate phase function providing a local parametrization of $\Lambda$;
 \item $a_j(x,\theta_j)\in S^{m+(n_X-2N_j)/4}(U_j\times \RR^{N_j},E)$ has support in the interior of a cone with compact base and included in $\cV_j$.
\end{itemize}
Such distributions are called Lagrangian distributions associated with $\Lambda$, with values in $E$. 
When $\Lambda$ is the conormal bundle of a submanifold, they are called conormal distributions. 

In the definition above, one can allow  conic Lagrangian local submanifolds of $T^*X\setminus 0$ and thus, the set $I^m(X,E)$  of all Lagrangian distributions with values in $E$ is a vector space. 

The principal symbol of an element in $I^m(X,\Lambda;E\otimes \Omega^{1/2}_X)$ can be defined as an element of  $S^{m+n/4}(\Lambda,I_\Lambda\otimes\hat{E})$ well defined modulo $S^{m+n/4-1}$. Here $I_\Lambda$ is the tensor product of the Maslov bundle with half densities over $\Lambda$ and  $\hat{E}$ is the pull back of $E$ onto $\Lambda$. The principal symbol map gives an isomorphism \cite[Theorem 25.1.9]{Horm-classics}
\begin{equation}
 \sigma : I^{[m]}(X,\Lambda;E\otimes \Omega^{1/2}_X) \longrightarrow S^{[m+n/4]}(\Lambda,I_\Lambda\otimes\hat{E}),
\end{equation}
with the conventions  $I^{[*]}=I^*/I^{*-1}$, $S^{[*]}=S^*/S^{*-1}$.
  
Let $X,Y,Z$ be $C^\infty$ manifolds and $\Lambda_1\subset  T^*X\setminus 0\times T^*Y \setminus 0$ and $\Lambda_2\subset  T^*Y\setminus 0\times T^*Z \setminus 0$ be   conic Lagrangian submanifolds closed in $T^*X\times T^*Y \setminus 0$  and $T^*Y\times T^*Z \setminus 0$ respectively. It is understood that the symplectic structures of $T^*X\times T^*Y$ and $T^*Y\times T^*Z$ are the product ones. Assume that the intersection of $\Lambda_1\times \Lambda_2$ with $T^*X\times N^*(\Delta_Y)\times T^*Z$ is clean with excess $e$, where $N^*(\Delta_Y)$ is the conormal space of the diagonal $\Delta_Y$ in $Y^2$. If $A_1\in I^{m_1}(X\times Y,\Lambda_1;\Omega^{1/2}_{X\times Y})$ and 
$A_2\in I^{m_2}(Y\times Z,\Lambda_2;\Omega^{1/2}_{Y\times Z})$ are properly supported, then \cite[Theorem 25.2.3]{Horm-classics}
\begin{equation}
 A= A_1\circ A_2 \in I^{m_1+m_2+e/2}(X\times Z,\Lambda,\Omega^{1/2}_{X\times Z}). 
\end{equation}
Here $A_1\circ A_2$ is defined through the Schwartz kernel theorem and $\Lambda$ is the conic Lagrangian local submanifold defined by the composition of $\Lambda_1$ and $\Lambda_2$:
\begin{equation}
 \Lambda=\Lambda_1\circ \Lambda_2 =\{(x,\xi,z,\zeta)\ ;\ \exists (y,\eta)\in T^*Y, (x,\xi,y,-\eta,y,\eta,z,\zeta)\in  \Lambda_1\times \Lambda_2\}.
\end{equation}
Under the same assumptions on $\Lambda_i$, $i=1,2$, there is thus a well defined product of principal symbols:
\begin{equation}
 S^{[m_1+(n_X+n_Y)/4]}(\Lambda_1,I_{\Lambda_1})\times S^{[m_2+(n_Y+n_Z)/4]}(\Lambda_2,I_{\Lambda_2})\overset{\circ}{\longrightarrow} S^{[m_1+m_2+e/2+(n_X+n_Z)/4]}(\Lambda,I_{\Lambda})
\end{equation}
which is defined abstractly by 
\begin{equation}
 a= a_1\circ a_2 = \sigma(\sigma^{-1}(a_1)\circ \sigma^{-1}(a_2))
\end{equation}
and computed concretely through the integral 
\begin{equation}
 a(\gamma) = \int_{C_\gamma} a_1\times a_2
\end{equation}
where $a_1,a_2,a$ are representants in $S^*$ of the given classes in $S^{[*]}$, $\gamma\in \Lambda_1\circ \Lambda_2$, the manifold $C_\gamma$ is the fiber of the projection map 
\begin{equation}
 p :\widetilde{\Lambda}:=\Lambda_1\times \Lambda_2\cap T^*X\times N^*(\Delta_Y) \times T^*Z\longrightarrow \Lambda_1\circ \Lambda_2
\end{equation}
and $a_1\times a_2$ is the density on $C_\gamma$ with values in $I_\Lambda$ resulting from the natural bundle homomorphism
\begin{equation}
\begin{tikzcd}
    I_{\Lambda_1}\otimes I_{\Lambda_2} \arrow{rr}{p} \arrow{dr}&   &  p^*(I_\Lambda)\otimes \Omega(\ker dp)\otimes\Omega^{-1/2}(T^*Y)
    \arrow{ld}   \\
     & \widetilde{\Lambda} &
\end{tikzcd}
\end{equation}
and from the trivialization of $\Omega^{-1/2}(T^*Y)$ using the canonical density of $T^*Y$ (see \cite[Theorems 21.6.6, 25.2.3]{Horm-classics})

\subsection{Lie groupoids, cotangent groupoids, associated foliations}
 
The following reminder about Lie groupoids is already included in \cite{LMV1}. We hope that this repetition will help the reading by improving the self-containness of the paper. 

A Lie groupoid is a manifold $G$ endowed with the additional following structures:
\begin{itemize}
 \item two surjective submersions $r,s: G\rightrightarrows G^{(0)}$ onto a manifold $G^{(0)}$ called the space of units. 
 \item An embedding $u : G^{(0)}\longrightarrow G$, which allows to consider  $G^{(0)}$ as a submanifold of $G$ and then such that
 \begin{equation}
   r(x) = x \quad ,\quad s(x)=x,\quad \text{ for all } x\in G^{(0)}.
 \end{equation}
 \item A $C^\infty$ map  
 \begin{equation}
   i : G\longrightarrow G, \ \ \gamma\longmapsto \gamma^{-1}
 \end{equation}
 called inversion and satisfying $s(\gamma^{-1})=r(\gamma)$ and $r(\gamma^{-1})=s(\gamma)$ for any $\gamma$.
 \item a $C^\infty$ map 
  \begin{equation}
  m : G^{(2)}=\{(\gamma_1,\gamma_2)\in G^2\ ;\ s(\gamma_1)=r(\gamma_2)\}\longrightarrow G, \ \ (\gamma_1,\gamma_2)\longmapsto \gamma_1\gamma_2
 \end{equation}	
 called the multiplication, satisfying the relations, whenever they make sense 
 \begin{align}
 & (\gamma_1\gamma_2)\gamma_3= \gamma_1(\gamma_2\gamma_3) & & r(\gamma)\gamma=\gamma & & \gamma s(\gamma)=\gamma \\
 &   \gamma\gamma^{-1} = r(\gamma) & & \gamma^{-1} \gamma= s(\gamma)\  & & \ r(\gamma_1\gamma_2)=r(\gamma_1),\ s(\gamma_1\gamma_2)=s(\gamma_2) .
 \end{align}
\end{itemize}
It follows from these axioms that $i$ is a diffeomorphism equal to its inverse, $m$ is a surjective submersion and  $\gamma^{-1}$ is the unique inverse of $\gamma$, for any $\gamma$, that is the only element of $G$ satisfying $ \gamma\gamma^{-1} = r(\gamma), \ \gamma^{-1} \gamma= s(\gamma)$. These assertions need a proof, and the unfamiliar reader is invited to consult for instance \cite{Mackenzie2005} and references therein. 

It is customary to write
\[
 G_x=s^{-1}(x),\quad G^x=r^{-1}(x),\ G_x^y= G_x\cap G^y, \ m_x=m\vert_{G^x\times G_x} : G^x\times G_x\longrightarrow G.
\]
$G_x$, $G^x$ are submanifolds and $G_x^x$ is a Lie group. The submersion $d :(\gamma_1,\gamma_2)\mapsto \gamma_1\gamma_2^{-1}$ defined on $G\underset{s}{\times}G$ is called the division map. 

Obviously, Lie groups, $C^\infty$ vector bundles, principal bundles, are Lie groupoids. Also, for any manifold $X$, the manifold $X\times X$ inherits a canonical structure of Lie groupoid with unit space $X$ and multiplication given by $(x,y).(y,z)=(x,z)$. The reader can find in \cite{Winkelnkemper1983,Pradines1986,Connes1994,NWX,MP,Debord2001,Monthubert1999,DLN,DL2009,VE2010I,VE2010II,DLR} further examples of groupoids as well as applications. 

The Lie algebroid $AG$ of a Lie groupoid $G$ is the vector bundle over $G^{(0)}$ defined by 
\begin{equation}
 AG = T_{G^{(0)}}G / TG^{(0)} = N G^{(0)}.
\end{equation}
Since $T_{G^{(0)}}G=\ker ds \oplus TG^{(0)}= \ker dr \oplus TG^{(0)}$, the bundle $AG$ can be replaced, up to canonical isomorphisms, by $\ker ds\vert_{G^{(0)}}$ or $\ker dr\vert_{G^{(0)}}$. We will often use the dual Lie algebroid $A^*G$, that is, the conormal space of $G^{(0)}$ in $G$.

 Differentiating a Lie groupoid $G$ produces a Lie groupoid $TG\rightrightarrows TG^{(0)}$ with the obvious structure maps $dr,ds,du,di,dm$ and the submanifold of composable pairs  coincides with the tangent space of the submanifold $G^{(2)}$, that is $(TG)^{(2)}=T(G^{(2)})\subset TG^2$. Another associated groupoid of particular interest in this work is the cotangent groupoid $\Gamma=T^*G\rightrightarrows A^*G$ discovered in \cite{CDW} and whose structure maps will be denoted $r_\Gamma,s_\Gamma,u_\Gamma, i_\Gamma,m_\Gamma$. All the structure maps of $T^*G$ and the choice of $A^*G$ as a unit space are dictated by the aim of defining the product in $T^*G$ by the natural formula:
 \begin{equation}
 (\gamma_1,\xi_1).(\gamma_2,\xi_2)= (\gamma_1.\gamma_2, \xi_1\oplus\xi_2)
\end{equation}
 with 
\[
 \xi_1\oplus\xi_2(dm(t_1,t_2))=\xi_1(t_1)+\xi_2(t_2), \quad \forall (t_1,t_2)\in T_{(\gamma_1,\gamma_2)}G^{(2)}.
\]
This makes sense if and only if $(\xi_1,\xi_2)\in T^*_{(\gamma_1,\gamma_2)}G^2$ vanishes on $\ker dm$, that is, denoting by
 \begin{equation}
  \label{eq:rest-lin-forms}
  \rho : T^*_{G^{(2)}}G^2\longrightarrow T^*G^{(2)}
 \end{equation}
the natural restriction map, if and only if 
\begin{equation}
 \label{eq:basic-composability-cond-cot-gpd}
 \rho(\xi_1,\xi_2) \in \ker (dm_{(\gamma_1,\gamma_2)})^\perp. 
\end{equation}
In that case we can set
\begin{equation}
 \label{eq:product-cot-gpd}
  \xi_1\oplus\xi_2 = ({}^tdm_{(\gamma_1,\gamma_2)})^{-1}(\rho(\xi_1,\xi_2)).
\end{equation}
This leads to the following formulas (see \cite{LMV1,CDW,Mackenzie2005,Pradines1988} for more details) for the remaining structure maps:
\begin{itemize}
 \item \( s_\Gamma(\gamma,\xi) = (s(\gamma),\overline{s}(\xi)) \) with  \(\overline{s}(\xi)=  {}^td(L_{\gamma})_{s(\gamma)}(\xi)\in A^*_{s(\gamma)}G=(T_{s(\gamma)}G/T_{s(\gamma)}G^{(0)})^* \);
 \item \(r_\Gamma(\gamma,\xi) = (r(\gamma),\overline{r}(\xi)) \) with  \(\overline{r}(\xi)= {}^td(R_{\gamma})_{r(\gamma)}(\xi) \in A^*_{r(\gamma)}G = (T_{r(\gamma)}G/T_{r(\gamma)}G^{(0)})^*\) ;
 \item \(i_\Gamma(\gamma,\xi) = (\gamma^{-1},-({}^tdi_{\gamma})^{-1}(\xi)). \)
\end{itemize}
Here, $R_\gamma : G_{r(\gamma)}\to G_{s(\gamma)}, \gamma_1\mapsto \gamma_1\gamma$ and $L_\gamma : G^{s(\gamma)}\to G^{r(\gamma)}, \gamma_2\mapsto \gamma\gamma_2$ denotes the (partially defined) right and left multiplication maps by $\gamma$ in $G$.
%

Taking into account the vector bundle structures 
\begin{equation}
 p : T^*G \longrightarrow G \quad ;\quad p^{2} : T^*G^{2} \longrightarrow G^{2} \quad ;\quad p^{(2)} : T^*G^{(2)} \longrightarrow G^{(2)} \quad ;\quad p^{(0)} : A^*G \longrightarrow G^{(0)},
\end{equation}
we observe that all structure maps of $T^*G$ are vector bundles homomorphisms and we get the following exact sequences:
\begin{equation}\label{diag:cot-gpd-mult-overview}
 \xymatrix{
  0 \ar[r] &  N^*G^{(2)} \ar[r]\ar^{p^2}[d]  & (T^*G)^{(2)}  \ar^{m_\Gamma}[r]\ar^{p^2}[d]  &  T^*G \ar[r]\ar^{p}[d]  & 0 \\
     & G^{(2)}  \ar^{=}[r] &  G^{(2)} \ar^{m}[r]&  G,  & 
}
\end{equation}
\begin{equation}\label{diag:cot-gpd-source-overview}
 \xymatrix{
  0 \ar[r] &  (\ker dr)^\perp \ar[r]\ar^{p}[d]  & T^*G  \ar^{s_\Gamma}[r]\ar^{p}[d]  &  A^*G \ar[r]\ar^{p^{(0)}}[d]  & 0 \\
     & G \ar^{=}[r] &  G \ar^{s}[r]&  G^{(0)} , & 
}
\end{equation}
and
\begin{equation}\label{diag:cot-gpd-target-overview}
 \xymatrix{
  0 \ar[r] &  (\ker ds)^\perp \ar[r]\ar^{p}[d]  & T^*G  \ar^{r_\Gamma}[r]\ar^{p}[d]  &  A^*G \ar[r]\ar^{p^{(0)}}[d]  & 0 \\
     & G \ar^{=}[r] &  G \ar^{r}[r]&  G^{(0)}.  & 
}
\end{equation}
It is useful to summarize the construction of the product $m_\Gamma$ in the commutative diagram below, in which the first two lines are exact. 
\begin{equation}\label{diag:cot-gpd-overview}
 \xymatrix{
  0 \ar[r] &  N^*G^{(2)} \ar[r] & T^*_{G^{(2)}} G^2 \ar^{\rho}[r] & T^*G^{(2)} \ar[r] & 0 \\
  0 \ar[r]  &  \ker m_\Gamma \ar[r] \ar^{=}[u] \ar^{(m,0)}[d] &  (T^*G)^{(2)} \ar^{\rho}[r]\ar^{\hookrightarrow}[u]\ar^{m_\Gamma}[d] &  (\ker dm)^\perp \ar[r]\ar^{\hookrightarrow}[u]\ar^{(m,({}^tdm)^{-1})}[d]  & 0 \\
  0 \ar[r]  &  G\times\{0\} \ar^{\hookrightarrow}[r]   &   T^*G  \ar^{=}[r]  &   T^*G \ar[r]   & 0.
}
\end{equation}
The map $(m,({}^tdm)^{-1}) : (\ker dm)^\perp \to T^*G$ will be noted $\widetilde{m_\Gamma}$ later on.

Let $G$ be a Lie groupoid $G$ and consider the equivalence relation on $G^{(0)}$
\begin{equation}
 x\sim_{G^{(0)}} y \quad  \text{ if }\quad    G^x_y\not=\emptyset.
\end{equation}
The equivalence class of $x\in G^{(0)}$, also called the orbit of $x$ (under the action of $G$ onto $G^{(0)}$) is denoted by $O_x$. We obviously have
\begin{equation}
 O_x = r(s^{-1}(x))=s(r^{-1}(x))\subset G^{(0)}.
\end{equation} It is true  that  the $O_x$ are all immersed submanifolds \cite[Theorem 1.5.11]{Mackenzie2005}, see also \cite{Pradines1986}, which define a singular Stefan foliation (see \cite[Section 1.8, p.51]{Mackenzie2005}). We call it the canonical foliation of $G^{(0)}$ and denote it by $\cF_{G^{(0)}}$.

The leaves of $\cF_{G^{(0)}}$ can be lifted to $G$ using $r$ and this gives rise to another Stefan foliation $\cF_G$ that we call the canonical foliation of $G$. Using $s$ instead of $r$ gives the same foliation. The leaves of $\cF_G$ are immersed submanifolds  and coincide with the equivalence classes of the equivalence relation on $G$ given by 
\begin{equation}
 \gamma_1\sim_G \gamma_2\quad  \text{ if }\quad    G_{r(\gamma_2)}^{s(\gamma_1)}\not=\emptyset.
\end{equation}
Finally, the projection maps $G^2\to G$, $G^{(2)}\to G$  are respectively denoted by $\pr{j}$, $\pr{(j)}$, $j=1,2$ and if 
   $E,F$ are  vector bundles over $G$, we will use the shorthand notation $E\boxtimes F$  to denote $\pr{(1)}^*(E)\otimes\pr{(2)}^*(F)\to G^{(2)}$.

\section{Families of Lagrangian submanifolds and of Lagrangian distributions}\label{sec:Lagrange-family}

\subsection{Families of Lagrangian submanifolds and submersions}

Let $\pi : M\to B$ be a submersion onto $B$, with fibers of dimension $n$ and base of dimension $q$. The inclusion $M_b\hookrightarrow M$ is denoted by $i_b$. We consider the vector bundle  $V^*M=(\ker d\pi)^*=\cup_{b\in B} T^*M_b$ over $M$ and we denote by $p$ both the projection maps $T^*M\to M$ and $V^*M\to M$. Similarly the natural submersions maps $T^*M\to B$ and $V^*M\to B$ are both denoted by $\sigma$, while the natural restriction map $T^*M\to V^*M$ is denoted by $\rho$. The fibers of $V^*M\to B$ are exactly the cotangent spaces $T^*M_b$, $b\in B$. We have a short exact sequence of vector bundles over $M$
\begin{equation}
 \label{eq:around-pi}
0\longrightarrow (\ker d\pi)^\perp \longrightarrow T^*M \overset{\rho}{\longrightarrow}V^*M\longrightarrow 0.
\end{equation}
We are interested in  $C^\infty$ families $(\Lambda_b)_{b\in B}$ of  (local, Lagrangian, conic) submanifolds subordinated to $\sigma$ in the sense of Definition \ref{defn:appendix-DG-5}.
By a slight abuse of vocabulary, we will say that they are subordinated to $\pi$. Similarly, we will say that $\Lambda\subset T^*M$ is transverse to $\pi$ if it is transverse to $\sigma=\pi\circ p : T^*M\to B$ in the sense of Definition \ref{defn:appendix-DG-5},  which is here obviously equivalent to the condition
\begin{equation}
\label{eq:transv-equiv-cond}
 T_xM_b + dp(T_{x,\xi}\Lambda) = T_xM \quad \forall b\in B,\ \forall x\in M_b;
\end{equation}
that is, equivalent to the transversality of the maps $i_b:M_b\to M$ and $p\vert_\Lambda : \Lambda\to M$ for any $b$. 
 The next theorem is a straight adaptation of Theorem 21.2.16 in \cite{Horm-classics}.
\begin{thm}\label{thm:family-Lagrangian-and-generating-function}
Let  $\cL=(\Lambda_b)_{b\in B}$ be a family subordinated to $\pi$ of conic lagrangian submanifolds and $L=\cup_{b\in B} \Lambda_b\subset V^*M\setminus 0$ the associated transversal submanifold. 
\begin{enumerate}
 \item For any $(m_0,\xi_0)\in\Lambda_{b_0}$,  there exists local trivializations of $\pi$  around  $m_0$ such that in the associated local coordinates $(x,b,\xi)$ of $V^*M$, the map 
\begin{equation}\label{eq:family-Lagrangian-and-generating-function-1}
  L\ni ( x,b,\xi) \longmapsto (b,\xi)
\end{equation}
is a local diffeomorphism. Such a local trivialization is called adapted to $\cL$ (or $L$). 
 \item In local trivializations adapted to $\cL$, there exists conic neighborhoods  $\cW$ of $(b_0,\xi_0)\in\RR^q\times (\RR^n\setminus 0)$ and $\cV$ of $(m_0,\xi_0)\in V^*M\setminus 0$ and a unique  $C^\infty$ function  $H:\cW\to\RR$  homogeneous of degree $1$   such that
\begin{equation}
 L\cap \cV =\{ (H'_\xi(b,\xi),b,\xi)\ ;\ (b,\xi)\in \cW\}.
\end{equation}
 In other words, the $C^\infty$ function
\(
 \phi(x,b,\xi) = \langle x,\xi\rangle - H(b,\xi)
\)
provides a family labelled by $b$ of non-degenerate phase functions $\phi(\cdot,b,\cdot)$ parametrizing $\Lambda_b$. 
\end{enumerate}
\end{thm}
Using the notions of sections, transversality and parametrizations introduced in Section \ref{Sec:vocabulary}, we see that the conclusions of the theorem hold for families  of conic lagrangian local submanifolds as well. We just need to change $L$ in \eqref{eq:family-Lagrangian-and-generating-function-1}  by a patch $L'$.  

Thus, families  of lagrangian local submanifolds are  parametrized by families of non-degenerate phases functions defined in open cones of $M\times \RR^n\setminus 0$.

\begin{proof} 

 Let $(y,z)$ be a local trivialisation around $m_0$.  Here, $z=(z_1,\ldots,z_q)$ gives local coordinates of $B$ at $b_0$ and for fixed $b$, $y=(y_1,\ldots,y_n)$ gives local coordinates of $M_b$.  Following the proof of \cite[Theorem 21.2.16]{Horm-classics}, we can perform a change of variables $x=x(y)$  so that, as submanifolds of $T^*M_b$, the space $\Lambda_{b_0}$ is transversal to the horizontal subspace $\xi=\xi_0$  at the point $(m_0,\xi_0)$ and then, so that the map $\Lambda_{b_0}\ni (x,b_0,\xi)\to \xi $ has a bijective differential at $(x_0,b_0,\xi_0)$. 
Moreover, by assumption, the map $L \ni (x,b,\xi)\mapsto b$ has a surjective differential at $(x_0, b_0,\xi_0)$. Therefore, the differential of  
$L \ni (x, b,\xi)\mapsto (b,\xi)$ is surjective   at $(x_0, b_0,\xi_0)$, hence bijective for $\dim L=n+q$.

We now turn to the second assertion which consists of routine computations (see for instance the end of the proof of \cite[Theorem 21.2.16]{Horm-classics}).  By 1., there exists a neighborhood $\cU= (U\times  W ) \times C$ of $(x_0,b_0,\xi_0)\in \RR^{n+q}\times(\RR^n\setminus 0)$ and a $C^\infty$ function $x(b,\xi)$ defined on $W\times C$ such that 
\[
 L\cap \cU = \{ (x(b,\xi),b,\xi)\ ;\ b\in W, \xi\in C\}.
\]
Since $x$ is necessarily homogeneous of degree $0$ in $\xi$, we can assume that $C$ is a cone. Since the canonical one form of $T^*M_b$ vanishes on $\Lambda_b$, we get 
$$
  \sum_j \xi_j d_\xi x_j(b,\xi)=0. 
$$
In other words, the linear form $u\mapsto \langle x'_\xi(b,\xi).u , \xi\rangle $ vanishes.  It follows  that 
\begin{equation}
x(b,\xi) = H'_\xi(b,\xi), \qquad \text{ with }  H(b,\xi) = \langle x(b,\xi) , \xi \rangle,
\end{equation}
and that, by Euler formula,  this function $H$ is unique among $C^\infty$ functions $K(b,\xi)$ homogeneous of degree $1$ in $\xi$ and satisfying $K'_\xi=x$. 
Finally, it is clear that for fixed $b$,  the function
$\phi(x,b,\xi):=\langle x,\,\xi\rangle-H(b,\xi)$
is a non-degenerate phase function parametrizing  $\Lambda_b$. 

\end{proof}

\begin{thm}\label{thm:gluing-family-of-Lagrangian}
Let  $(\Lambda_b)_{b\in B}$  be a family    subordinated to $\pi$ of conic lagrangian local submanifolds. There exists a unique  conic lagrangian local submanifold $\Lambda\subset T^*M$ transverse to $\pi$ such that 
\begin{equation}\label{eq:gluing-condition}
   i_b^*\Lambda = \Lambda_b,\quad b\in B.
\end{equation}
One says that $\Lambda$ is the gluing of the family $(\Lambda_b)_{b\in B}$. 
\end{thm}
\begin{proof} We first assume that $(\Lambda_b)_{b\in B}$ is a family of submanifolds. Assume that $\Lambda$ is a conic lagrangian submanifold of $T^*M$ satisfying \eqref{eq:gluing-condition}. Let  $\kappa : \cU\to U\times W$, $\kappa(m)=(x,b)$, be an adapted local trivialisation  and $H$ the corresponding function  constructed in Theorem \ref{thm:family-Lagrangian-and-generating-function}. By assumption, we have in these coordinates
\begin{equation}
 \kappa_*(\Lambda\cap T^*\cU ) \subset \{ (H'_\xi(b,\xi),b,\xi,\tau);\ \xi\in \cC,\ (b,\tau)\in T^*W\}.
\end{equation}
The projection $(x,b,\xi,\tau)\to (b,\xi)$ restricted to $\kappa_*(\Lambda\cap T^*\cU )$ is still a local diffeomorphism since $\dim \Lambda=n+q$, thus $\tau$ is  a $C^\infty$ function of $(b,\xi) $.  Since $\Lambda$ is conic and Lagrangian, the fundamental one form of $T^*M$ vanishes identically on $\Lambda$, which yields
\begin{eqnarray*}
 0&=& \sum_j \xi_j d(H'_{\xi_j})(b,\xi) + \sum_l\tau_ldb_l \\
&=& \sum_{i,j} \xi_j H''_{\xi_i\xi_j}(b,\xi) d\xi_i +\sum_{l,j}  \xi_j H''_{b_l\xi_j}(b,\xi) db_l +\sum_l\tau_ldb_l \\
 &=& \sum_{l,j}  \xi_j H''_{b_l\xi_j}(b,\xi) db_l +\sum_l\tau_ldb_l,\quad  \text{ since } H'_{\xi_i} \text{ is homogeneous of degree 0 in  } \xi,\\
 & = & \sum_{l,j}   H'_{b_l}(b,\xi) db_l +\sum_l\tau_ldb_l,\quad  \text{ since } H'_{b_l} \text{ is homogeneous of degree 1 in  } \xi,
\end{eqnarray*}
which proves that $\tau(b,\xi) =-H'_b(b,\xi)$ and thus 
\begin{equation}\label{eq:gluing-family-of-Lagrangian-1}
 \kappa_*(\Lambda\cap T^*\cU ) = \{ (H'_\xi(b,\xi),b,\xi, -H'_b(b,\xi));\ \xi\in \cC,\ b\in W\}\subset  (T^*U\times T^*W)\setminus 0.
\end{equation}
This proves the unicity and the transversality of $\Lambda$ with respect to $\pi$ as well. It also proves the existence in open subsets of the form $T^*\cU$, $\cU$ being the domain of an adapted local trivialisation. Remark for future reference that given $(m,\xi)\in\Lambda_b$, there is a unique $(m,\zeta)\in \Lambda$ such that $\rho(m, \zeta)=(m,\xi)$.

The existence follows from the local existence and the unicity. Indeed, let $(\kappa_j,\cU_j) $, $j=1,2$, be  two adapted local trivialisations such that $\cU_1\cap\cU_2\not=\emptyset$ and $\Lambda_1,\Lambda_2$ the submanifolds of $T^*\cU_1$ and $T^*\cU_2$ defined by \eqref{eq:gluing-family-of-Lagrangian-1}.  The previous argument of unicity proves that over $T^*\cU_1\cap\cU_2$, we have $\Lambda_1=\Lambda_2$. This allows to define a solution $\Lambda$ globally on $T^*M$ using a cover by adapted trivialisations. 
 
Now, let us consider the general case. Choose a countable cover of $\cL=(\Lambda_b)_b$ by families $\cL_j=(\Lambda_{bj})_{b\in U_j}$, $j\in J$, of conic Lagrangian submanifolds. By the first part of the proof there exists for any $j$ a unique  $\Lambda_j\subset T^*M$   gluing $\cL_j$. Then $\Lambda=\cup_J \Lambda_j$ is a gluing of  $\cL$ and this proves the existence.  
If $\Lambda'_1$ is a patch contained in another solution $\Lambda'$, then $\rho(\Lambda'_1)$ is contained in $\rho(\Lambda')=L=\cup_B \Lambda_b$. 
For any $j\in J$, the set $L_j=\cup_b\Lambda_{bj}$ is a patch of $L$ and by the remark made just after the proof of the unicity in the submanifold case, we get that $\rho^{-1}(L_j)\cap \Lambda'_1$ is contained in the unique conic Lagrangian submanifold $\Lambda_j$ gluing $\cL_j$. Therefore, $\Lambda'_1\subset \cup_J\Lambda_j=\Lambda$ and the unicity follows directly. 
\end{proof}

Conversely, we have

\begin{thm}\label{thm:Lagrangian-slicing}
Let $\Lambda\subset T^*M\setminus 0$ be a conic lagrangian submanifold transverse to $\pi$. Then
\begin{enumerate}
\item $\Lambda\cap (\ker d\pi)^\perp=\emptyset$. 
\item  $(i_b^*(\Lambda))_{b\in B}$ is a family of conic lagrangian local submanifolds of $T^*M_b\setminus 0$. In other words, $\rho(\Lambda)$ is a local submanifold of $V^*M$ transverse to $\pi$ and for any $b$, the fiber
\begin{equation}
\Lambda_b= i_b^*(\Lambda)=\rho(\Lambda)\cap T^*M_b
\end{equation}
is a conic Lagrangian local submanifold of $T^*M_b\setminus 0$.
\end{enumerate}
\end{thm}
Clearly, the statement generalizes to the local case. 
\begin{proof} 
\begin{enumerate}

\item On one hand, by dualizing the transversality condition \eqref{eq:transv-equiv-cond}, we get 
\begin{equation}\label{eq:dualizing-transv-Lambda}
(\ker d\pi)^\perp\cap (dp(T\Lambda))^\perp =M\times \{0\} \subset T^*M.
\end{equation}
On the other hand, the inclusion
\begin{equation}\label{eq:inclusion-lambda-dpTLambda-perp}
\Lambda  \subset (dp(T\Lambda))^\perp
\end{equation}
holds.  Indeed, by conicity of $\Lambda$, any $(x,\xi)$ in $\Lambda$  corresponds canonically to a vertical vector in $T_{(x,\xi)}\Lambda$,  denoted by $v(\xi)$. Since $\Lambda$ is Lagrangian, we have with these notations
\[\xi(dp(Z))= \omega(v(\xi),Z)= 0,\quad \forall Z\in T_{(x,\xi)}\Lambda,\] 
where $\omega$ denotes the symplectic form of $T^*M$.  Therefore $\Lambda\cap (\ker d\pi)^\perp\subset \{0\}$ and since by assumption $\Lambda\subset T^*M\setminus 0$, the first assertion is proved. 

 \item As observed in  \cite[Chap. 4, Par. 4]{GS}, the transversality assumption \eqref{eq:transv-equiv-cond} is actually equivalent 
to the transversality of the intersection of the canonical relation 
\[
 \Lambda(i_b) = \{ (m,-\xi,m,\zeta)\in T^*M_b\times T^*M\ ; \ m\in M_b, \zeta\vert_{T_mM_b}=\xi\}
\]
with $\Lambda$, viewed as a canonical relation from $T^*M$ to a point. Therefore, The Hormander's product of canonical relations applies \cite[Theorem 21.2.14]{Horm-classics}, that is, the map
\[
 \rho_b : \Lambda\cap T_{M_b}^*M \longrightarrow T^*M_b\setminus 0\ ,\ (m,\zeta)\longmapsto (m,\zeta\vert_{T_mM_b})
\]
is an immersion with range $\rho_b(\Lambda)=\Lambda_b=i^*_b(\Lambda)$ a conic Lagrangian local submanifold of $T^*M_b\setminus 0$, for any $b$.
From now on, let $(m_0,\zeta_0)\in \Lambda$, $b_0=\pi(m_0)$, $(m_0,\xi_0)=\rho(m_0,\zeta_0)$ and choose a local trivialization $\kappa(m)=(x,b)$  of $\pi$  around $m_0$. After applying if necessary a diffeomorphism in the $x$ variables independent of $b$, we can assume that $\kappa$ is such that in a neighborhood of $(m_0,\xi_0)$ in $T^*M_{b_0}$, the projection $\Lambda_{b_0}\ni (x,\xi)\to \xi$ has a bijective differential. Moreover, by assumption, the map  $(x,b,\zeta)\to b$ has a surjective differential everywhere. It follows that the map 
\begin{equation}\label{eq:basic-property-transerve-lagrangian}
 \Lambda\ni (x,b,\xi,\tau) \longmapsto (b,\xi)\in\RR^q\times \RR^n
\end{equation}
has a surjective differential, therefore bijective for dimensional reason. In particular, this proves that the map $\rho : \Lambda\to V^*M$ is an immersion, thus $\rho(\Lambda)$ is a local submanifold. It is also obvious  from the same argument that $\rho(\Lambda)$ is transverse to $\pi$, which proves that $(i^*(\Lambda_b))_{b\in B}$ is a $C^\infty$ family of conic Lagrangian local submanifolds. 
\end{enumerate}
\end{proof}

\subsection{Families of Lagrangian distributions and submersions}

\begin{defn}\label{defn:family-FIO-submersion}
Let $\pi : M\longrightarrow B$ be a $C^\infty$ submersion of a manifold $M$ of dimension $n_M$ onto a manifold $B$ of dimension $n_B$.  A $C^{\infty}$ family of Lagrangian distributions of order $m$ relative to $\pi$ is a family $u_b\in I^m(\pi^{-1}(b),\Lambda_b,\Omega^{1/2}_\pi)$, $b\in B$, such that  $(\Lambda_b)_{b\in B}$ is a $C^\infty$ family and in any local trivialization $\kappa : \cU\to U\times W$  of $\pi$,
we have
\[
 \kappa_*(u|_\cU) = \int e^{i\phi(x,b,\theta)}a(x,b,\theta)d\theta,
\]
with $a\in S^{m+(n_M-n_B-2N)/4}(U\times W\times \RR^N)$ and $(x,b,\theta)\mapsto \phi(x,b,\theta)$ is $C^\infty$ and a non-degenerate phase function in $(x,\theta)$ which parametrizes locally $\Lambda_b$, for all $b$. 
\end{defn}

\begin{prop}\label{prop:global-distri-to-family-submersion-case}
Let  $B\ni b \mapsto u_b\in I^m(\pi^{-1}(b),\Lambda_b,\Omega^{1/2}_\pi)$ be a $C^{\infty}$ family. The formula
\begin{equation}
    \langle  \widetilde{u},f\rangle = \int_B \langle u_b,f\rangle ,\ f\in C^\infty_c(M,\Omega^{1/2}_\pi\otimes \pi^*(\Omega_B))
\end{equation}
defines a Lagrangian distribution
\begin{equation}
 \widetilde{u} \in I^{m-n_B/4}(M,\Lambda)
\end{equation}
where $\Lambda$ is the gluing of the family $(\Lambda_b)_b$. The map $(u_b)_{b\in B}\mapsto \widetilde{u}$ is bijective.

\end{prop}

\begin{proof}
 Let $u_b\in I^m(\pi^{-1}(b),\Lambda_b,\Omega^{1/2}_\pi)$ be a $C^{\infty}$ family. In sufficently small local trivializations, we have
\begin{equation}\label{eq:local-form-family-LD}
   u_b(x) = \int e^{i\phi(b,x,\theta)}a(b,x,\theta)d\theta 
\end{equation}
for some  $C^{\infty}$ family of non-degenerate phases functions $(\phi(b,\cdot,\cdot))_b$ parametrizing the family $(\Lambda_b)_b$ and some symbol $a\in S^{m+(n_M-n_B-2N)/4}(U\times W\times \RR^N)$.

Since $(x,\theta)\mapsto\phi(b,x,\theta)$ is non-degenerate phase function for any $b$, the function $(b,x,\theta)\mapsto \phi(b,x,\theta)$ is actually a non-degenerate phase function. We have necessarily $\Lambda_\phi=\Lambda$ locally, because on one hand $i_b^*(\Lambda_\phi)=\Lambda_b$  for any $b$ and on the other hand, $\Lambda$  is the unique lagrangian satisfying this condition. It follows that $\widetilde{u}$ is given locally by the oscillatory integral :
\[
 \widetilde{u}(b,x)=\int e^{i\phi(b,x,\theta)}a(b,x,\theta)d\theta,
\]
which proves that $\widetilde{u}\in I^{m-n_B/4}(M,\Lambda)$. Conversely, if $v\in I^{m-n_B/4}(M,\Lambda)$ then locally 
\begin{equation}\label{eq:local-form-product-LD}
   v(x) = \int e^{i\phi(b,x,\theta)}a(b,x,\theta)d\theta 
\end{equation}
for some non-degenerate phase function $\phi$ parametrizing $\Lambda$ and some symbol $S^{m-n_B/4+(n_M -2N)/4}(U\times W\times \RR^N)$. Since $\Lambda$ is transverse to $\pi$, the restriction $v_b$ of $v$ to $M_b$ is allowed and given by the $C^\infty$ family $b\mapsto v_b(x) = \int e^{i\phi(b,x,\theta)}a(b,x,\theta)d\theta$ where $\phi$ is regarded as a non degenerate phase function in $(x,\theta)$ for fixed $b$. This proves that $u\mapsto \widetilde{u}$ is bijective.
\end{proof}

\section{Lagrangian submanifolds of the cotangent groupoid}\label{sec:Lag-subm-of-Gamma}
%

The  cotangent    groupoid $T^*G \rightrightarrows A^*G$  \cite{CDW,Mackenzie2005,Pradines1988}  plays a basic role in the convolution of distributions on $G$ \cite{LMV1}. It is thus natural to investigate the behavior of Lagrangian submanifolds of  $T^*G$ under convolution. 

\subsection{Generalities}
We begin by classical facts in symplectic geometry. 
\begin{prop}\label{prop:push-forward-lagrangian} 
Let $(S,\omega_S),(T,\omega_T)$ be symplectic manifolds, $H$ a submanifold of $S$ and $\mu : H\to T$ a surjective submersion  such that
\begin{equation}\label{eq:push-forward-lagrangian-1}
  \mu^*(\omega_T) = \omega_S\vert_H. 
\end{equation}
\begin{enumerate}
\item The following assertions are equivalent
\begin{itemize} 
\item[(a)] $H$ is coisotropic;
\item[(b)] $ (\ker d\mu)^{\perp_{\omega_S}} = TH$;
\item[(c)] the graph  $\Gr{\mu}=\{(x,\mu(x))\ ;\ x\in H\}$ is a Lagrangian submanifold of $S\times (-T)$.
\end{itemize}
\item Assume that the previous assertions are true. If $\widetilde{\Lambda}$ is a  Lagrangian local submanifold  of $S$ in clean intersection with  $H $  then 
\begin{equation}
 \label{eq:clean-compo-lagr-gpd-style-rank}
   \Lambda := \mu(\widetilde{\Lambda}\cap H)
\end{equation}
is  a  local Lagrangian submanifold of $T$.  If moreover $\widetilde{\Lambda}$ is a submanifold and the map $\mu : \widetilde{\Lambda}\cap H\to \Lambda$ has compact and connected fibers, then $\Lambda$ is a submanifold. 
\end{enumerate}
\end{prop}
\begin{proof}
\begin{enumerate}
 \item The condition \eqref{eq:push-forward-lagrangian-1} implies that for any $x$, $\ker d\mu_x\subset (T_xH)^{\perp_{\omega_S}}$. Let us assume that $H$ is coisotropic, that is, that $(T_xH)^{\perp_{\omega_S}}\subset T_xH$ for all $x$. Let $u\in (T_xH)^{\perp_{\omega_S}}$. Then by assumption
 \[  
   \omega_T(d\mu(u),d\mu(v))= \omega_S(u,v)= 0 \text{ for all } v\in T_xH,
 \]
 which by surjectivity of $d\mu$ proves that $u\in \ker d\mu_x$. This gives (a)$\Rightarrow$(b) and the converse implication is trivial. 
 
Let $(s,t)\in (T\Gr{\mu})^{\perp_{\omega}}$ and choose $u\in T_xH$ with $d\mu(u)=t$. Then  $\omega_T(d\mu(u),d\mu(s'))=\omega_S(s,s')$ 
for all $( s',d\mu(s'))\in T\Gr{\mu}$. Using  \eqref{eq:push-forward-lagrangian-1}, this gives $u-s\in (TH)^{\perp_\omega}$ and 
assuming (b) this gives $t=d\mu(s)$, which proves  that $\Gr{\mu}$ is coisotropic and thus (c) 
since  \eqref{eq:push-forward-lagrangian-1} is obviously equivalent to the isotropy of  $\Gr{\mu}$.

For (c)$\Rightarrow$(a) we apply the following elementary lemma
\begin{lem}
Let $\lambda$ be a coisotropic linear subspace in a product of symplectic vector spaces $S_1\times S_2$. Then $\pr{j}(\lambda)$ is a  coisotropic subspace of $S_j$, $j=1,2$.  
\end{lem}
  
\item Using a decomposition of $\widetilde{\Lambda}$ into patches, it is sufficient to assume that $\widetilde{\Lambda}$ is a submanifold. Now, the result follows from a symplectic reduction procedure: see \cite[Proposition 21.2.13, Theorem 21.2.14]{Horm-classics} or \cite[page 12]{Weinstein1979}. We outline the proof.
 
 By assumption $\widetilde{\Lambda}\cap H$ is a $C^\infty$ submanifold and at any point $x\in \widetilde{\Lambda}\cap H$ 
\begin{equation}
 T_{x}(\widetilde{\Lambda}\cap H)= T_{x}\widetilde{\Lambda}\cap T_{x}H. 
\end{equation}
Since $\ker d\mu_x\subset (\ker d\mu_x)^{\perp_{\omega}} = T_xH$, the symplectic reduction  (\cite[proposition 21.2.13]{Horm-classics}) applied to $\lambda=T_{x}\widetilde{\Lambda}$ asserts that 
\[
 \lambda'= \left(T_{x} \widetilde{\Lambda}\cap H \right) / \left(T_{x}(\widetilde{\Lambda})\cap\ker d\mu_x\right)
\]
is a Lagrangian subspace of the symplectic vector space $S'=T_{x}H / \ker d\mu_{x}\simeq T_{\mu(x)}(T)$. Therefore $\rank\ d\mu_x=\dim T/2$ is independent of $x$ and the image $\Lambda=\mu(\widetilde{\Lambda}\cap H)$ is a local submanifold  of $T$ of dimension $\dim T/2$. Assumption \eqref{eq:push-forward-lagrangian-1} implies that $d\mu_x(T_x\widetilde{\Lambda}\cap H)$ is Lagrangian, therefore $\Lambda$ is a Lagrangian local submanifold. 
If the fibers of $\mu\vert_{\widetilde{\Lambda}\cap H}$ are moreover compact and connected,  it follows by standard arguments of differential geometry that $\Lambda$ is actually a  submanifold of $T$.
\end{enumerate}
\end{proof}
We now give a generic example in which Proposition \ref{prop:push-forward-lagrangian} applies. This example also shows how, and when, clean phase functions arise in the task of parametrizing Lagrangian submanifolds.  
\begin{prop}\label{prop:model-case-of-convol-phases}
Let $X, Y$ be manifolds, $Z\subset X$ a submanifold and $f :Z\to Y$ a submersion. Set $H=(\ker df)^\perp\subset T^*X$ and 
\[
  \mu   : H\ni (x,\xi)\longmapsto (f(x),{}^tdf^{-1}_x(\xi))\in T^*Y.
\]
The following assertions hold.
\begin{enumerate}
\item  $\Gr{\mu}$ is a Lagrangian submanifold of $T^*X\times (-T^*Y)$.
\item Let  $\widetilde{\Lambda}$ be a conic Lagrangian local submanifold of $T^*X\setminus 0$  intersecting cleanly   $H$ with excess $e$ and such that $\widetilde{\Lambda}\cap N^*Z=\emptyset$. Let   
\(
 (x,\xi)\in \widetilde{\Lambda}\cap H
\)
and  
\[
 \widetilde{\phi}  : U \times\RR^{N}\longrightarrow \RR \quad ;\    U  \text{ open subset of } X,
\]
 be a non-degenerate phase function parametrizing $\widetilde{\Lambda}$ around $(x,\xi)$. Setting $V=U\cap Z$,
the restriction  $\phi$ of $\widetilde{\phi}$ to $V\times \RR^N$
is a phase function on the cone bundle  
\begin{equation}
 \label{eq:clean-compo-phases-gpd-extended-covariable}
 f\circ \pr{1} : C= V\times( \RR^{N}\setminus 0) \longrightarrow f(V)\subset Y.
\end{equation}
This phase function is clean with excess $e$ and  parametrizes  $\Lambda=\mu(\widetilde{\Lambda}\cap H)$ around 
$\mu(x,\xi)=(f(x),{}^t(df^{-1}(\xi)))$.  
\end{enumerate}   
\end{prop}

\begin{proof}
\begin{enumerate}
\item This is immediately checked using local coordinates $(x',x'',x''')$ for $X$ such that $Z$ is given by $x'''=0$, $x'$ gives local coordinates for $Y$ and $f(x',x'')=x'$. Then one has $H=\{(x', x'', 0,\xi', 0,\xi''')\}$ so that $H$ is coisotropic and one can apply the first part of Proposition \ref{prop:push-forward-lagrangian}.

\item All the assertions being local, we may assume that  $\widetilde{\Lambda}$ is $C^\infty$.
 It is obvious that $\phi$ is $C^\infty$ and   homogeneous in the fibers of the given conic manifold.
Assume that  $d\phi$  vanishes at a point $(x,\theta)\in V\times \RR^N$. This means that $\widetilde{\phi}$ satisfies
\[
 \widetilde{\phi}'_x(x,\theta)(t)=0 \ \forall t\in T_xZ, \quad \widetilde{\phi}'_\theta(x,\theta)=0
\]
This implies that $(x,\widetilde{\phi}'_x(x,\theta))\in \widetilde{\Lambda}\cap N^*Z$, which contradicts the assumptions. Thus $\phi$ is a phase function. 

To precise $C_\phi$, we note $y\in Y$ the space coordinate of $\phi$ and 
\[
   \omega = (z,\theta), \quad \text{ where }  z\in f^{-1}(y)\subset Z
\]
the parameters. Then 
\begin{align}\label{proof:compo-phases-1}
  C_\phi  = \{ (x,\theta) \ ;\ \phi'_\omega(x,\theta)=0 \} 
\end{align}
Observe that $ \phi'_\omega = (\widetilde{\phi}'_z,\widetilde{\phi}'_{\theta})= (\phi'_z,\phi'_{\theta})$ thus 
\begin{align}
  C_\phi  = \{ (x,\theta)\in C_{\widetilde{\phi}} \ ;\ \widetilde{\phi}'_z(x,\theta)=0 \}.
\end{align}
Observe that $\widetilde{\phi}'_z(x,\theta)=0$ means exactly that $\widetilde{\phi}'_x(x,\theta)\in (\ker df)^\perp$, therefore
\begin{equation}
 (x,\theta)\in C_\phi \Leftrightarrow (x,\widetilde{\phi}'_x(x,\theta)) \in \widetilde{\Lambda}\cap H.
\end{equation}
It follows the local diffeomorphism $T_{\widetilde{\phi}} : C_{\widetilde{\phi}} \to \widetilde{\Lambda}$ maps $C_\phi$ onto $\widetilde{\Lambda}\cap H$:
\begin{equation}
 \label{proof:compo-phases-3}
   C_\phi \ni (x,\theta)\overset{T_{\widetilde{\phi}}}{\longmapsto} (x,\widetilde{\phi}'_x(x,\theta))\in \widetilde{\Lambda}\cap H.
\end{equation}
This proves that $C_\phi$ is a $C^\infty$ submanifold of $ C_{\widetilde{\phi}}$ since by assumption $\widetilde{\Lambda}\cap H$ is a $C^\infty$ submanifold. Recall that  $C_\phi$ is given by the equations
\begin{equation}\label{eq:equations-for-C-phi}
(x,\theta)\in V\times \RR^N,\quad \widetilde{\phi}'_z(x,\theta)=0,\quad \widetilde{\phi}'_{\theta}(x,\theta) =0. 
\end{equation}
The first one means that $(x,\widetilde{\phi}'_x(x,\theta))=T_{\widetilde{\phi}}(x,\theta)\in H$ and the second one that  $(x,\widetilde{\phi}'_x(x,\theta))=T_{\widetilde{\phi}}(x,\theta)\in\widetilde{\Lambda}$. Recall that $H$ is given by the equation 
\begin{equation}
  \rho(x,\xi)=(x,0),
\end{equation}
where $\rho : T_Z^*X\to (\ker df)^*, (x,\xi)\mapsto (x,\xi_z)$ is the submersion given by the restriction of linear forms to $\ker df$. Since by assumption we have
\begin{equation}
   T(\widetilde{\Lambda}\cap H ) = T \widetilde{\Lambda}\cap TH
\end{equation}
and since $T \widetilde{\Lambda}$ and $TH$ are given respectively by  the equations $d\phi'_{\theta}=0$ and $d\rho = 0$, it follows that 
$(t,\zeta)\in TZ\times \RR^N$ belongs to $TC_\phi$ if and only if $ dT_{\widetilde{\phi}}(t,\zeta)\in T(\widetilde{\Lambda}\cap H)$, that is if and only if 
\begin{equation}
d\phi'_{\theta}(t,\zeta)= 0 \text{ and } d\rho dT_{\widetilde{\phi}}(t,\zeta) = 0
\end{equation}
which, taking into account the definition of $\rho$ and its linearity in the fibers, is equivalent to 
\begin{equation}
 d\phi'_{\theta}(t,\zeta)= 0 \text{ and }  d\phi'_z(t,\zeta) = 0
\end{equation}
that is to $d\phi'_{\omega}(t,\zeta)= 0$ 
and this proves that $\phi$ is a clean phase function. 

Remember that $\phi$ is a phase function on the cone bundle $V\times \RR^{N}\overset{  f \circ\pr{1} }{\to}Y$, that is, the space variable is $y\in Y$ and the parameter variable is $\omega=(z,\theta)$ with $z\in f^{-1}(y)$. The differential $\phi'_h(y,\omega)\in T^*_y Y$ of $\phi$ in the \enquote{horizontal direction $y$} is well defined if and only if the vertical differential $\phi'_{\omega}$   vanishes and then
\begin{equation}
 \phi'_h(y,\omega)(v) = d_z\phi(z,\theta)(u) \quad \forall u\in T_{z}Z\text{ such that } df(u)=v.
\end{equation}
 Since $ d\phi_{(z,\theta)}(u) =  d\widetilde{\phi}_{(z,\theta)}(u)$, it follows that, around $\mu(x,\xi)$
\begin{align}
 \Lambda_\phi & = \{ (y,\phi'_h(y,\omega)) \ ; \  (y,\omega)=(z,\theta)\in C_\phi\} \nonumber \\
              &=  \{ (f(z), {}^t df^{-1}(\phi'_{z}(z, \theta)) \ ; \  (z,\theta)\in C_\phi\} \\
 &= \Lambda.\nonumber 
\end{align}
 We have $\dim H= n_X+ n_Y$ and $\dim C_\phi=\dim\widetilde{\Lambda}\cap H$.  Then
\begin{eqnarray}
 e&=& (2n_X-\dim\widetilde{\Lambda})+ (2n_X-\dim H) -(2n_X-\dim \widetilde{\Lambda}\cap H) \nonumber  \\
  &=& n_X +n_X-n_Y -2n_X +\dim \widetilde{\Lambda}\cap H = \dim\widetilde{\Lambda}\cap H-\dim\Lambda\\
  &=& \dim C_\phi - \dim\Lambda \nonumber 
\end{eqnarray}
and the latter is by definition the  excess $e$ of $\phi$.
\end{enumerate}
\end{proof}
\begin{remark}\label{rem:clean-phase-and-symplectic-reduction}
 Let  $\phi$ be phase function over $Y$, defined on the total space of a given cone bundle $p :C\to Y$.   Then $\widetilde{\Lambda}=\mathrm{Gr}(\phi')$ is a Lagrangian submanifold of $T^*C$ and in the notations of the previous proposition with $Z=X=C$, $f=p$, we get 
\[
 \widetilde{\Lambda} \cap H \text{ is clean if and only if } \phi \text{ is a clean phase function }
\]
and since $\mu(v,\xi) = (p(v), \xi_{\mathrm{hor}})$ for all $(v,\xi)\in H=(\ker dp)^\perp$, we also have
\[
  \mu(\widetilde{\Lambda} \cap H) = \Lambda_\phi 
\]
where $\Lambda_\phi$ is defined in \eqref{eq:Lambda-phi}.
 \end{remark}
\begin{prop}\label{lem:compo-LD-model-case}
We use the notations and assumptions of Proposition \ref{prop:model-case-of-convol-phases}. Let $\Omega$ be any line bundle such that $\Omega\vert_Z=\Omega_Z$. 
We note 
\[
i^* : \cE'_{\widetilde{\Lambda}}(X,\Omega)\longrightarrow \cE'_{i^*\widetilde{\Lambda}}(Z,\Omega_Z)
\]
 the restriction to $Z$ of distributions on $X$ and 
 \[
 f_*: \cE'_{i^*\widetilde{\Lambda}}(Z,\Omega_Z)\longrightarrow \cE'_{\Lambda}(Y,\Omega_Y),  
 \]
the push-forward along $f$. The map 
 \begin{align*}
 f_\# : I^m_c(X,\widetilde{\Lambda};\Omega)& \longrightarrow  I^{m+e/2+(n_X-2n_Z+n_Y)/4}_c (Y,\Lambda) \\
                u & \longmapsto f_*(i^*u)
 \end{align*}
is well defined.
\end{prop} 
For non compactly supported distributions, we get the same result by taking care of supports for the push-forward operation. For instance, giving   any $\varphi\in C^\infty(Z)$ such that $f : \supp{\varphi}\to Y$ is proper, the conclusion of the lemma holds true with 
\[
  u \longmapsto f_*(\varphi i^*u).
\]

Proposition \ref{lem:compo-LD-model-case} could be deduced from \cite[Theorem 25.2.3]{Horm-classics}, but the direct proof below is instructive.
\begin{proof}[Proof of the Lemma]
Let 
$$
A(x) = \int e^{i\widetilde{\phi}(x,\theta)}a(x,\theta)d\theta \in I^m(X,\widetilde{\Lambda}).
$$
Here $\widetilde{\phi} : U\times \RR^N\to \RR$ is a non degenerate phase function parametrizing $\widetilde{\Lambda}$ and $a\in S^{m+(n_X-2N)/4}(U\times \RR^N)$. 
Since $\WF{A}\subset \widetilde{\Lambda}$ and, by assumption, $\widetilde{\Lambda}\cap N^*Z=\emptyset$, the distribution $i^*(A)$ is well defined (\cite[Chp. 6, Section 1]{GS}) and given by the oscillatory integral
\begin{equation}\label{eq:push-forward-LD-model-case-1}
i^*(A)(z) = \int e^{i\phi(z,\theta)}a(z,\theta)d\theta \in I^{m +(n_X-n_Z)/4}(Z,i^*\widetilde{\Lambda})
\end{equation}
where we recall that $i^*\widetilde{\Lambda} =p(\widetilde{\Lambda}\cap T^*X\vert_{Z})$ with $p:T^*X\vert_{Z}\to T^*Z$ is the canonical projection, since applying the previous proposition to the case when $Y=Z$ gives exactly that $\phi=\widetilde{\phi}\vert_{(U\cap Z)\times \RR^N}$ is a non degenerate phase function parametrizing $i^*\widetilde{\Lambda}= p(\widetilde{\Lambda}\cap T^*X\vert_{Z})$.
%

The next step consists in pushing $i^*(A)$ forward with $f$. This amounts to integrate in the fibers of $f$ the Lagrangian distribution $i^*A$, which gives:
\begin{equation}\label{eq:push-forward-LD-model-case-2}
f_\#A(y) = \int_{f^{-1}(y)\times\RR^N} e^{i\phi(z,\theta)}a(z,\theta)dzd\theta
\end{equation}
where the integral is understood in the distributional sense. We already know by the previous proposition that $\phi$ is a clean phase function over $ W= f(U\cap Z)$ subordinated to the cone bundle $ (U\cap Z)\times \RR^N \to W,\ (z,\theta)\mapsto f(z)$. To conclude, it just remains to pay attention to the fact that the fiber part of the variable $z$ is not homogeneous and thus, strictly speaking, $a$ is not a symbol on $W$. Working in local coordinates, we can write
\begin{equation}
  a(x,\theta) = a(y,z',\theta)\in  S^{m+(n_X-2N)/4}( \RR^{n_Y}\times\RR^{n_Z-n_Y} \times\RR^N).
\end{equation}
Setting $\omega(z',\theta)=(\vert \theta\vert . z', \theta)$; $\psi(y,\omega)=\phi(y,z',\theta)$ and $b(y,\omega)=a(y,z',\theta)\vert\det(\omega^{-1})\vert$, we get $\vert\det(\omega^{-1})\vert=\vert\theta\vert^{n_Y-n_Z}$ and thus $b\in S^{m+(n_X-2N)/4+n_Y-n_Z}(\RR^{n_Y} \times \RR^{n_Z-n_Y+N})$. It follows that 
\begin{equation}\label{eq:push-forward-LD-model-case-3}
f_\#A(y) = \int_{f^{-1}(y)\times\RR^N} e^{i\psi(y,\omega)}b(y,\omega)d\omega
\end{equation}
belongs to $I^{m'}(Y,\Lambda)$ where 
\[
 m' - e/2 + (n_Y-2(n_Z-n_Y+N))/4 = m +(n_X-2N)/4+n_Y-n_Z
\]
that is $m'=  m +e/2  +(n_X-n_Y)/4-(n_Z-n_Y)/2= m+e/2+(n_X-2n_Z+n_Y)/4$ .
\end{proof}

\subsection{Elementary operations on Lagrangian submanifolds of $T^*G$}
Recall \cite{CDW,Mackenzie2005,Pradines1988} that a  groupoid $\Gamma$ endowed with a symplectic form is symplectic  if
\begin{equation}
 \label{eq:symplectic-gpd-cond}
 \mathrm{Gr}(m_\Gamma)= \{ (\gamma_1,\gamma_2,\gamma)\in \Gamma^3\ ;\ \gamma = \gamma_1\gamma_2\}
\end{equation}
is a Lagrangian submanifold of $  \Gamma\times \Gamma\times (-\Gamma) $.
This assumption on the graph allows us to  apply Proposition \ref{prop:push-forward-lagrangian} with $S=\Gamma^2$, $T=\Gamma$, $H=\Gamma^{(2)}$ and $\mu=m_\Gamma$.  
\begin{cor}\label{cor:clean-compo-Lagrangian-in-symplectic-gpd}
 Let $\Gamma$ be a symplectic groupoid with multiplication map $m_\Gamma$. Let $\widetilde{\Lambda}$ be a  Lagrangian local submanifold  of $\Gamma^2$. If $\widetilde{\Lambda}\cap \Gamma^{(2)} $ is clean  then 
\begin{equation}
 \label{eq:clean-compo-lagr-gpd-style-rank-2}
   \Lambda := m_\Gamma(\widetilde{\Lambda}\cap \Gamma^{(2)})
\end{equation}
is  a Lagrangian local submanifold of $\Gamma$.  
\end{cor}
Applying this with $\widetilde{\Lambda}=\Lambda_1\times\Lambda_2$ where $\Lambda_1,\Lambda_2$ are Lagrangian local submanifolds of $\Gamma$,  we get natural sufficient conditions to perform the convolution of Lagrangian submanifolds.
\begin{defn}\label{defn:convolable-Lagrangrian} Let $\Lambda_1,\Lambda_2$ be two Lagrangian local submanifolds of $\Gamma$.
  We say that $\Lambda_1$ and $\Lambda_2$ are {\sl cleanly convolable} if $\Lambda_1\times \Lambda_2$ cleanly intersects $\Gamma^{(2)}$. In that case, we denote by
\begin{equation}
 \Lambda_1 . \Lambda_2=m_\Gamma(\Lambda_1\times \Lambda_2\cap \Gamma^{(2)})
\end{equation}
the convolution product of $\Lambda_1$ and $\Lambda_2$.
\end{defn}
Another (obvious) operation is the transposition of Lagrangian submanifolds. Let $\Lambda$ be a local Lagrangian submanifold of $\Gamma$. Then 
\[
 \Lambda^\star = i_\Gamma(\Lambda)
\]
is again a local Lagrangian submanifold of $\Gamma$.

When $\Gamma=T^*G$, with $G$ a Lie groupoid \cite{CDW,LMV1}, Proposition \ref{prop:model-case-of-convol-phases} applies with $X=G^2$, $Y=G$, $Z= G^{(2)}$, $ f=m$ and this gives us a practical way to parametrize the convolution product of two conic Lagrangian submanifolds of $T^*G$ which are cleanly convolable by the help of clean phase functions.
\begin{cor}\label{prop:clean-compo-phases-gpd}
 Let  $\Lambda_1,\Lambda_2$ be local conic Lagrangian submanifolds of $T^*G\setminus 0$ which are cleanly convolable with excess $e$ and satisfy $ \Lambda_1\times \Lambda_2 \cap N^*(G^{(2)})=\emptyset$. Let   
\(
 (\gamma_1,\xi_1,\gamma_2,\xi_2)\in (\Lambda_1\times \Lambda_2)\cap\Gamma^{(2)}
\)
and  
\(
 \phi_j : U_j\times\RR^{N_j}\longrightarrow \RR
\)
 be non-degenerate phases functions parametrizing $\Lambda_j$ around $(\gamma_j,\xi_j)$,  $j=1,2$. Then
\begin{equation}
 \label{eq:clean-composition-phase-def}
 (U_1\times U_2\cap G^{(2)})\times( \RR^{N_1}\setminus 0)\times(\RR^{N_2}\setminus 0)\longrightarrow \RR,\ (\gamma_1,\gamma_2,\theta_1,\theta_2)\longmapsto
  \phi_1(\gamma_1,\theta_1)+\phi_2(\gamma_2,\theta_2) 
\end{equation}
is a phase function over $U=U_1.U_2\subset G$  associated with the cone bundle  
\begin{equation}
 \label{eq:clean-compo-phases-gpd-extended-covariable-2}
  (U_1\times U_2\cap G^{(2)})\times( \RR^{N_1}\setminus 0)\times(\RR^{N_2}\setminus 0) \longrightarrow G, \ (\gamma_1,\gamma_2,\theta_1,\theta_2)\longmapsto \gamma_1\gamma_2.
\end{equation}
This phase function is clean  with excess $e$ and parametrizes   $\Lambda_1.\Lambda_2$ around 
$(\gamma_1\gamma_2,\xi_1\oplus\xi_2)$. 
\end{cor}

\subsection{Invertibility of Lagrangian submanifolds of $T^*G$}

\begin{defn} Let $\Gamma$ be a symplectic groupoid.
\begin{enumerate}
 \item  A Lagrangian submanifold $\Lambda\subset \Gamma$ is invertible if there exists a Lagrangian submanifold $\Lambda'\subset\Gamma$  cleanly convolable with $\Lambda$ and such that 
 \begin{equation}\label{defn:condition-invertibility-Lagrangian}
  \Lambda. \Lambda'=r_\Gamma(\Lambda)\quad \text{ and }\quad \Lambda' . \Lambda =s_\Gamma(\Lambda). 
 \end{equation}
 $\Lambda'$ is then called an inverse of $\Lambda$.
\item A Lagrangian local   submanifold   $\Lambda$ is locally invertible if it can be covered by  invertible patches and any Lagrangian  local  submanifold  made of inverses of  the corresponding invertible patches is called a local inverse of $\Lambda$.
\end{enumerate}
\end{defn}
\begin{thm}\label{thm:inversibility-equiv-transversal-adjoint}
 Let $\Lambda$ be a Lagrangian submanifold of $\Gamma$. Then $\Lambda$ is locally invertible (resp. invertible) if and only if the maps 
\begin{equation}\label{prop:equiv-condition-invertibility-Lagrangian}
 r_\Gamma : \Lambda \longrightarrow \Gamma^{(0)} \text{ and }  s_\Gamma : \Lambda \longrightarrow  \Gamma^{(0)}
\end{equation}
are local diffeomorphisms (resp. diffeomorphisms onto their ranges). In that case, $\Lambda$ is transversally convolable with $\Lambda^\star$ which  provides a local inverse (resp. an inverse) of $\Lambda$. 
\end{thm}
\begin{proof}
Let assume that $\Lambda$ is locally inversible. By restricting our attention to a sufficiently small patch, we can assume that $\Lambda$ is invertible. Let $\Lambda'$ be an inverse. Firstly, note that $\Lambda.\Lambda'$ is a local submanifold of $T^*G$ contained in $\Gamma^{0}$.
Since $\Lambda.\Lambda'$ and $\Gamma^{(0)}$ are Lagrangian we have  $\dim A^*G=\dim \Lambda.\Lambda'$ and thus  each patch of $\Lambda.\Lambda'$ is an open subset of $\Gamma^{(0)}$. It follows that $\Lambda.\Lambda'$ itself is an open subset of $\Gamma^{(0)}$ and therefore a true submanifold. Now, by assumption,
\begin{equation}\label{thm:inversibility-equiv-transversal-adjoint-2}
m_\Gamma : (\Lambda\times \Lambda')^{(2)} \longrightarrow \Lambda.\Lambda'=r_\Gamma(\Lambda)
\end{equation}
is a surjective submersion. Since the map $r_\Gamma$ is equal to the identity map in restriction to $\Gamma^{(0)}$, we have the equality of maps 
\begin{equation}\label{thm:inversibility-equiv-transversal-adjoint-3}
 r_\Gamma\circ\pr{1} = r_\Gamma\circ m_\Gamma =m_\Gamma : (\Lambda\times \Lambda')^{(2)} \longrightarrow\Lambda.\Lambda'\subset \Gamma^{(0)}. 
\end{equation}
It follows that 
\begin{equation}\label{thm:inversibility-equiv-transversal-adjoint-4}
  r_\Gamma\circ\pr{1} : (\Lambda\times \Lambda')^{(2)} \longrightarrow\Lambda.\Lambda'=r_\Gamma(\Lambda)
\end{equation}
is a submersion.
Observe also that 
\begin{equation}\label{thm:inversibility-equiv-transversal-adjoint-5}
\pr{1} :  (\Lambda\times \Lambda')^{(2)}\to \Lambda
\end{equation}
 is surjective. Indeed, for any $\gamma\in\Lambda$, there exists by surjectivity of the map \eqref{thm:inversibility-equiv-transversal-adjoint-2} an element $(\gamma_1,\gamma_2)\in  (\Lambda\times \Lambda')^{(2)}$ such that $\gamma_1\gamma_2=r_\Gamma(\gamma)$. In particular $r(\gamma_1)=r(\gamma)$ and $\gamma_1^{-1}=\gamma_2\in \Lambda'$. Thus $(\gamma_1^{-1},\gamma)\in  (\Lambda'\times \Lambda)^{(2)}$ and the assumption $\Lambda'.\Lambda\subset \Gamma^{(0)}$ implies $\gamma=\gamma_1=\pr{1}(\gamma_1,\gamma_2)$. 

Since the map \eqref{thm:inversibility-equiv-transversal-adjoint-5} is surjective, we deduce from the surjectivity of the differential of \eqref{thm:inversibility-equiv-transversal-adjoint-4} at any point the surjectivity of the differential of 
$r_\Gamma : \Lambda \to \Gamma^{(0)}$ everywhere too. By equality of dimension, $r_\Gamma$ is then a local diffeomorphism. The same holds for $s_\Gamma$. 

Conversely, let us assume that $r_\Gamma,s_\Gamma : \Lambda\to \Gamma^{(0)}$ are local diffeomorphisms. Then the map 
\[
 s_\Gamma\times r_\Gamma\vert_{\Lambda \times  \Lambda^\star} : \Lambda \times  \Lambda^\star \longrightarrow \Gamma^{(0)}\times \Gamma^{(0)}
\]
is also a local diffeomorphism. It follows that 
\begin{equation}
  (\Lambda \times \Lambda^\star)^{(2)} = (s_\Gamma\times r_\Gamma)_{\vert \Lambda \times  \Lambda^\star}^{-1}(\Delta_{\Gamma^{(0)}})=
  (s_\Gamma\times r_\Gamma)^{-1}(\Delta_{\Gamma^{(0)}})\cap \Lambda\times \Lambda^\star
\end{equation}
is a submanifold of dimension $n$ of $\Gamma^2$ with tangent space given by 
\[
 T(\Lambda \times \Lambda^\star)^{(2)} = T\Gamma^{(2)}\cap T(\Lambda \times \Lambda^\star).
\]
Therefore the intersection $\Lambda \times \Lambda^\star \cap \Gamma^{(2)}$ is clean with excess  satisfying
\[
 e= \mathrm{codim}(\Lambda \times \Lambda^\star)+ \mathrm{codim}(\Gamma^{(2)}) -\mathrm{codim}((\Lambda \times \Lambda^\star)^{(2)})
  = 2n +n -3n =0,
\]
in other words, we get $(\Lambda \times \Lambda^\star ) \pitchfork \Gamma^{(2)}$. 
Moreover, for any $\delta\in \Lambda$, there exists an open conic neighborhood $U$ of $\delta$ in $\Gamma$ such that 
\[
 r_\Gamma, s_\Gamma : \Lambda_U=\Lambda\cap U \longrightarrow \Gamma^{(0)}
\]
are diffeomorphisms onto their respective images. By the previous arguments, $(\Lambda_U\times  i_\Gamma(\Lambda_U) \pitchfork \Gamma^{(2)}$ and if $\eta\in  i_\Gamma(\Lambda_U)$ is such that $(\delta,\eta)\in (\Lambda_U\times i_\Gamma(\Lambda_U))^{(2)}$ then by injectivity of $s_\Gamma$ we get $\eta=\delta^{-1}$. It follows that $\Lambda_U.i_\Gamma(\Lambda_U)=r_\Gamma(\Lambda_U)$. This proves that $\Lambda$ is locally invertible and since $i_\Gamma(\Lambda_U)=(\Lambda^\star)_{U^{-1}}$, we conclude that $\Lambda^\star$ is a local inverse. 

Now, assume that $r_\Gamma,s_\Gamma$ are diffeomorphisms onto their ranges, that is, are injective local diffeomorphisms. If there exists $\delta\in \Lambda$ and $\eta\in\Lambda^\star$ such that $\delta\eta\not\in \Gamma^{(0)}$ then $\delta, \eta^{-1}\in \Lambda$, $\delta\not=\eta^{-1}$ but $s_\Gamma(\delta)=s_\Gamma(\eta^{-1})$ which contradicts the injectivity of $s_\Gamma$. This gives the inclusion $\Lambda.\Lambda^\star\subset \Gamma^{(0)}$ and then the equality $\Lambda.\Lambda^\star=r_\Gamma(\Lambda)$ follows from the definition of $\Lambda^\star$. We get the equality $\Lambda^\star .\Lambda=s_\Gamma(\Lambda)$ using the injectivity of $r_\Gamma$.

Conversely, assume that $\Lambda'$ is an inverse of $\Lambda$. Let  $u\in r_\Gamma(\Lambda)$. Since $\Lambda.\Lambda'=r_\Gamma(\Lambda)$,
there exists $(\delta_1,\delta_1')\in (\Lambda\times\Lambda')^{(2)}$ such that $\delta_1.\delta_1'=u$. Let $\delta\in \Lambda$  be such that $r(\delta)=u$. Then $(\delta_1',\delta)\in (\Lambda'\times\Lambda)^{(2)}$ and thus $\delta_1'.\delta\in \Gamma^{(0)}$. This gives
\[
 \delta_1 = \delta_1'{}^{-1}=\delta.
\]
In other words, $r_\Gamma\vert_{\Lambda} : \Lambda \longrightarrow \Gamma^{(0)}$ is injective. The same holds for $s_\Gamma$. 
\end{proof}
\begin{rmk}\ 
 \begin{enumerate}
  \item We have proved that the (local) invertible Lagrangian submanifolds of $\Gamma$ are precisely the Lagrangian (local) bissections of $\Gamma$. Here we follow the terminology of \cite{AS2006} for bissections, while in \cite[see Paragraphs I.3 and II.1]{CDW} bissections are required to project onto $\Gamma^{(0)}$: this is a minor and technical distinction implying that the set $\mathrm{Gr}(\Gamma)$ is no more here a group but a groupoid with unit space given by the collection of open subspaces of $\Gamma^{(0)}$. 
  \item In particular we recover results from \cite{Horm-classics} in the case where $G=M\times M$ is the pair groupoid on a manifold $M$. Then a conic Lagrangian submanifold of $\Gamma=T^*G$ is (locally) invertible  if and only if it   coincides (locally) with the graph  of a partially defined homogeneous canonical transformation \cite[Sections 25.3 and 21.2]{Horm-classics}, that is, the graph of a homogeneous symplectomorphism  from an open conic subset  of $T^*M$ to another one.
 \end{enumerate}
\end{rmk}

\subsection{$G$-relations}
From now on, $\Gamma=T^*G$ for a given Lie groupoid $G$. By construction of $T^*G$, we know that $N^*(G^{(2)})=\ker m_\Gamma$ where $m_\Gamma : (T^*G)^{(2)}\to T^*G$ is regarded as a vector bundles homomorphism. Since  
$r_\Gamma\circ m_\Gamma = r_\Gamma\circ \pr{1}$, we obtain that a sufficient condition to get the technical asusmption $ \Lambda_1\times \Lambda_2 \cap N^*(G^{(2)})=\emptyset$ in Corollary \ref{prop:clean-compo-phases-gpd} is for instance $ \Lambda_1\cap \ker r_\Gamma=\emptyset$. We also proved in \cite{LMV1} that if $u\in\cD'(G,\Omega^{1/2})$ and $\WF{u}\cap \ker \sigma_\Gamma=\emptyset$ for $\sigma=s,r$ then $u$ gives by convolution an adjointable $G$-operator. For these reasons, we set
\begin{defn}\label{defn:G-rel}
A   set $\Lambda\subset T^*G\setminus 0$ is called admissible if 
\begin{equation}\label{eq:nozeros-condition}
 \Lambda\cap \ker s_\Gamma= \Lambda\cap \ker r_\Gamma =\emptyset.
\end{equation}
 A $G$-relation is an admissible conic Lagrangian submanifold of $T^*G$. Local  $G$-relations are defined accordingly. 
\end{defn}
\begin{exam}
 Any conic Lagrangian submanifold $\Lambda\subset T^*G\setminus 0$ which is invertible is a $G$-relation. Otherwise, we will deduce from the conicity assumption that the differential of either $ s_\Gamma : \Lambda\to A^*G$  or $r_\Gamma : \Lambda\to A^*G$ at some point is not injective.
\end{exam}

%
%
%
%

Let us comment further the definition. Firstly, for any $G$-relation $\Lambda$ , the set  
\begin{equation}\label{eq:rem-about-G-rel-1}
 m^*(\Lambda)=\{ (\gamma_1,\gamma_2,\zeta)\in T^*(G^{(2)})\ ;\ \exists (\gamma,\xi)\in \Lambda\ ,\ \gamma_1\gamma_2=\gamma,\ {}^t(dm)_{\gamma_1,\gamma_2}(\xi)=\zeta\}\subset T^*(G^{(2)})\setminus 0, 
\end{equation}
is still a conic Lagrangian submanifold of $T^*(G^{(2)})\setminus 0$, since $m: G^{(2)}\to G$ is a surjective submersion \cite[Chp. 4, Proposition 4.1]{GS}. Moreover, remembering the map $\widetilde{m_\Gamma}$ defined after  Diagram \eqref{diag:cot-gpd-overview}) we get:
\begin{equation}\label{eq:rem-about-G-rel-2}
 m^*(\Lambda)= \widetilde{m_\Gamma}^{-1}(\Lambda). 
\end{equation}
This implies the inclusion, which will be reused later:
\begin{equation}\label{eq:rem-about-G-rel-3}
 \ker d\widetilde{m_\Gamma} \subset Tm^*(\Lambda).
\end{equation}
 Next, we relate Condition \eqref{eq:nozeros-condition}  with the \enquote{no-zero} condition \cite{Horm-classics,Melrose1981} required for homogeneous canonical relations. For that purpose, we introduce the family of sets $i_x^*(m^*(\Lambda))=m_x^*(\Lambda)$, $x\in G^{(0)}$, that is:
\begin{equation}\label{eq:rem-about-G-rel-4}
 m_x^*\Lambda =\{ (\gamma_1,\xi_1,\gamma_2,\xi_2)\in T^*G^x\times T^*G_x\ ;\ \exists (\gamma,\xi)\in \Lambda,\ \gamma_1\gamma_2=\gamma,\ {}^t(dm_x)_{\gamma_1,\gamma_2}(\xi)=(\xi_1,\xi_2)\},
\end{equation}
and we prove
 \begin{prop}\label{lem:equiv-admiss-in-G-n-in-pull-back-by-jx-2}
 Let $W\subset T^*G\setminus 0$. Then  $W$  is admissible if and only if 
\begin{equation}\label{eq:no-zero-family}
 m_x^*W\subset (T^*G_x\setminus 0)\times (T^*G^x\setminus 0),\ \quad \forall x\in G^{(0)}.
\end{equation}
\end{prop}
\begin{proof}
Differentiating $m_x$, we get $d(m_x)_{(\gamma_1,\gamma_2)}(t_1,t_2)= dR_{\gamma_2}(t_1)+dL_{\gamma_1}(t_2)$, $t_1\in T_{\gamma_1}G_x$, $t_2\in T_{\gamma_2}G^x$ thus  
\[
 {}^t d(m_x)_{(\gamma_1,\gamma_2)}(\xi)= ({}^tdR_{\gamma_2}(\xi),{}^tdL_{\gamma_1}(\xi))= ({}^tdR_{\gamma_1^{-1}}(\overline{r}(\xi)), {}^tdL_{\gamma_2^{-1}}(\overline{s}(\xi)))\in T_{\gamma_1}^*G_x\times  T_{\gamma_2}^*G^x.
\]
It follows that 
\[
 m_x^*(W)=\left\lbrace \left(\gamma_1,{}^tdR_{\gamma_1^{-1}}(\overline{r}(\xi)), \gamma_2, {}^tdL_{\gamma_2^{-1}}(\overline{s}(\xi))\right)\ ;\  
\ (\gamma_1,\gamma_2)\in G_x\times G^x,\  (\gamma_1\gamma_2,\xi)\in W \right\rbrace  .
\]
Since ${}^tdR_{\gamma_i^{-1}}$ and ${}^tdL_{\gamma_2^{-1}} $ are bijective, the result follows. 
\end{proof}
 Observe that if $\Lambda$ is a $G$-relation, then the subsets $m_x^*(\Lambda)$ are not necessarily   Lagrangian submanifolds of $T^*(G_x \times G^x)$. An example will be given below (Example \ref{exam:G-FIO-not-G-FFIO}): for the $G$-relation  $\Lambda$ considered in \eqref{eq:G-rel-not-family-G-rel}, we get that  $m^*_{(0,0)}\Lambda$, given in \eqref{exam:FIO-not-FFIO-10}, is isotropic but not Lagrangian. This bad behavior leads us to 
\begin{defn}\label{defn:family-G-relation}\ 
 A family $G$-relation $\Lambda$ is a $G$-relation such that the pull-back  $m^*(\Lambda)\subset T^*(G^{(2)})$ is a Lagrangian submanifold
  transverse to $\pi : G^{(2)}\to G^{(0)}$.
Local family $G$-relations are defined accordingly. 
\end{defn}
If $\Lambda$ is a $G$-relation, we obtain from Theorems \ref{thm:gluing-family-of-Lagrangian} and \ref{thm:Lagrangian-slicing} that  $\Lambda$ is a family $G$-relation  if and only if $(m_x^*\Lambda)_{x\in G^{(0)}}$ is a $C^\infty$ family of conic Lagrangian submanifolds subordinated to $\pi : G^{(2)}\to G^{(0)}$, which gives a first justification for the terminology. 
 
Next, if $\Lambda$ is a family $G$-relation, then the family given by $\Lambda_x=m_x^*(\Lambda)$, $x\in G^{(0)}$, is equivariant in the following sense:
\begin{equation}\label{eq:equivariance-family-lag}
 \forall x,y\in G^{(0)},\ \forall (\gamma_1,\gamma_2,\xi_1,\xi_2)\in \Lambda_x,\ \forall \gamma\in G_y^x,\quad (\gamma_1\gamma,\gamma^{-1}\gamma_2,{}^t(dR_{\gamma^{-1}})(\xi_1),{}^t(dL_{\gamma})(\xi_2))\in \Lambda_y.
\end{equation}
Indeed, if $c_\gamma : G_x\times G^x\longrightarrow G_y\times G^y$, $\gamma\in G_y^x$, is the map defined by $c_\gamma(\gamma_1,\gamma_2)= (\gamma_1\gamma,\gamma^{-1}\gamma_2)$, then the commutative diagram
\begin{equation}
 \label{diag:justify-equivar-family-lag}
  \xymatrix{
     G_x\times G^x \ar[d]_{c_\gamma}\ar[rr]^{m_x}& & G \\
      G_y\times G^y \ar[urr]_{m_y}&  &
}
\end{equation}
yields the equality $(c_\gamma)^*(\Lambda_y)=\Lambda_x$ and then the property \eqref{eq:equivariance-family-lag}.
All the previous remarks are unchanged if submanifolds are replaced by local submanifolds and we now prove that all local family  $G$-relations come from equivariant families, which ends the  justification of the terminology. 
\begin{thm}\label{thm:bijection-Lagrangian-over-G-n-equiv-family}
Let $(\Lambda_x)_x$ be a $C^\infty$ equivariant family of conic local canonical relations in $(T^*G_x\setminus 0)\times (T^*G^x\setminus 0)$. Then there exists a unique local family $G$-relation $\Lambda$ such that 
\[
   m_x^*(\Lambda)=\Lambda_x, \ \text{ for all } x\in G^{(0)}.
\]
\end{thm}
\begin{proof}
To prove the existence, we can decompose the family into patches and then, we can assume that  $(\Lambda_x)_x$ is a family of submanifolds. Let $\widetilde{\Lambda}\subset T^*(G^{(2)})\setminus 0$  denotes its gluing (Theorem \ref{thm:gluing-family-of-Lagrangian}).   Recall that it is the unique Lagrangian submanifold such that $i_x^*\widetilde{\Lambda}=\Lambda_x$ for all $x$. 
In an appropriate local trivialization of $\pi : G^{(2)}\to G^{(0)}$, we have 
\begin{equation}
 \widetilde{\Lambda} = \{(\gamma_1,\gamma_2,\xi_1,\xi_2,\tau)\in T^*(G^{(2)})\ ; \ (\gamma_1,\gamma_2,\xi_1,\xi_2)\in \Lambda_{x} \}
\end{equation}
where $\tau$ is a $C^\infty$ function of $\xi_1,\xi_2$ and $x=s(\gamma_1)=r(\gamma_2)$. It is understood that 
$(\xi_1,\xi_2,\tau)\in T^*_{(\gamma_1,\gamma_2)}G^{(2)}\simeq T^*_{\gamma_1}G_x\times T^*_{\gamma_2}G^x \times T^*_xG^{(0)}$ where the decomposition comes from the local trivialisation of  $\pi : G^{(2)}\to G^{(0)}$.

Let $\widetilde{\lambda}=(\delta,\xi)\in\widetilde{\Lambda}$ with $\delta=(\gamma_1,\gamma_2)$ and $\xi=(\xi_1,\xi_2,\tau)$. Let   $u=(u_1,u_2)\in \ker dm_{\gamma_1,\gamma_2}$ and choose a $C^\infty$ path  $t\mapsto \gamma(t)$ in $G$ such that $\gamma(0)=x$, $\frac{d}{dt}\gamma_1\gamma(t)\vert_{t=0}=u_1$, $\frac{d}{dt}\gamma(t)^{-1}\gamma_2\vert_{t=0}=u_2$. It gives rise to a $C^\infty$ path in $(\ker d\pi)^*$ defined by
\begin{equation}
 \lambda_t = (\gamma_1\gamma(t),\gamma(t)^{-1}\gamma_2,{}^t(dR_{\gamma(t)^{-1}})(\xi_1),{}^t(dL_{\gamma(t)})(\xi_2))
\end{equation}
Thanks to the equivariance, we have
\begin{equation}
 \lambda_t\in \Lambda_{s(\gamma(t))} \text{ for all } t.
\end{equation}
Thus, we get a $C^\infty$ path in $\widetilde{\Lambda}$ as well:
\begin{equation}
 \widetilde{\lambda}(t) = (\gamma_1\gamma(t),\gamma(t)^{-1}\gamma_2,{}^t(dR_{\gamma(t)^{-1}})(\xi_1),{}^t(dL_{\gamma(t)})(\xi_2),\tau(t)) =(\delta(t),\xi(t))
\end{equation}
Since $\widetilde{\Lambda}$ is conic and Lagrangian, the canonical one form $\alpha=\xi d\delta$ vanishes identically on it and in particular we get for all $t$
\begin{equation}
 (\widetilde{\lambda})^*\alpha(t) = \langle  \xi(t)  , \delta'(t)\rangle =0.
\end{equation}
For $t=0$, this gives $\langle \xi ,\ u\rangle =0$, and therefore
\begin{equation}\label{eq:bijection-Lagrangian-over-G-n-equiv-family-5}
 \widetilde{\Lambda}\subset (\ker dm)^\perp=\rho(T^*G)^{(2)}\subset \rho(T^*G^{2} )
\end{equation}
where $\rho : T^*G^2\to T^*(G^{(2)})$ is the natural restriction of linear forms seen in the diagram (\ref{diag:cot-gpd-overview}). Note that for every Lagrangian submanifold $\widetilde{\Lambda} $ in $T^*(G^{(2)})$, then $\Lambda =\rho ^{-1}(\widetilde{\Lambda})$ is a Lagrangian submanifold in $T^*G^{2}$ (it is the push-forward of $\widetilde{\Lambda} $ by the natural immersion $G^{(2)} \to G^2$- see \cite[Prop 4.2]{GS}).

We can then apply Corollary \ref{cor:clean-compo-Lagrangian-in-symplectic-gpd} to the Lagrangian $\rho^{-1}(\widetilde{\Lambda})$.
Indeed, by construction, $\rho^{-1}(\widetilde{\Lambda}) \subset (T^*G)^{(2)}$ and thus the clean intersection assumption of Corollary \ref{cor:clean-compo-Lagrangian-in-symplectic-gpd} is trivially satisfied. It follows that $\Lambda=m_\Gamma(\rho^{-1}(\widetilde{\Lambda}))$ is a local $G$-relation such that $m^*(\Lambda)=\widetilde{\Lambda}$. Hence it is a local family $G$-relation such that $m_x^*\Lambda=\Lambda_x$ for any $x$ by Theorem \ref{thm:gluing-family-of-Lagrangian}.
%

Let $\Lambda,\Lambda'$ be two local family $G$-relations answering the question. Set $\widetilde{\Lambda}=m^*(\Lambda)=\widetilde{m_\Gamma}^{-1}(\Lambda)$ and $\widetilde{\Lambda'}=m^*(\Lambda')=\widetilde{m_\Gamma}^{-1}(\Lambda')$. Since 
\begin{equation}
 m_x^*(\Lambda) =i_x^*\widetilde{\Lambda} = \Lambda_x = m_x^*(\Lambda')= i_x^*\widetilde{\Lambda'},\ \forall x,
\end{equation}
Theorem \ref{thm:gluing-family-of-Lagrangian} implies $\widetilde{\Lambda}=\widetilde{\Lambda'}$. Since $\widetilde{m_\Gamma}$ is surjective, we conclude
\begin{equation}
 \Lambda = \widetilde{m_\Gamma}(\widetilde{m_\Gamma}^{-1}(\Lambda)) = \widetilde{m_\Gamma}(\widetilde{m_\Gamma}^{-1}(\Lambda'))= \Lambda'.
\end{equation}
\end{proof}
We give a another characterization  of family $G$-relations.
\begin{prop}\label{prop:charact-family-G-rel}
Let $\Lambda$ be a  $G$-relation and $ p: T^*G\to G$ the natural projection map. Then  $\Lambda$ is a family $G$-relation   if and only if 
\begin{equation}\label{eq:transver-cond-G-rel-1}
  \forall x\in G^{(0)},\quad m_x\pitchfork p\vert_\Lambda,
\end{equation}
  that is, $dm_x(T_{\gamma_1}G_x\times T_{\gamma_2}G^x)+dp(T_{(\gamma,\xi)}\Lambda)= T_\gamma G$, for all $x$, $(\gamma,\xi)\in\Lambda$ and $(\gamma_1,\gamma_2)\in m_x^{-1}(\gamma)$.
\end{prop}
\begin{proof}
Let $\Lambda$ be a $G$-relation. The inclusion \eqref{eq:rem-about-G-rel-3} and the equality $dp^{(2)}(\ker d\widetilde{m_\Gamma})=\ker dm$ yields the inclusion 
\begin{equation}\label{proof:charact-family-G-rel-1}
 \ker dm \subset dp(Tm^*\Lambda).
\end{equation}
Therefore, for all $x\in G^{(0)}$ and omitting other base points, we have 
\begin{eqnarray*}
 &&dm(TG_x\times G^x) + dp(T\Lambda)= TG \\
   &\Leftrightarrow & dm(TG_x\times G^x) + dp(T\widetilde{m_\Gamma}(\widetilde{m_\Gamma}^{-1}(\Lambda)))= TG \\
    &\Leftrightarrow & dm(TG_x\times G^x) + dm.dp^{(2)}(T \widetilde{m_\Gamma}^{-1}(\Lambda))= TG \\ 
    &\Leftrightarrow & TG_x\times G^x + dp^{(2)} (T \widetilde{m_\Gamma}^{-1}(\Lambda)) = TG^{(2)}\qquad  \text{ by } \eqref{proof:charact-family-G-rel-1} \\
    &\Leftrightarrow &  d\pi.dp^{(2)} (T \widetilde{m_\Gamma}^{-1}(\Lambda)) = T_xG^{(0)}. 
\end{eqnarray*}
The last line means that $\widetilde{m_\Gamma}^{-1}(\Lambda)$ is transversal to $\pi$ so the proof is ended. 
\end{proof}
The condition introduced in Proposition \ref{prop:charact-family-G-rel} has a strong geometrical meaning. We have 
\begin{equation}\label{def:feuilletage_groupoid}
 (dm_x)_{(\gamma_1,\gamma_2)}(T_{\gamma_1}G_x\times T_{\gamma_2}G^x) = T_\gamma G_{s(\gamma)}+T_\gamma G^{r(\gamma)} =T_\gamma\cF_G;
\end{equation}
for any $(\gamma_1,\gamma_2)\in G_x\times G^x$ such that $\gamma_1\gamma_2=\gamma$. Here $T_\gamma\cF_G$ denotes the tangent space at $\gamma$ of the leaf of $\cF_G$ passing through $\gamma$. It follows that   the condition introduced in Proposition \ref{prop:charact-family-G-rel} means 
\begin{equation}\label{eq:transver-cond-G-rel-bis}
  p\vert_\Lambda : \Lambda \to G \text { and } \cF_G \text{ are transversal},
\end{equation}
that is, 
\begin{equation}\label{eq:transver-cond-G-rel-ter}
 T_\gamma\cF_G + dp(T_{\gamma,\xi}\Lambda) = T_\gamma G, \quad\text{ for all } (\gamma,\xi)\in \Lambda. 
\end{equation}
%
Furthermore, we may get rid of the projection $p$. Indeed,  \eqref{eq:transver-cond-G-rel-ter} is clearly equivalent to 
\begin{equation}\label{eq:transver-cond-G-rel-4}
 T_{(\gamma,\xi)}(T_L^*G) + T_{(\gamma,\xi)}\Lambda = T_{(\gamma,\xi)}T^*G, \quad\text{ for all } (\gamma,\xi)\in \Lambda,
\end{equation}
where  $L$ is the leaf of $\cF$ containing $\gamma$. That is, 
\begin{equation}\label{eq:transver-cond-G-rel-5}
 T_L^*G \pitchfork \Lambda   \quad\text{ for all } L\in\cF_G.
\end{equation}

 We deduce from \eqref{eq:transver-cond-G-rel-ter}:
\begin{prop}
 Let $\Lambda$ be a  $G$-relation. Then $\Lambda$ is a family $G$-relation if and only if
 \begin{equation*}
 \forall (\gamma,\xi)\in \Lambda,\ dr(T_\gamma G_{s(\gamma)})+ dr(d\pi T_{\gamma,\xi}\Lambda) =T_{r(\gamma)} G^{(0)} \text{ or } ds(T_\gamma G^{r(\gamma)})+ ds(d\pi T_{\gamma,\xi}\Lambda) =T_{s(\gamma)}  G^{(0)}.
 \end{equation*}
\end{prop}
Therefore, if $\Lambda$ is a $G$-relation and $r\circ \pi : \Lambda\to G^{(0)}$ or $ s\circ\pi: \Lambda\to G^{(0)}$ are submersions then $\Lambda$ is a family $G$-relation. The converse is false: consider $G=X\times X\times Y$ with its natural structure of groupoid (fibered pair groupoid) and $\Lambda=N^*V\setminus 0$ where $V=\{(x_0,x_0)\}\times Y$. 
 
\begin{defn}\label{defn:strong-G-relations}
 A $G$-relation onto which the maps $s\circ \pi$ and $r \circ \pi$ are submersions is called a strong $G$-relation. 
\end{defn}
Next, we analyse the behavior of family $G$-relations under convolution. Unfortunately, it is not true that the convolution of   family $G$-relations is a   family $G$-relation. 
\begin{exam} Set $X=Z=\RR^n$, $n=k+(n-k)$, $G=X\times X\times Z\rightrightarrows X\times Z$, decompose $z\in Z$ into $(z',z'')$ with $z'\in\RR^k,z''\in\RR^{n-k}$ and consider $C^\infty$ maps  $x_j,y_j : Z\to X$, $j=1,2$  defined by 
\[
x_1(z)=y_1(z)=y_2(z)=z \text{ and } x_2(z)=(z',-z'').
\]
\ignore{\red{\[
[ y_1(z)=x_2(z) \Leftrightarrow z''=0]\quad \text{ and } \quad [ d_{z''}(x_2 - y_1) \text{ injective } ]
\]}}
Introduce the submanifolds of $G$ 
\[
 V_j = \mathop{Graph}(x_j,y_j) = \{ (x_j(z),y_j(z),z)\ ;\ z\in Z\},\quad j=1,2
\]
and the conic Lagrangian submanifolds of $T^*G\setminus 0$
\begin{align*}
\Lambda_j =\{ (x_j(z),\xi_j,y_j(z),\eta_j,z,-{}^tdx_j(\xi_j)-{}^tdy_j(\eta_j))\ ; \ z\in Z,\ \xi_j,\eta_j\in\RR^n\setminus 0\}\subset N^*V_j.
\end{align*}
Thanks to the  subsets that we have removed from the conormal spaces, $\Lambda_1$ and $\Lambda_2$ are admissible. Since $T\cF= TX\times TX\times 0$, the transversality condition \eqref{eq:transver-cond-G-rel-ter} is satisfied for a given $\Lambda$  if and only if $d\pi(T\Lambda)$ projects onto $TZ$ . This is clearly the case for $\Lambda_1,\Lambda_2$ which are then family $G$-relations.
With the choices made, the intersection 
\begin{align*}
 & \Lambda_1\times \Lambda_2 \cap (T^*G)^{(2)} \\
&  = \{ (x_1(z'),\xi_1,y_1(z'),\eta_1,z',-{}^td(x_1,y_1)(\xi_1,\eta_1), y_1(z'),-\eta_1,y_2(z'),\eta_2,z',-{}^td(x_2,y_2)(-\eta_1,\eta_2)\ ; \\  & \qquad z'\in \RR^k,\ \xi_1,\eta_1,\eta_2\in\RR^n\setminus 0 \}
\end{align*}
is clean. We obtain
\begin{align*}
\Lambda_1 . \Lambda_2 & =\{  (x_1(z'),\xi_1,y_2(z'),\eta_2,z',-{}^td(x_1,y_1)(\xi_1,\eta_1)-{}^td(x_2,y_2)(-\eta_1,\eta_2)\ ; \\  & \qquad z'\in \RR^k,\ \xi_1,\eta_1,\eta_2\in\RR^n\setminus 0 \}.
\end{align*}
Here, $\Lambda= \Lambda_1 . \Lambda_2$ is a $G$-relation but not a family $G$-relation since the projection on $TZ$ of  $d\pi (T\Lambda)$ is $T\RR^k\times 0$.
\end{exam}
\ignore{
\begin{exam}
Let $H_1,H_2$ be two transversal vector planes in $\RR^3$ and $H$ a third vector plane such that the restrictions of the orthogonal projection $p : \RR^3\to H$  to $H_1$ and $H_2$ is bijective. Let $G=H\times H\times \RR^3$ be the cartesian product of the pair groupoid $H\times H$ by the group $(\RR^3,+)$. We identify $\RR^3{}^*\simeq\RR^3$ using the euclidean structure.
Let $a:H^*\times H^* \to \RR^3$ be a linear map   which will be given later and consider 
\begin{equation}
 \Lambda_j = \{  (p(z),\xi,p(z),\eta, t+a(\xi,\eta), z) \ ;\ z\in H_j, t\in H_j^\perp, \xi,\eta\in H^*\setminus 0\}\subset T^*G, \quad j=1,2.
\end{equation}
 These are $C^\infty$ conic submanifolds of dimensions $7$. We may adjust $a$ in order that they are Lagrangian. Denote $\omega$ the symplectic form of $T^*G$.
\begin{align*}
   &\omega ((p(z_1),\xi_1,p(z_1),\eta_1,t_1+a(\xi_1,\eta_1),z_1)\ ;\ (p(z_2),\xi_2,p(z_2),\eta_2,t_2+a(\xi_2,\eta_2),z_2)) \\ 
   = & \xi_2((p(z_1)) -\xi_1(p(z_2)) + \eta_2(p(z_1))-\eta_1(p(z_2))+z_2(a(\xi_1,\eta_1))-z_1(a(\xi_2,\eta_2))\\
 = & (\xi_2+\eta_2)\circ p (z_1)-z_1(a(\xi_2,\eta_2)) \ - \ (\xi_1+\eta_1)\circ p (z_2)-z_2(a(\xi_1,\eta_1)).
\end{align*}
Thus, we define $a$ to be the linear map $H^*\times H^* \to \RR^3$ defined by 
\[
 a(\xi,\eta).z  = (\xi+\eta)\circ p (z), \quad \forall z\in\RR^3.
\]
These choices give isotropic submanifolds $\Lambda_1,\Lambda_2$ of dimension $7$, they are therefore Lagrangian. They are also strong family $G$-relations since, for instance,
\[
 s\circ \pi : T^*G\to G^{(0)}\simeq H,\ (x,\xi,y,\eta,t,z)\mapsto y
\]
and we thus get by assumption on $p$ that $s\circ \pi :\Lambda_j\to H$ has a surjective differential everywhere. 

Recall now that $\Gamma^{(2)}= (T^*G)^{(2)}$ is the submanifold consisting of the elements
\begin{equation}
 (x_1,\xi_1,y,\eta,t_1,z,y,-\eta,y_2,\eta_2,t_2,z)
\end{equation}
where $x_1,\xi_1,t_1,t_2$ and $y,\eta,z$ are arbitrary, and the product is then equal to $(x_1,\xi_1,y_2,\eta_2,t_1+t_2,z)$.
 
Since $(T^*G)^{(2)}$ is a linear subspace and $\Lambda_1\times \Lambda_2$ is open  in a linear subspace, their intersection is automatically clean. 

Finally, denoting by $K$ the line $H_1\cap H_2$, we have
\begin{equation}
 \Lambda_1*\Lambda_2 = \{ (p(z),\xi,p(z),\eta,t+a(\xi,\eta),z)\ ;\ \xi,\eta\in H^*\setminus 0,\ t\in K^\perp, \ z\in K \}
\end{equation}
from which we see that $s\circ \pi(\Lambda) = p(K)\subsetneq H$, which proves that $\Lambda_1.\Lambda_2$ is not a strong family $G$-relation. 
\orange{remarques : \\
1) en mixant les deux exemples, on doit pouvoir produire une exemple de family $G$-relations fortes convolables dont le produit de convolution n'est pas une family $G$-relation, mais je n'ai plus le courage...
2) J'ai emprunt\'e des chemins tortueux pour en arriver l\`a, je ne serai pas \'etonn\'e qu'il existe des contre-exemples plus simples et plus lumineux.... 
}
\end{exam}
}
There are also contre-examples for strong $G$-relations. Actually, to obtain that $\Lambda_1 * \Lambda_2$ is a family $G$-relation, what matters is the position of the cartesian product $\Lambda_1\times \Lambda_2$  with respect to $\Gamma^{(2)}$ and not the position of each $\Lambda_j$ in $T^*G$. 
\begin{thm}\label{thm:myster-cdt-stab-G-rel-conv}
 Let  $\Lambda_1,\Lambda_2$ be convolable $G$-relations. 
 \begin{enumerate}
  \item $\Lambda_1.\Lambda_2$ is a local $G$-relation and   
\begin{equation}\label{thm:myster-cdt-stab-G-rel-conv-1}
 \Lambda_1.\Lambda_2 = (\Lambda_1\cup 0)*(\Lambda_2\cup 0)\setminus 0.
\end{equation}
\item  $\Lambda_1.\Lambda_2$ is a local family $G$-relation if and only if $\Lambda_1$ and $\Lambda_2$ satisfy 
\begin{equation}\label{eq:unwanted-condition}
(T_{\gamma_1}\cF_G\times T_{\gamma_2}\cF_G)^{(2)} + dp^2(T_{(\gamma_1,\xi_1,\gamma_2,\xi_2)}(\Lambda_1\times \Lambda_2)^{(2)}) = T_{(\gamma_1,\gamma_2)}G^{(2)}  
\end{equation}
for all $(\gamma_1,\xi_1,\gamma_2,\xi_2)\in (\Lambda_1\times \Lambda_2)^{(2)}$. 
\end{enumerate}
\end{thm}
\begin{rmk}\ 
 \begin{enumerate}
  \item The right hand side in \eqref{thm:myster-cdt-stab-G-rel-conv-1} is the natural set containing the wave front set of the  convolution product $u_1*u_2$ of distributions on $G$ such that $\WF{u_j}\subset \Lambda_j$ \cite{LMV1}.
  \item The conclusions of the theorem are identical if we start with local submanifolds.
 \end{enumerate}
\end{rmk}
Clean convolability assumption together Condition \eqref{eq:unwanted-condition} will be called {\sl complete convolability}.
The proof of the theorem uses an elementary fact about Lie groupoids. 
\begin{lem}\label{lem:stab-foliat-gpd-under-mult}
For any $(\gamma_1,\gamma_2)\in G^{(2)}$, we have
\begin{equation}
\label{eq:lem-stab-fol-under-mult}
(T_{\gamma_1}\cF_G\times T_{\gamma_2}\cF_G)^{(2)} = (dm_{(\gamma_1,\gamma_2)})^{-1}(T_{\gamma_1\gamma_2} \cF_G).
\end{equation}
\end{lem}
\begin{proof}[Proof of the lemma] 
Let $T$ be a Lie groupoid and $\cF_T$ its canonical foliation. If $(\delta_1,\delta_2)\in T^{(2)}$ then 
 $\delta_1,\delta_2$ and $\delta=\delta_1\delta_2$ are in the same leaf $L$. From the very definition of the leaves of $\cF_T$, we get
\begin{equation}\label{proof:stab-foliat-gpd-under-mult-v2-1}
 (L\times L)^{(2)}=m_T^{-1}(L).
\end{equation}
If $G$ is a Lie groupoid and $T=TG$, we have \(\displaystyle \cF_{TG} = \{ TL\ ; \ L\in \cF_G\}\).
The lemma follows. 
\end{proof}

\begin{proof}[Proof of the theorem]\
\begin{enumerate}
 \item If $\Lambda_j$, $j=1,2$ is admissible, then $(\Lambda_j\times G\times\{0\})^{(2)}= (G\times\{0\}\times\Lambda_j)^{(2)} =\emptyset$, which yields  \eqref{thm:myster-cdt-stab-G-rel-conv-1}, and  
 $r_\Gamma(\Lambda_1.\Lambda_2)\subset r_\Gamma(\Lambda_1)$ and $s_\Gamma(\Lambda_1.\Lambda_2)\subset s_\Gamma(\Lambda_2)$, which yields the admissibility of $\Lambda_1.\Lambda_2$. We then know that $\Lambda=\Lambda_1.\Lambda_2$ is a  local Lagrangian submanifold of $T^*G\setminus 0$ by Corollary \ref{cor:clean-compo-Lagrangian-in-symplectic-gpd}. The homogeneity of $\Lambda$ in the fibers is obvious. 
 
\item Using the equalities
\begin{equation}
 dm.dp^2(T((\Lambda_1\times\Lambda_2)^{(2)}))= dp .dm_\Gamma (T((\Lambda_1\times\Lambda_2)^{(2)})) =dp (T\Lambda)
\end{equation}
Lemma \ref{lem:stab-foliat-gpd-under-mult} and the fact that $\ker(dm) \subset (T\cF_G\times T\cF_G)^{(2)}$, we get the equivalence
\begin{eqnarray*}
  T\cF_G+dp(T\Lambda)= TG \Leftrightarrow (T\cF_G\times T\cF_G)^{(2)} + dp^2(T(\Lambda_1\times \Lambda_2)^{(2)}) = TG^{(2)}.
\end{eqnarray*}
where the suitable base points are understood.
\end{enumerate}
\end{proof}

\begin{rmk}
Geometrically, Condition \eqref{eq:unwanted-condition} means that the composable part of $\Lambda_1\times\Lambda_2$ has a projection into $G^{(0)}$ transversal to the canonical  foliation of $\cF_{G^{(0)}}$.  More precisely, it is easy to check that  \eqref{eq:unwanted-condition} is equivalent to  
\begin{equation}
 d\sigma^2(T_{(\gamma_1,\xi_1,\gamma_2,\xi_2)}(\Lambda_1\times \Lambda_2)^{(2)}) +TO_x = T_xG^{(0)}, \quad \forall 
 (\gamma_1,\xi_1,\gamma_2,\xi_2)\in (\Lambda_1\times \Lambda_2)^{(2)}).
\end{equation}
Here $x=s(\gamma_1)=r(\gamma_2)$ and $\sigma^2 :(T^*G)^{(2)}\to G^{(0)}$ is  defined by $\sigma^2(\gamma_1,\xi_1,\gamma_2,\xi_2)=s(\gamma_1)$.
\end{rmk}

\begin{prop}\label{prop:G-relation-passing-to-subgroupoid}
 Let $\Lambda$ be a $G$-relation and $Y\subset G^{(0)}$ be a saturated submanifold. We note $H=G_Y^Y$ the induced Lie subgroupoid, $i : H\hookrightarrow G$ the inclusion and $\rho :T^*_HG\longrightarrow T^*H$ the restriction map.  
 \begin{enumerate}
 \item If $i$ and $p :\Lambda\to G$ are transversal, then $\Lambda\cap N^*H=\emptyset$ and $i^* \Lambda$ is a local $H$-relation.
  \item If $\Lambda$ is a family, then the assumption in (1) is satisfied and $\rho^* \Lambda$ is a local family $H$-relation.
 \end{enumerate}
\end{prop}
\begin{proof}
\begin{enumerate}
\item By transversality of $H$ and $p :\Lambda\to G$, the set $i^*\Lambda=\rho(\Lambda)$ is a local Lagrangian submanifold. Since $Y$ is saturated, we have $G_x=H_x$ and  $G^x=H^x$ for all $x\in Y$. This yields the equality $A_Y^*G=A^*H$ and the commutative diagram

 \begin{center}
 \begin{tikzcd}
 T_H^*G \arrow[rr, "\rho"] \arrow[d, xshift=0.7ex, "s_{T^*G}"] \arrow[d, xshift=-0.7ex, "r_{T^*G}"'] & & \arrow[d, xshift=0.7ex, "s_{T^*H}"] \arrow[d, xshift=-0.7ex, "r_{T^*H}"'] T^*H \\
 A_Y^*G \arrow[rr, equal] & & A^*H 
 \end{tikzcd}
 \end{center}
In particular $\ker \sigma_{T^*H}=\rho (\ker \sigma_{T^*G})$, $\sigma=s,r$ and the admissibility of $\Lambda$ implies 
\[
 \rho(\Lambda)\cap \ker \sigma_{T^*H}=\rho(\Lambda \cap \ker \sigma_{T^*G})= \emptyset,
\]
where the first equality holds for $\ker \rho \subset \ker \sigma_{T^*G}$.

\item By assumption we have
\begin{equation}\label{eq:G-relation-passing-to-subgroupoid-2}
 T_\gamma \cF_G + dp_G(T_{(\gamma,\xi)}\Lambda) = T_\gamma G, \ \forall (\gamma,\xi)\in \Lambda\cap T_H^*G.
\end{equation}
By saturation of $Y$ again, we have $\cF_H=\{L\in \cF_G\ ;\ L\cap H\not=\emptyset\}$. Thus $T_\gamma \cF_G=T_\gamma \cF_H\subset T_\gamma H $, for all $\gamma\in H$ and \eqref{eq:G-relation-passing-to-subgroupoid-2} gives 
\begin{equation}
\label{eq:G-relation-passing-to-subgroupoid-2-0}
 T_\gamma H + dp(T_{(\gamma,\xi)}\Lambda) = T_\gamma G,\ \forall (\gamma,\xi)\in \Lambda\cap T_H^*G.
\end{equation}
which is the assumption made in (1) and observe that is also equivalent to the property:
\begin{equation}
\label{eq:G-relation-passing-to-subgroupoid-2-1}
  T^*_HG \pitchfork \Lambda.
\end{equation}
We also get from \eqref{eq:G-relation-passing-to-subgroupoid-2}
\begin{equation}\label{eq:G-relation-passing-to-subgroupoid-2-2}
  T_\gamma H= (T_\gamma \cF_G + dp_G(T_{(\gamma,\xi)}\Lambda))\cap T_\gamma H=T_\gamma \cF_H + dp_G(T_{(\gamma,\xi)}\Lambda)\cap T_\gamma H , \ \forall (\gamma,\xi)\in \Lambda\cap T_H^*G.
\end{equation}
Furthermore, 
\begin{eqnarray}\label{eq:G-relation-passing-to-subgroupoid-2-3}
 dp_G(T_{(\gamma,\xi)}\Lambda)\cap T_\gamma H
 &=&  dp_G(T_{(\gamma,\xi)}\Lambda)\cap T_{(\gamma,\xi)}T^*_HG) \text{ since } \ker d(p_G)\vert_H\subset T^*_HG \nonumber\\
 &=& dp_G(T_{(\gamma,\xi)}(\Lambda \cap T^*_HG)) \text{ using } \eqref{eq:G-relation-passing-to-subgroupoid-2-1}\nonumber\\
 &=& dp_H\circ d\rho (T_{(\gamma,\xi)}(\Lambda \cap T^*_HG)) \text{ since } (p_G)\vert_H= p_H\circ \rho \nonumber\\
 &=& dp_H(T\rho(\Lambda)) \text{ by definition of } \rho(\Lambda).
\end{eqnarray}
Using the result of this computation in \eqref{eq:G-relation-passing-to-subgroupoid-2-2} proves that $\rho(\Lambda)$ is a family $H$-relation. 
\end{enumerate}
\end{proof}

\section{Fourier integral operators on groupoids}\label{G-FIO}

\subsection{Definitions}
Following \cite{Horm-classics}, we are lead to 
\begin{defn}
 Let $G$ be a Lie groupoid. Distributions belonging to $I(G,\Lambda;\Omega^{1/2})$ where $\Lambda$ is any (family) local $G$-relation are called  (family) Fourier integral $G$-operators.  
\end{defn}
We abbreviate Fourier integral $G$-operators into $G$-FIO and family Fourier integral $G$-operators into $G$-FFIO.
If $\Lambda$ is a $G$-relation then it is by definition admissible and we get from \cite{LMV1}
\begin{equation}
 \label{eq:defn-FIO-consequence-1}
  I(G,\Lambda;\Omega^{1/2}) \subset \cD'_{r,s}(G,\Omega^{1/2}).
\end{equation}
In particular, any $G$-FIO $u$ produces an equivariant $C^\infty$ family of operators $u_x :C^\infty_c(G_x)\to C^\infty(G_x)$,  $x\in G^{(0)}$, but each $u_x$ is not necessarily a Fourier integral operator  on $G_x$. It is worth to give an example. 

\begin{exam}\label{exam:G-FIO-not-G-FFIO}
 Consider the fibred pair groupoid $G=X\times X\times Z\rightrightarrows X\times Z$ with $X=Z=\RR$. 
Consider the open cone 
\begin{equation}
 \cC=\{(\gamma,\theta)\in G\times \RR^2\setminus 0\ ;\  \theta\in\cC_{x_1}  \}
\end{equation}
 where $\gamma=(x_1,x_2,x_3)$ and $\theta\in\cC_{x_1}$ means  
\begin{equation}
  2x_1\theta_2+\theta_1\not = 0,\   \theta_1\not=0.
\end{equation}
The function
\begin{equation}
 \phi : (\gamma,\theta) \longmapsto (x_1-x_2).\theta_1+ (x^2_1-z).\theta_2
\end{equation}
is a non degenerated phase function with associated Lagrangian given by 
\begin{equation}\label{eq:G-rel-not-family-G-rel}
 \Lambda =\{ (x,x,x^2,\theta_1+2x\theta_2,-\theta_1,-\theta_2)\ ;\ x\in\RR,  \theta\in\cC_{x} \}\subset T^*G\setminus 0.
\end{equation}
$\Lambda$ is a $G$-relation, but fails to be a family $G$-relation at the points where $x=0$.
Consider the closed cone 
\begin{equation}
 F=\{ (\gamma,\theta)\in G\times \RR^2 \ ;\ \vert \gamma\vert\le 1,\  2\vert  \theta_2\vert\le \vert \theta_1\vert   \}\subset \cC\cup G\times \{0\}
\end{equation}
and choose even functions $\chi\in C^{\infty}_c(\RR)$ and $b\in C^{\infty}_c(\RR^3)$ such that $\chi(0)=\chi''(0)=1$,  $\chi(t)=0$ if $\vert t\vert \ge \frac{1}{2}$, $\supp{b}\subset \{\gamma,\ \vert \gamma\vert \le 1\}$ and $b(0)=1$. Let choose a symbol $a\in S^{1}(G\times \RR^2)$, with support in $F$, such that 
 $a(\gamma,\theta)= b(\gamma)\chi(\theta_2/\theta_1)\theta_1$ when $\vert \theta\vert \ge 1$. Then 
 \begin{equation}
  u(\gamma) =\int e^{i\phi(\gamma,\theta)} a(\gamma,\theta)d\theta \in I^*(G,\Lambda)
 \end{equation}
and we look at the distribution $u_0=m^*_{(0,0)}u$ on $G^{(0,0)}\times G_{(0,0)}\simeq \RR^2$. It is given by 
\begin{equation}\label{exam:FIO-not-FFIO-7}
  u_0(x_1,x_2) =\int e^{i((x_1-x_2).\theta_1+ x^2_1.\theta_2) } a_0(x_1,x_2,\theta_1,\theta_2)d\theta_1d\theta_2 \in \cD'(\RR^2)
 \end{equation}
understood as a distribution where $a_0(x_1,x_2,\theta_1,\theta_2)=a(x_1,x_2,0,\theta_1,\theta_2)$. Indeed, observe that  
\begin{equation}
 \phi_0 : (x,\theta)\longmapsto (x_1-x_2).\theta_1+ x^2_1.\theta_2
\end{equation}
is a phase function on $\cC_0=\{ (x,\theta)\in \RR^2\times\RR^2\ ;\ \theta\in \cC_{x_1}\} $ and that $a_0\in S^*(\RR^2\times\RR^2)$ is supported in 
\begin{equation}
 F_0 =\{(x,\theta)\ ; \  \vert x \vert\le 1,\  2\vert  \theta_2\vert\le \vert \theta_1\vert \}\subset \cC_0\times \RR^2\times\{0\}.
\end{equation}
It follows that \eqref{exam:FIO-not-FFIO-7} is an oscillatory integral \cite[Paragraph 7.8]{Horm-classics} and thus, by \cite[Theorem 8.1.9]{Horm-classics} 
\begin{equation}\label{exam:FIO-not-FFIO-10}
 \WF{u_0} \subset \Lambda_{0}=\{ (0,0,\theta_1,-\theta_1)\ ; \ \theta_1\not=0\}.
\end{equation}
As expected, $\Lambda_{0}$ fails to be a Lagrangian submanifold of $T^*\RR^2\setminus 0$ (actually, it is a one dimensional isotropic conic submanifold) and even more, there is no Lagrangian submanifold $\Lambda'$ of $T^*\RR^2\setminus 0$ such that $u_0\in I^*(\RR^2,\Lambda')$. Before proving this assertion, observe that \eqref{exam:FIO-not-FFIO-10} implies $u_0\in C^\infty(\RR^2\setminus 0)$ and that for any $x$ with $x_1\not=0$
\begin{eqnarray*}
u_0(x)&=&b_0(x)\int e^{i((x_1-x_2).\theta_1+ x^2_1.\theta_2) } \chi(\theta_2/\theta_1)\theta_1 d\theta_1 d\theta_2 \text{ modulo } C^\infty(\RR^2) \\
   &=& b_0(x)\int e^{i((x_1-x_2).\theta_1+ x^2_1.\theta_1\theta_2) } \theta_1^2\chi(\theta_2)d\theta_1 d\theta_2 
   = b_0(x)\int e^{i(x_1-x_2).\theta_1 } \widehat{\chi}(-x_1^2\theta_1)\theta_1^2 d\theta_1 \\
   &=& b_0(x)x_1^{-6}\int e^{i\frac{x_2-x_1}{x_1^2}.\theta_1 } \widehat{\chi}(\theta_1)\theta_1^2  d\theta_1 
   = x_1^{-6}b_0(x)\chi''(\frac{x_2-x_1}{x_1^2}).
\end{eqnarray*}
Thus, $u_0$ is not $C^\infty$ at $(0,0)$ and $\WF{u_0}$ contains at least a half line in $T^*_{(0,0)}\RR^2$. Since $u_0$ is even, $\WF{u_0}$ also contains the opposite half line. This proves the equality in \eqref{exam:FIO-not-FFIO-10}. Now assume that $u_0\in I^m(\RR^2,\Lambda')$ for some Lagrangian $\Lambda'$. If the principal symbol $\sigma(u_0)$ does not vanish at some point $(x_0,\xi_0)\in \Lambda'$, then $(x_0,\xi_0)\in \WF{u_0}$. Thus $\sigma(u_0)$ must vanish on $\Lambda'\setminus \Lambda_{0}$. Since $\Lambda_0$ is one dimensional, it has empty interior in $\Lambda'$ and it follows that 
 $\sigma(u_0)$ vanishes identically. Thus $u_0\in I^{m-1}(\RR^2,\Lambda')$ and repeating the argument proves that $u_0$ is $C^\infty$, which is a contradiction.
\end{exam}
The phenomenon enlighted in this example precisely disappears for Fourier integral $G$-operators associated with family $G$-relations. Indeed,  
\begin{thm}\label{thm:equiv-G-FIO-and-families} Let $\Lambda$ be a family $G$-relation and $u\in \cD'(G,\Omega^{1/2})$. Then 
 $u\in I(G,\Lambda;\Omega^{1/2})$ if and only if 
 $u$ is a $G$-operator and $u_x=m^*_x(u)\in I(G_x\times G^x,m_x^*\Lambda; \Omega^{1/2}_{G_x\times G^x})$ for all $x\in G^{(0)}$.
\end{thm}
\begin{proof}
Let us assume $u\in I(G,\Lambda;\Omega^{1/2})$. Then, as recalled before the theorem, $u$ is a $G$-operator and the pull-back distribution by the submersion $m$ gives 
\[
 m^*(u)\in I(G^{(2)},m^*\Lambda;m^*\Omega^{1/2})
\]
Since $m^*\Lambda$ is transversal to $\pi : G^{(2)}\to G^{(0)}$, Proposition \ref{prop:global-distri-to-family-submersion-case} gives the result for all the  $m^*_x(u)$, $x\in G^{(0)}$.

Conversely,  Proposition \ref{prop:global-distri-to-family-submersion-case} gives rise to distribution $\widetilde{u}\in  I(G^{(2)}, m^*\Lambda; m^*\Omega^{1/2})$ such that $\widetilde{u}\vert_{G_x\times G^x}= u_x$ and the result follows from the proposition \ref{lem:compo-LD-model-case} applied to $X=Z=G^{(2)}$, $Y=G$ and $f=m$, which yields $u=m_*\widetilde{u}\in I(G,\Lambda;\Omega^{1/2})$.
\end{proof}

\subsection{Adjoint and composition}
Now, we can consider $G$-FFIO equivalently as family of usual Fourier integral operators or as single Lagrangian distributions on $G$, whose underlying Lagrangian submanifold $\Lambda$ has suitable properties. The  second choice leads to simpler and more conceptual statements and also reveals the role played by the cotangent groupoid $T^*G$. Moreover, most of the statements hold true for the more general class of $G$-FIO. The next two theorems argue for this point of view. 

 \begin{thm}\label{thm:FIO-adjoint} Let $\Lambda$ be a $G$-relation and set  $\Lambda^\star=i_{\Gamma}\Lambda$. If $A\in I^m(G,\Lambda)$ then $A^\star\in I^m(G,  \Lambda^\star)$.
\end{thm}
\begin{proof}
It is sufficient to consider the case $A(\gamma)=\int e^{i\phi(\gamma,\theta)}a(\gamma,\theta)d\theta$ with $\phi$ a non degenerate phase function parametrizing locally $\Lambda$. Then 
\begin{equation}
A^\star(\gamma)=\int e^{-i\phi(\gamma^{-1},\theta)}\overline{a(\gamma^{-1},\theta)}d\theta.
\end{equation}
The function $b(\gamma,\xi)=\overline{a(\gamma^{-1},\theta)}$ is a symbol of the same order as $a$.  The function $\psi(\gamma,\theta)= -\phi(\gamma^{-1},\theta)$ is also a non degenerate phase function and 
\begin{equation}
\Lambda_\psi = \{ (\gamma,\xi)\in T^*G\ ;\ (\gamma^{-1},-{}^t(di_\gamma)(\xi))\in \Lambda_\phi\}.
\end{equation}
Since $i_\Gamma(\gamma,\xi)=  (\gamma^{-1},-{}^t(di_\gamma)(\xi))$, we get the result. 
\end{proof}
Note that if $\Lambda$ is moreover a family, then $\Lambda^\star$ too, and the adjoint of a $G$-FFIO $u\in I(G,\Lambda)$ is given by the family of adjoints of each Fourier integral operator $u_x$ on $G_x$.
\begin{thm}
 \label{thm:compo-FIO-on-gpd}
Let  $\Lambda_1,\Lambda_2$ be closed $G$-relations which are cleanly convolable with excess $e$. 
If $A_1\in I^{m_1}_c(G,\Lambda_1)$ and $A_2\in I^{m_2}_c(G,\Lambda_2)$ then
\begin{equation}
 \label{eq:comp-FIO-on-gpd}
 A_1.A_2 \in I^{m_1+m_2+e/2+n^{(0)}/2-n/4}(G,\Lambda_1.\Lambda_2).
\end{equation}
Here $n$ is the dimension of $G$ and $n^{(0)}$ is the dimension of $G^{(0)}$. \\
If moreover, $\Lambda_1,\Lambda_2$ are families and completely convolable (i.e. condition \eqref{eq:unwanted-condition} is fulfilled), then $ A_1.A_2$ is a family Fourier integral $G$-operator.
\end{thm}

\ignore{
\begin{prop}\label{prop:push-forward-symbols} We use the notations and hypotheses of Proposition \ref{prop:model-case-of-convol-phases}. 
Let $\widetilde{\Sigma}$ be a line bundle over $T^*X$ such that, canonically, $\widetilde{\Sigma}\vert_H\simeq \Omega^{1/2}(\ker d\mu)$. Then for any  $ a\in  S^{m}(\widetilde{\Lambda}, I_{\widetilde{\Lambda}}\otimes \widetilde{\Sigma})$, the integral 
\begin{equation}\label{eq:push-forward-symbols-1}
 (y,\eta)\in \Lambda,\quad \mu_\#(a)(y,\eta) = 
\int_{\mu^{-1}(y,\eta)\cap \widetilde{\Lambda}} a(x,\xi) \in I_\Lambda(y,\eta)
\end{equation}
is  unambiguously defined and induces a linear map
\begin{equation}\label{eq:push-forward-symbols-2}
  \mu_\# : S^{[m]}(\widetilde{\Lambda}, I_{\widetilde{\Lambda}}\otimes \widetilde{\Sigma}) \longrightarrow S^{[m+e/2+(n_Y-n_X)/2]}(\Lambda, I_{\Lambda}).
\end{equation}
\end{prop}
\begin{proof}
For any $(x,\xi)\in\widetilde{\Lambda}\cap H$, we have by assumption $a(x,\xi)\in  \lbrack I_{\widetilde{\Lambda}}\otimes \Omega^{1/2}(\ker d\mu)\rbrack_{(x,\xi)}$. Applying the Hörmander map $\mu_\#$ of \cite[Theorem 21.6.6]{Horm-classics}, we get 
\begin{equation}
\forall (x,\xi)\in \widetilde{\Lambda}\cap H \text{ and } (y,\eta)=\mu(x,\xi),\quad a(x,\xi)\in  (I_\Lambda)_{(y,\eta)}\otimes \Omega^{1/2}(\mu^{-1}(y,\eta)\cap  \widetilde{\Lambda})_{(x,\xi)}
\end{equation}
and thus the integral in \eqref{eq:push-forward-symbols-1} is well defined and $\mu_\#(a)\in C^\infty(\Lambda,I_\Lambda)$. It remains to prove that it is a symbol of the right order. 
\end{proof}
}

\begin{proof}[Proof of theorem \ref{thm:compo-FIO-on-gpd}]

We wish to apply Lemma \ref{lem:compo-LD-model-case} to the following data:  $X=G^2$, $Y=G$, $Z=G^{(2)}$, $f=m_G$ $\widetilde{\Lambda}=\Lambda_1\times \Lambda_2$ and $\widetilde{\phi}=\phi_1+\phi_2$ where $\phi_j : U_j\times(\RR^{N_j}\setminus 0)\to\RR$ are non degenerated phase functions  parametrizing   $\Lambda_j$ in a conic neighborhood of points $(\gamma_j,\xi_j)\in\Lambda_j$,  the latter points satisfying $(\gamma_1,\xi_1,\gamma_2,\xi_2)\in \Lambda_1\times \Lambda_2\cap (T^*G)^{(2)}$. We may assume that
\begin{equation}
 A_j(\gamma_j)=\int e^{i\phi_j(\gamma_j,\theta_j)}a_j(\gamma_j,\theta_j)d\theta_j,
\end{equation}
where $a_j\in S^{m_j+(n-2N_j)/4}(U_j\times(\RR^{N_j})$. 

The only technical (and usual) obstruction is that 
\begin{equation}
 a(\gamma_1,\gamma_2,\theta_1,\theta_2)=a_1(\gamma_1,\theta_1)a_2(\gamma_2,\theta_2)
\end{equation}
is not a symbol in general. The conditions of admissibility on $\Lambda_j$ allows to remove the regions in $(\theta_1,\theta_2)$ where the symbolic estimates for $a$ fail. 

Indeed, thanks to the admissibility assumptions on $\Lambda_1$ and $\Lambda_2$, we can reduce the problem to the case where $a_1,a_2$ have support in  compactly generated cones $\cC_1,\cC_2$ on which $\widetilde{s}(\phi'_{1\gamma})$ and $\widetilde{r}(\phi'_{2\gamma})$ never vanish. 
Argumenting on the degree one homogeneity of $\widetilde{s}(\phi'_{1\gamma})$, $\widetilde{r}(\phi'_{2\gamma})$ with respect to $\theta_1,\theta_2$, we can find constants $C_1,C_2$ such that 
\begin{equation}
 \text{ if } (\gamma_j,\theta_j)\in\cC_j \text{ and } \widetilde{s}(\phi'_{1\gamma}(\gamma_1,\theta_1))= 
\widetilde{r}(\phi'_{2\gamma}(\gamma_2,\theta_2))  \text{ then } C_1|\theta_2|<|\theta_1|<C_2|\theta_2|.
\end{equation}
We choose a homogeneous function $\chi(\theta_1,\theta_2)$ of degree $0$ equal to $1$ when $ C_1|\theta_2|/2<|\theta_1|<2C_2|\theta_2|$ and supported in $C_1|\theta_2|/3<|\theta_1|<3C_2|\theta_2|$. We set
\begin{equation}
  b(\gamma_1,\gamma_2,\theta_1,\theta_2)=\chi(\theta_1,\theta_2)a(\gamma_1,\gamma_2,\theta_1,\theta_2)
\end{equation}
and
\begin{equation}
  r(\gamma_1,\gamma_2,\theta_1,\theta_2)=(1-\chi(\theta_1,\theta_2))a(\gamma_1,\gamma_2,\theta_1,\theta_2).
\end{equation}
We have by construction of $\chi$,
\begin{equation}
 C_1|\theta_2|/3<|\theta_1|<3C_2|\theta_2| \text{ in } \supp b. 
\end{equation}
which allows to check that $b\in S^{m_1+m_2+(n-N_1-N_2)/2}$. Therfore we can apply the lemma to 
\[
\widetilde{B}=\int e^{i\widetilde{\phi}(\gamma_1,\gamma_2,\theta_1,\theta_2)}b(\gamma_1,\gamma_2,\theta_1,\theta_2)d\theta_1 d\theta_2,
\] 
and we get 
\begin{equation}
B(\gamma) = \int_{m^{-1}(\gamma)\times \RR^{N_1+N_2}} e^{i\phi(\gamma,\eta,\theta_1,\theta_2)}b(\gamma,\eta,\theta_1,\theta_2)d\eta d\theta_1 d\theta_2 \in I^{m_1+m_2+e/2+n^{(0)}/2-n/4}(G,\Lambda_1*\Lambda_2).
\end{equation}
Moreover, arguing again on the degree one homogeneity of $\widetilde{\phi}$ with respect to $(\theta_1,\theta_2)$ and using the expression of $\phi'_\omega$ given in  the proof of Proposition \ref{prop:model-case-of-convol-phases}, we also get 
\begin{equation}
 |\theta_1|+|\theta_2| <  C |\phi'_\eta(\gamma,\eta,\theta_1,\theta_2)|  \text{ in } \supp r, 
\end{equation}
where we have set $\gamma=\gamma_1\gamma_2$, $\eta=(\gamma_1,\gamma_2)\in m^{-1}(\gamma)$ and $\phi(\gamma,\eta,\theta_1,\theta_2)=\widetilde{\phi}(\gamma_1,\gamma_2,\theta_1,\theta_2)$. 
The previous estimates shows that 
\begin{equation}
 R(\gamma)=\int_{m^{-1}(\gamma)\times \RR^{N_1+N_2}} e^{i\phi(\gamma,\eta,\theta_1,\theta_2)}r(\gamma,\eta,\theta_1,\theta_2)d\eta d\theta_1 d\theta_2
 \end{equation}
belongs to $C^\infty(G)$ and we conclude that 
\begin{equation}
A_1*A_2(\gamma) = B(\gamma) \mod C^\infty(G),
\end{equation}
which proves the theorem. 

\end{proof}

\ignore{\begin{prop}
This induces a  map $\mu_\#$  between principal symbols 
 \begin{equation}\label{eq:push-forward-symbols-2}
  \mu_\# : S^{[m+n_X/4]}(\widetilde{\Lambda}, I_{\widetilde{\Lambda}}\otimes \widetilde{\Sigma}) \longrightarrow S^{[m+e/2+(n_X-2n_Z+2n_Y)/4] }(\Lambda, I_{\Lambda})
\end{equation}
 which sends  $[a]\in  S^{[*]}(\widetilde{\Lambda}, I_{\widetilde{\Lambda}}\otimes \widetilde{\Sigma})$ to the class of 
 \begin{equation}\label{eq:push-forward-symbols-1}
 (y,\eta)\in \Lambda,\quad \mu_\#(a)(y,\eta) = 
\int_{\mu^{-1}(y,\eta)\cap \widetilde{\Lambda}} a(x,\xi) \in I_\Lambda(y,\eta)
\end{equation}
\end{prop}}

\subsection{Principal symbol}
By \cite[Section 25.1]{Horm-classics}, the principal symbol of  $A\in I^m(G,\Lambda;\Omega^{1/2})$ belongs to $S^{[m+n/4]}(\Lambda,I_\Lambda\otimes\hat{\Omega}^{1/2}\otimes\hat{\Omega}^{-1/2}_G)$ and the principal symbol map gives rise to an isomorphism
\begin{equation}\label{eq:symbol-FIO-gpd}
 \sigma : I^{[m]}(G,\Lambda;\Omega^{1/2}) \longrightarrow S^{[m+n/4]}(\Lambda,I_\Lambda\otimes\hat{\Omega}^{1/2}\otimes\hat{\Omega}^{-1/2}_G).
\end{equation}
Here we have set $\hat{E}=(p\vert_\Lambda)^*(E\vert_\Lambda)$ for any bundle $E\to G$. 
To  understand the product formula of symbols of $G$-FIO, we analyse the auxiliary bundle 
\begin{equation}\label{eq:Sigma-bundle}
\Sigma^{\alpha} = \Omega^{-\alpha}_G\otimes \Omega^{\alpha}
\end{equation}
involved in the right hand side of \eqref{eq:symbol-FIO-gpd} with $\alpha=\frac12$. As we shall see, $\hat{\Sigma}$ is strongly related to the groupoid structure of $T^*G$.

We need a simple statement about vector bundles epimorphisms.
\begin{lem}\label{lem:vb-epimorphisms-and-kernel-exact-seq}
 Let $f:X\to Y$ be a submersion, $p:E\to X, q:F\to Y$ be $C^\infty$ vector bundles and $g : E\to F$ an $C^\infty$ epimorphism:
\begin{equation}\label{diag:vb-epimorphisms-and-kernel-exact-seq}
 \xymatrix{
  0 \ar[r] &  \ker g\ar[r]\ar^{p}[d]  & E \ar^{g}[r]\ar^{p}[d]  &  F \ar[r]\ar^{q}[d]  & 0 \\
     & X \ar^{=}[r] &  X \ar^{f}[r]&  Y  & 
}
\end{equation}
Then the sequence 
\begin{equation}\label{eq:vb-epimorphisms-and-kernel-exact-seq}
 0\longrightarrow  p^*(\ker g) \overset{v}{\longrightarrow} \ker dg \overset{dp}{\longrightarrow}p^*\ker df\longrightarrow 0.
\end{equation}
is exact. The map $v$ is defined by 
\[
  p^*(\ker g)_{(x,e)}\ni \lambda \longmapsto v(\lambda)=\frac{d}{dt}(x,e+t\lambda)\vert_{t=0}\in \ker dg_{(x,e)}\cap \ker dp_{(x,e)}.
\]

\end{lem}
\begin{proof} 
Thanks to the diagram \eqref{diag:vb-epimorphisms-and-kernel-exact-seq}, we have $dp(\ker g)\subset \ker df$ and the map 
\begin{equation}\label{proof:vb-epimorphisms-and-kernel-exact-seq-1}
\ker dg \ni (e,u) \overset{dp}{\longmapsto} (e,dp_e(u)) \in p^*\ker df
\end{equation}
is well defined. To prove its surjectivity, we work in local coordinates on $TE$ associated with local coordinates on $X$ and local trivializations of $E$,  so that the map \eqref{proof:vb-epimorphisms-and-kernel-exact-seq-1} corresponds to 
\begin{equation}\label{proof:vb-epimorphisms-and-kernel-exact-seq-2}
 (x,e, t,u) \longmapsto  (x,e, t).
\end{equation}
Writing $g(x,e)=(f(x),\widetilde{g}(x,e))$, we compute
\begin{align}\label{proof:vb-epimorphisms-and-kernel-exact-seq-3}
(dg)_{(x,e)}(t,u) & = (df_x(t), (d_x\widetilde{g})_{(x,e)}(t) + (d_e\widetilde{g})_ {(x,e)}(u) )\\
 & = (df_x(t), (d_x\widetilde{g})_{(x,e)}(t) + \widetilde{g}(x,u) ). 
\end{align}
Since $\widetilde{g}$ is fiberwise linear and surjective, the linear equation $(d_x\widetilde{g})_{(x,e)}(t) + \widetilde{g}(x,u) =0$ for fixed $x,e$ and $t$ has solutions in $u$. Let $u_{x,e,t}$ be such a solution. Then, for any $t\in\ker df_x$, the element  $(t,u_{x,e,t})$ belongs to $\ker dg_{(x,e)}$. 

Next, it follows from \eqref{proof:vb-epimorphisms-and-kernel-exact-seq-1} and  \eqref{proof:vb-epimorphisms-and-kernel-exact-seq-3} that $(x,e, t,u)\in \ker dp\cap \ker dg $ if and only if $t=0$ and $u\in \ker\widetilde{g}$. Since in these coordinates
\begin{equation}\label{proof:kernel-ds-Gamma-exact-seq-4}
  v: p^*(\ker g)\ni (x,e,u)\longmapsto (x,e,0,u)\in \ker dp\cap \ker dg
\end{equation}
we get that $v(p^*(\ker g))$ is the kernel of the vector bundle  epimorphism 
$\ker dg \overset{d p}{\longrightarrow}p^*\ker df$.
\end{proof}
As announced, we can interpret $\Sigma$ in terms of densities bundles associated with the groupoid structure of $T^*G$. 
\begin{prop}\label{cor:basic-identities-densities-on-Gamma} We have canonically
\begin{align}
\hat{\Sigma}^{1/2}& \simeq \Omega^{1/2}(\ker ds_\Gamma)   \simeq \Omega^{1/2}(\ker dr_\Gamma) \label{cor:basic-identities-densities-on-Gamma:line1} \\ 
 \Omega(\ker dm_\Gamma) & \simeq \pr{(1)}^*\hat{\Sigma}^{1/2} \otimes \pr{(2)}^*\hat{\Sigma}^{1/2} \simeq m_\Gamma^*\hat{\Sigma} \label{cor:basic-identities-densities-on-Gamma:line2}\\
  \Omega(\ker dm_\Gamma) & \simeq (p^{2})^*(\Omega(\ker m_\Gamma)\otimes \Omega(\ker dm)). \label{cor:basic-identities-densities-on-Gamma:line3}
\end{align}
\end{prop}
 \begin{proof}
Applying the lemma to \eqref{diag:cot-gpd-source-overview}, \eqref{diag:cot-gpd-target-overview} and \eqref{diag:cot-gpd-mult-overview}, one gets the exact sequence
\begin{equation}\label{eq:kernel-ds-Gamma-exact-seq}
 0\longrightarrow p^*((\ker dr)^\perp) \longrightarrow \ker ds_\Gamma \overset{dp }{\longrightarrow}p^*\ker ds \longrightarrow 0,
\end{equation}
\begin{equation}\label{eq:kernel-dr-Gamma-exact-seq}
 0\longrightarrow  p^*((\ker ds)^\perp) \longrightarrow \ker dr_\Gamma \overset{dp}{\longrightarrow}p^*\ker dr\longrightarrow 0.
\end{equation}
and
\begin{equation}\label{eq:kernel-dm-Gamma-exact-seq}
 0\longrightarrow (p^{2})^*\ker m_\Gamma \longrightarrow \ker dm_\Gamma \overset{dp^{2}}{\longrightarrow}(p^{2})^*\ker dm\longrightarrow 0.
\end{equation}
 To prove (\ref{cor:basic-identities-densities-on-Gamma:line1}), observe that \eqref{eq:kernel-ds-Gamma-exact-seq} gives 
\[
 \Omega^{1/2}(\ker ds_\Gamma)\simeq p^*\Omega^{1/2}(\ker ds)\otimes p^*\Omega^{1/2}((\ker dr)^\perp)
\]
and that $\Omega^{1/2}((\ker dr)^\perp)=\Omega^{-1/2}(TG/\ker dr)=\Omega^{-1/2}_G\otimes \Omega^{1/2}(\ker dr)$. Next, observe that for any Lie groupoid $G$, the maps
\[
   \ker dm \ni (\gamma_1,\gamma_2,X_1,X_2)\longmapsto (\gamma_1,\gamma_2,X_2)\in \pr{(2)}^*(\ker ds)
\]
and 
\[
 \pr{(1)}^*(\ker ds) \ni (\gamma_1,\gamma_2,X_1)\longmapsto (\gamma_1,\gamma_2,(dR_{\gamma_2})_{\gamma_1}(X_1))\in m^*(\ker ds)
\]
are  isomorphisms of vector bundles over $G^{(2)}$.  Similarly, $\ker dm\simeq \pr{(1)}^*(\ker dr) $ and $\pr{(2)}^*(\ker dr)\simeq m^*(\ker dr)  $. These information used for the groupoid $T^*G$ give
\begin{align}
 \Omega(\ker dm_\Gamma) &\simeq  \Omega^{1/2}(\ker dm_\Gamma)\otimes \Omega^{1/2}(\ker dm_\Gamma) \\
 & \simeq  \pr{(1)}^*\Omega^{1/2}(\ker dr_\Gamma)\otimes \pr{(2)}^*\Omega^{1/2}(\ker ds_\Gamma)\\
 & \simeq  m_\Gamma^*(\Omega(\ker ds_\Gamma))\simeq  m_\Gamma^*(\Omega(\ker dr_\Gamma))
\end{align}
where we have used (\ref{cor:basic-identities-densities-on-Gamma:line1}) to pass from the second to the third line.
 This  proves (\ref{cor:basic-identities-densities-on-Gamma:line2}) and then (\ref{cor:basic-identities-densities-on-Gamma:line3}) follows directly from \eqref{eq:kernel-dm-Gamma-exact-seq}.
\end{proof}

\begin{prop}\label{prop:prod-densities-n-compo-lag} Let  $\Lambda_1,\Lambda_2$ be closed $G$-relations which are cleanly convolable. Let $\Lambda=\Lambda_1.\Lambda_2$. We have a natural homomorphism of vector bundles over $\Lambda_1\times\Lambda_2\cap \Gamma^{(2)}$:
\begin{equation}
 \label{eq:prod-densities-n-compo-lag}
  (\hat{\Sigma}^{1/2}\otimes I_{\Lambda_1})\boxtimes (\hat{\Sigma}^{1/2}\otimes I_{\Lambda_2})\longrightarrow  m_\Gamma^*(I_{\Lambda}\otimes \hat{\Sigma}^{1/2})  \otimes \Omega(\ker dm_\Gamma\cap T(\Lambda_1\times\Lambda_2)).
\end{equation}
\end{prop}
\begin{proof}
Applying \cite[Theorem 21.6.6]{Horm-classics}, we get
\begin{equation}
 \label{eq:ident-Maslov-times-densities-n-compo-lag}
  I_{\Lambda_1}\boxtimes I_{\Lambda_2}\longrightarrow m_\Gamma^*I_{\Lambda}\otimes \Omega^{-1/2}(\ker dm_\Gamma)\otimes \Omega(\ker dm_\Gamma\cap T(\Lambda_1\times\Lambda_2))
\end{equation}
Contrary to what happens in the proof of   of \cite[Theorem 21.6.7]{Horm-classics}, the bundle  $\Delta= \ker dm_\Gamma$ is not necessarily symplectic (actually, it may even be odd dimensional since the fibers are of dimension $n=\dim G$) and we cannot expect any natural trivialization of the corresponding density bundle.  This is where the bundle $\Sigma$ is useful. 

Using \eqref{cor:basic-identities-densities-on-Gamma:line2} in Corollary \ref{cor:basic-identities-densities-on-Gamma} we get 
\begin{equation}
 \label{proof:prod-densities-n-compo-lag-3}
  (\hat{\Sigma}^{1/2}\otimes I_{\Lambda_1})\boxtimes (\hat{\Sigma}^{1/2}\otimes I_{\Lambda_2}) \simeq \Omega(\ker dm_\Gamma) \otimes (I_{\Lambda_1}\boxtimes  I_{\Lambda_2})\simeq m_\Gamma^*(\hat{\Sigma}) \otimes (I_{\Lambda_1}\boxtimes  I_{\Lambda_2}).
\end{equation}
Using  \eqref{cor:basic-identities-densities-on-Gamma:line2} again to get $\Omega(\ker dm_\Gamma)^{1/2} \simeq m_\Gamma^*(\hat{\Sigma}^{1/2})$ and combining \eqref{proof:prod-densities-n-compo-lag-3} and \eqref{eq:ident-Maslov-times-densities-n-compo-lag}, we obtain \eqref{eq:prod-densities-n-compo-lag}.

\end{proof}

These identifications of Maslov and densities related  bundles allow to apply the formula for the product of principal symbols given in \cite[Theorem 25.2.3]{Horm-classics}. In the present situation, it gives

\begin{cor}
 \label{cor:compo-symbole-without-Maslov}
Let  $\Lambda_1,\Lambda_2$ be closed $G$-relations which are cleanly convolable with excess $e$ and set $\Lambda=\Lambda_1.\Lambda_2$.   Let $A_j\in I^{m_j}(G,\Lambda_j;\Omega^{1/2})$ be compactly supported $G$-FIO and $a_j\in S^{m_j+n/4}(\Lambda_j,\hat{\Sigma}^{1/2}\otimes I_{\Lambda_j})$  be representants of the principal symbol of $A_j$.  Let $(a_1\boxtimes a_2)_{\gamma,\xi}$ be the density  on the compact manifold $m^{-1}_\Gamma(\gamma,\xi)\cap \Lambda_1\times\Lambda_2$ with values in $ \hat{\Sigma}^{1/2}\otimes I_{\Lambda}$ as given by   \eqref{eq:prod-densities-n-compo-lag}. Then $a_1*a_2$ defined by 
\begin{equation}
 (\gamma,\xi)\in \Lambda,\quad  a_1*a_2(\gamma,\xi) = 
\int_{m^{-1}_\Gamma(\gamma,\xi)\cap \Lambda_1\times\Lambda_2}  a_1\boxtimes a_2
\end{equation}
belongs  $S^{m_1+m_2+e/2+n^{(0)}/2}(\Lambda,\hat{\Sigma}^{1/2}\otimes I_{\Lambda})$ and represents the principal symbol of $A=A_1*A_2$. 
\ignore{In other words, $(a_1,a_2)\to a_1*a_2$ induces the map
\begin{equation}
  S^{[m_1]}(\Lambda_1,\hat{\Sigma}^{1/2}\otimes I_{\Lambda_1})\times S^{[m_2]}(\Lambda_2,\hat{\Sigma}^{1/2}\otimes I_{\Lambda_2})\longrightarrow S^{[m_1+m_2+e/2-n/4]}(\Lambda,\hat{\Sigma}^{1/2}\otimes I_{\Lambda}).
\end{equation}}
\end{cor}

\ignore{
\begin{proof}
This needs to checked for symbols supported in small open cones on which parametrizations of the lagrangian exist. Thus, let $\phi_j : U_j\times \RR^{N_j}\to \RR$ be non-degenerate phase functions parametrizing locally $\Lambda_j$ and $a_j\in S^{m_j}(\Lambda_j,\hat{\Sigma}^{1/2}\otimes I_{\Lambda_j}) $ be symbols with support in a compactly generated cone included in $\Lambda_{\phi_j}$. Let us fix local charts $(U_j,\chi_j)$ for $G$ and denote by $Q_{j,\gamma,\theta}$ the non-degenerate quadratic form on $T_\gamma G\times \RR^{N_j}$ associated with the Hessian of $\phi_j$ in these local coordinates. Then we have a local section (actually, a local basis) of the bundle $I_{\Lambda_j}$ over the open subset $\Lambda_{\phi_j}\subset \Lambda_j$:
\begin{equation}
 \iota_j : \Lambda_{\phi_j} \ni(\gamma,\xi)\longmapsto \int e^{iQ_{j,\gamma,\theta}(t,\tau)}d\tau \vert dt\vert^{1/2}\in I_{\Lambda_{j}}(\gamma,\xi)
\end{equation}
where $\xi=\phi_{j,\gamma}'(\gamma,\theta)$. Since $\iota_j$ has degree $N_j/2$ in $\xi$, we can write, after the trivialization of $\Sigma$:
\begin{equation}
 a_j(x,\xi)= b_j(x,\xi)\iota_j(x,\xi) \quad \text{ with } b_j\in S^{m_j-N_j/2}(\Lambda).
\end{equation}
\end{proof}}
We end with some direct consequences of the previous statements. 

\subsection{Composition with pseudodifferential operators} 
\begin{thm}\label{thm:pseudodiff-bimodule}
Any closed $G$-relation $\Lambda$ is transversally convolable with the unit $G$-relation $A^*G$ and the convolution product of distributions turns  $ I(G,\Lambda;\Omega^{1/2})$  into a $\Psi_c(G,\Omega^{1/2})$-bimodule : 
\[
  \Psi_c(G;\Omega^{1/2})* I(G,\Lambda;\Omega^{1/2})\subset I(G,\Lambda;\Omega^{1/2})  \ ; \  I(G,\Lambda;\Omega^{1/2})*\Psi_c(G;\Omega^{1/2}) \subset I(G,\Lambda;\Omega^{1/2}). 
\]
\end{thm}
 When $\Lambda= A^*G$,  we recover the fact that $\Psi_c(G)$ is an algebra. 
\begin{proof}
$\Lambda_0=A^*G$ and $\Lambda$ transversally convolable means that $T(\Lambda_0\times \Lambda)+ T\Gamma^{(2)}=T\Gamma^2$ at any point $(\delta_1,\delta_2)\in \Lambda_0\times \Lambda\cap\Gamma^{(2)}$. Passing to the symplectic othogonal, this is equivalent to
\begin{equation}
  T_{(\delta_1,\delta_2)}(\Lambda_0\times \Lambda)\cap \ker (dm_\Gamma)_{(\delta_1,\delta_2)} = 0
\end{equation}
and the latter follow from general properties of Lie groupoids. Indeed, Let $\Gamma$ be any Lie groupoid and consider $\gamma\in \Gamma$, $r(\gamma)=x,s(\gamma)=y\in \Gamma^{(0)} $, $(t_1,t_2)\in T_{x,\gamma}\Gamma^{(0)}\times \Gamma \cap \ker (dm_\Gamma)_{(x,\gamma)} $.

Since $r_{\Gamma}\vert_{\Gamma^{(0)}}=s_{\Gamma}\vert_{\Gamma^{(0)} }=\mathrm{Id}$, we get $ds_{\Gamma}(t_1)= dr_{\Gamma}(t_1)=t_1$ and since $r_{\Gamma}\circ m_\Gamma=r_{\Gamma}\circ \pr{1}$, $s_{\Gamma}\circ m_\Gamma=s_{\Gamma}\circ \pr{2}$ we get from the assumption on $(t_1,t_2)$ that 
\[
 t_1 = dr_{\Gamma}(t_1)= dr_{\Gamma}\circ dm_\Gamma(t_1,t_2)=0 \text{ and } ds_\Gamma(t_2)=ds_{\Gamma}\circ dm_\Gamma(t_1,t_2)=0.
\]
Also, we get $0=ds_\Gamma(t_1)=dr_\Gamma(t_2)$, therefore 
\[
 (0,t_2)\in \ker (dm_\Gamma)_{(x,\gamma)} ,\quad t_2\in T_\gamma \Gamma^{x}_{y}.
\]
Then 
\[
 (0,d(R_{\gamma^{-1}})_\gamma(t_2))\in \ker (dm_\Gamma)_{(x,x)} ,\quad d(R_{\gamma^{-1}})_\gamma(t_2)\in T_x \Gamma^{x}_{x}.
\]
Since $(dm_\Gamma)_{(x,x)}(u_1,u_2)=u_1+u_2$ if $u_j\in T_x \Gamma^{x}_{x}$ we also conclude $t_2=0$ and this proves that $\Lambda_0$ and $\Lambda$ transversally convolable. This is obviously the same with $\Lambda_0$ on the right. In both cases the fibers of the convolution $\Lambda_0.\Lambda=\Lambda.\Lambda_0=\Lambda$ are just points, hence the convolution is proper and connected. Now Theorem \ref{thm:compo-FIO-on-gpd} gives the conclusion. 
\end{proof}

Combining Theorems \ref{thm:pseudodiff-bimodule} and \ref{thm:compo-FIO-on-gpd}, we obtain
\begin{thm}\label{thm:Egorov}({\sl Egorov theorem for groupoids}).
 Let   $\Lambda,\Lambda'\subset T^*G\setminus 0$  be convolable closed $G$-relations such that 
\begin{equation}\label{defn:condition-weak-invertibility-Lagrangian}
  \Lambda. \Lambda'\subset A^*G\setminus 0 \text{ and } \Lambda' . \Lambda \subset A^*G\setminus 0. 
 \end{equation}
Then 
 \begin{equation}
  \label{thm:Egorov-as-sub-bi-module}
   I_c(G,\Lambda;\Omega^{1/2})*\Psi(G;\Omega^{1/2})* I_c(G,\Lambda';\Omega^{1/2})\subset \Psi(G;\Omega^{1/2}).
 \end{equation}
\end{thm}


We recall that the usual conventions for the order of conormal and Lagrangian distributions and for the order of pseudodifferential operators on groupoids yield
\begin{equation}
 \Psi^{m}(G) = I^{m+(n-2n^{(0)})/4}(G,A^*G;\Omega^{1/2}). 
\end{equation}
Let us also recall that in \cite{Vassout2006} were introduced a class of generalized smoothing operators and of Sobolev spaces. Recall that $C^*(G)$ denotes the C$^*$-algebra associated with the groupoid $G$.
The set of generalized smoothing operators $\Psi^{-\infty}(G) $ (Definition 24 p77 in \cite{Vassout2006}) is defined as the subset of $C^*(G)$ of those elements $R$ such that the closure of $P_1RP_2$ is again in $C^*(G)$ for any two compactly supported pseudodifferential $G$ operators $P_1$ and $P_2$.
For $s>0$ the Sobolev module $H^s$ of rank $s$ is defined as the $C*(G)$-module dom($\overline{P}$) endowed with scalar product $\scal{x}{y}_s=\scal{Px}{Py}+\scal{x}{y}$ where $P$ is any elliptic operator of order $s$. The corresponding Sobolev module $H^{-s}$ is defined by duality.\\ Next we can give, using techniques coming from \cite{MP} and \cite{Vassout2006}, a result on the continuity in the spirit of Theorem 25.3.1 in \cite{Horm-classics}
\begin{thm}\label{thm:continuity_C*}
 Let $\Lambda$ be an locally invertible local $G$-relation and $A\in I_c^{m}(G,\Lambda;\Omega^{1/2})$. 
 \begin{enumerate}
 \item If $m=(n-2n^{(0)})/4$, then the associated Fourier integral $G$-operator still denoted by $A$ extends into an operator which is a bounded multiplier of $C^*(G)$ :
 \begin{equation}
   A\in \cM(C^*(G)).
 \end{equation}
 \item If $m<(n-2n^{(0)})/4$ then $A$ extends into an element of $C^*(G)$.
 \item In the general case, $A$ can be extended to a morphism from $H^s$ to $H^{s-m'}$ with $$m'=m-(n-2n^{(0)})/4.$$
 \end{enumerate}
\end{thm}
\begin{proof}
By definition of $G$-FIO and the assumptions of the theorem, $A$ can be decomposed into a finite sum $A=\sum A_i$ where for all $i$,  $A_i\in I_c^{m}(G,\Lambda_i;\Omega^{1/2})$ and $\Lambda_i$ is an invertible patch of $\Lambda$. Therefore, we can directly assume that $\Lambda$ is an invertible $G$-relation. 

From \cite{LMV1},  we know that $A$ is an adjointable $G$-operator and Theorem \ref{thm:FIO-adjoint} gives $A^*\in I_c^{m}(G,\Lambda^\star;\Omega^{1/2})$.  Remember that,  $ \scal{ }{ }$ denoting the Hilbertian product of $C^*(G)$ seen as a $C^*(G)$-Hilbert module, the adjoint $A^*$ is characterized by 
\begin{equation}\label{proof:continuity-C*-2}
  \scal{Au}{v}=\scal{u}{A^*v} \quad \forall u,v \in C^\infty_c(G,\Omega^{1/2}).
\end{equation}
Observe that, by the Cauchy-Schwarz inequality for Hilbert modules, we have 
\begin{equation}\label{proof:continuity-C*-3}
  \| Au\|^2 = \| \scal{Au}{Au}\|=\| \scal{A^*Au}{u}\| \le \| A^*A u\|\| u\| \quad \forall u \in C^\infty_c(G,\Omega^{1/2}),
\end{equation}
and similarly for $A^*$. Since $\Lambda$ is invertible, Theorems \ref{thm:compo-FIO-on-gpd} and \ref{thm:inversibility-equiv-transversal-adjoint} give 
\[
 A^*A \in \Psi^{2m-\frac{n-2n^{(0)}}{2}}_c(G).
\]
Now a fundamental result \cite[Theorem 18]{MP}, \cite[Proposition 39]{Vassout2006} is 
\begin{equation}
 \label{proof:continuity-C*-5}
 \Psi^{0}_c(G)\subset \cM(C^*(G)) \text{ and } \Psi^{m'}_c(G)\subset C^*(G) \text{ for any } m'<0
\end{equation}
and then we can proceed as in  \cite{MP},\cite{Vassout2006}:
\begin{enumerate}
\item Let assume $m=(n-2n^{(0)})/4$. Then  $A^*A\in\Psi^{0}_c(G) $ and there exists $C\ge 0$ such that, using \eqref{proof:continuity-C*-3}:
\[
 \| Au\|^2 \le C\| u\|^2 \quad \forall u \in C^\infty_c(G,\Omega^{1/2}),
\]
and similarly for $A^*$. This allows to extend by continuity the relation \eqref{proof:continuity-C*-2} to all $u,v\in C^*(G)$, which then proves that $A\in\cL(C^*(G))\simeq \cM(C^*(G))$.

\item Let assume $m<(n-2n^{(0)})/4$. Then $A\in \cM(C^*(G))$ as before and by \eqref{proof:continuity-C*-5}, $A^*A\in C^*(G)$, which implies that  $A\in C^*(G)$ too since  $C^*(G)$ is an ideal of $\cM(C^*(G))$ (see \cite[Chapter 1]{Ped1979} for instance).  
\item In the general case, we know from \cite{Vassout2006} that for all $s$ there exists an invertible elliptic pseudodifferential $G$-operator $P(s)$ of order exactly $s$ , of inverse $P(-s)$, and that this operator $P(s)$ is an isomorphism of Hilbert modules between $H^s$ and $C^*(G)$. \\ Then $P(s-m)AP(-s)$ is an element of $I_c^{(n-2n^{(0)})/4}(G,\Lambda;\Omega^{1/2})$ hence by the first result, a bounded morphism of Hilbert modules between $C^*(G)$ and itself, and the result follow by multiplying on the left by $P(m-s)$ an on the right by $P(s)$.

\end{enumerate}
 
 \end{proof}

\section{An example: b-FIO}\label{sec:b-FIO}

Applying the previous constructions to the groupoid defined in \cite{DLR} to describe the pseudodifferential $\mathsf{S}$-calculus on manifolds with iterated fibred corners, we have at hands tools to determine what Lagrangian submanifolds of $T^*X^2_\pi$ are suitable to define a decent class of Fourier integral operators on manifolds with iterated fibred corners, and then equivalently on stratified spaces. Since this will be treated in details elsewhere, we content ourselves here with the elementary example of a manifold $X$ with boundary $\partial X=Y$ equipped with the trivial fibration $\pi : Y\to S=\{\cdot\}$. This example has also the advantage to allow a comparison of our constructions with existing ones \cite{Melrose1981}.

Let $X^2_b$ denote the stretched product of $X^2$ with respect to $Y^2$, $\beta : X^2_b\to X^2$ the blow-down map  \cite[Sections 4.1,4.2]{Melrose1993}. 
This a  manifold with corners  whose boundary consists in three boundary hypersurfaces 
\begin{equation}\label{eq:b-FIO-1}
 \mathrm{bf} = \beta^{-1}(Y^2), \quad  \mathrm{lb} = \overline{\beta^{-1}(Y\times\overset{\circ}{X})},\quad  \mathrm{rb} = \overline{\beta^{-1}( \overset{\circ}{X}\times Y)}. 
\end{equation}
The open submanifold
\begin{equation}\label{eq:b-FIO-2}
 \cG_b = X^2_b\setminus (\mathrm{lb}\cup\mathrm{rb})
\end{equation}
has a Lie groupoid structure \cite{MP, NWX} with unit space 
\begin{equation}\label{eq:b-FIO-3}
  \Delta_b=\overline{\beta^{-1}(\Delta_{\overset{\circ}{X}})}
\end{equation}  
where $\Delta_{\overset{\circ}{X}}$ is the diagonal of $\overset{\circ}{X}\times\overset{\circ}{X}$. 
The map $\beta$ induces a diffeomorphism between $\Delta_b$ and $\Delta_X$. 
The Lie algebroid is given by the stretched tangent bundle ${}^bTX$ and the anchor map is the natural map
\begin{equation}\label{eq:b-FIO-4}
 \iota : {}^bTX \longrightarrow TX.
\end{equation}
The groupoid structure of $\cG_b$  is the  unique $C^\infty$ extension of the pair groupoid 
structure of $\overset{\circ}{X}\times\overset{\circ}{X}$ and it coincides with the natural one on 
\begin{equation}\label{eq:b-FIO-5}
 G_b=\overset{\circ}{X}\times\overset{\circ}{X}\cup Y^2\times \RR_+^* \rightrightarrows X,
\end{equation}
through the diffeomorphism: 
\begin{align}\label{eq:b-FIO-6}
G_b\simeq \{(p,q,\lambda)\in X^2\times \RR_+^* \ ;\ x(q)=\lambda x(p) \} &\longrightarrow \cG_b\subset X^2_b \\
 (p,q,\lambda) & \longmapsto \begin{cases}
                                 (p,q) &\text{ if } p\not \in \partial X \\
                                 [(t,\lambda t, p',q')_{t\in[0,\epsilon]}]  &\text{ if } p \in \partial X 
                             \end{cases}
\end{align}
Above, $x$ is a defining function for the boundary of $X$ and points in a neighborhood $\cU$ of $\partial X$ are written using the convention $p=(x(p),p')$ with $p'\in Y$ after using the collar neighborhood theorem. Assuming that $\cU\simeq [0,\infty)\times Y$, we have a $C^\infty$ groupoid isomorphism
\begin{equation}\label{eq:b-FIO-7}
 \RR_+\rtimes \RR_+^*\times Y^2 \simeq (\cG_b)_{\cU}^{\cU} 
\end{equation}
given by 
\begin{align}\label{eq:b-FIO-8}
  \RR_+\rtimes\RR_+^*\times Y^2 &\longrightarrow   G_b\simeq \cG_b \\
  (x,\lambda,y,z) & \longmapsto   \begin{cases}
                                    (x,\lambda x , y,z) & \text{ if } x>0\\
                                     (\lambda, y,z)& \text{ if } x=0.
                                  \end{cases}
\end{align}
We assume in the sequel that a defining function is fixed, and we identify  $\cG_b$ with the left hand side of \eqref{eq:b-FIO-8}. In particular, we note $\mathrm{bf}$ the subgroupoid   
$ Y^2\times \RR_+^*$. For more details, see \cite{MP,NWX}.

\bigskip 
It is natural to focus on lagrangian conic submanifolds of $T^* X^2_b\setminus 0$ which are contained in $T^*\cG_b$. Indeed, since a diffeomorphism of $X$ onto itself maps necessarily the boundary to the boundary, the  inclusion $\Lambda\subset T^*\cG_b$ hold for  canonical relations $\Lambda$ given by  graphs of co-differential of diffeomophisms of $X$. For the same reason, one could moreover ask that the lagrangian submanifolds $\Lambda \subset T^*\cG_b\setminus 0$ under consideration have compact projections in $\cG_b$, but there is no immediate need for requiring this property.

\begin{prop}\label{prop:b-lagrangian}
Let $\Lambda\subset T^*\cG_b\setminus 0$ be a family $\cG_b$-relation and $i : \mathrm{bf}\hookrightarrow \cG_b$. Then
\begin{enumerate}
 \item $\Lambda\cap T^*(\overset{\circ}{X}\times \overset{\circ}{X})$ is a canonical relation contained in $(T^* \overset{\circ}{X}\setminus 0)\times (T^* \overset{\circ}{X}\setminus 0)$. 
 \item The projection $i^*\Lambda\subset T^*\mathrm{bf}\setminus 0$ is a family $\mathrm{bf}$-relation.
\end{enumerate}
\end{prop}
\begin{proof}
The follows from  Proposition \ref{prop:G-relation-passing-to-subgroupoid} since  $\overset{\circ}{X}\times \overset{\circ}{X}$ and $\mathrm{bf}$ are saturated subgroupoids (actually,  they are the only two leaves of $\cF_{\cG_b}$).
\end{proof}

In particular, given a family $\cG_b$-relation $\Lambda$, the $\cG_b$-FFIO associated with $\Lambda$ give by restriction to $\overset{\circ}{X}\times \overset{\circ}{X}$ and $\mathrm{bf}$, respectively, an ordinary Fourier integral operator acting on $\overset{\circ}{X}$ and a Fourier integral operator acting on $\RR_+^*\times Y$ which is invariant with respect to the homotethies in $\RR^*_+$. 

We end this section by  comparing $\cG_b$-relations and  {\sl boundary canonical relations} as defined by R. Melrose    \cite[chap III, Definition 2.19]{Melrose1981}. R. Melrose used the later to develop the analog of Hormander's theory of Lagrangian distributions in the framework of manifolds with boundary and totally characteristic operators. 

Let $\Lambda$ be a  boundary canonical  relation  with  $N=M=X$ in \cite[chap III, Definition 2.19]{Melrose1981}. 
Condition (2.10) of \cite[chap III, Definition 2.19]{Melrose1981} is the requirement that $\Lambda$ is a conic Lagrangian submanifold of $T^*X^2_b$ contained in $T^*\cG_b$ and Condition (2.15) there coincides with the admissibility condition (\enquote{no zeros condition}). Thus $\Lambda$ is $\cG_b$-relation. 

Now Condition (2.9) of \cite[chap III, Definition 2.19]{Melrose1981} reads
\begin{equation}\label{eq:b-lagrangian-7}
  T \Lambda + T T^*_{\mathrm{bf}}\cG_b = T T^*\cG_b. 
\end{equation}
This is equivalent to $T\mathrm{bf} + dp(T\Lambda) = T_{\mathrm{bf}}\cG_b$, which gives Condition \eqref{eq:transver-cond-G-rel-1} of Proposition \ref{prop:charact-family-G-rel} at any  point $(\gamma,\xi)\in \Lambda$ with $\gamma\in \mathrm{bf}	$. Since Condition \eqref{eq:transver-cond-G-rel-1} is here empty at any  point $(\gamma,\xi)\in \Lambda$ with $\gamma\in \overset{\circ}{X}\times \overset{\circ}{X}$, it follows that $\Lambda$ is a family $\cG_b$-relation.

It follows that boundary canonical relations as defined in \cite[chap III, Definition 2.19]{Melrose1981} are $\cG_b$-relations. 
We now analyse (2.13) of \cite[chap III]{Melrose1981}. We know that $\Lambda_1=i^*(\Lambda)\subset T^*\mathrm{bf}$ is a $\mathrm{bf}$-relation. The fiber bundle considered in  \cite[chap III, (2.13)]{Melrose1981} is here
\begin{equation}\label{eq:b-lagrangian-12}
  T^*\mathrm{bf} \supset F = \{ (\lambda,y,z,\nu,\xi,\eta)\ ; \ \nu=0\} \longrightarrow  T^*Y^2.
\end{equation}
Thus the first part of  Condition (2.13) consists in requiring that 
\begin{equation}\label{eq:first-half-2.13}
 TF + T\Lambda_1 = T T^*\mathrm{bf}
\end{equation}
and we can find $\cG_b$-relations that do not fulfill this equality. Indeed, let us fix $(p_0)\in Y^2$
and set 
\begin{equation}
 V= \{ (t,\lambda, p_0)\ ;\ t\in\RR_+,\ \lambda\in \RR_+^*\}\subset \cG_b. 
\end{equation}
Let $\cC$ be the  cone in $T^*\cG_b\setminus 0$ defined by 
\begin{equation}
  \cC = \{ (x,\lambda,y_1,y_2,\tau,\nu,\xi_1,\xi_2) \ ; \ \epsilon \vert\xi_1\vert< \vert\xi_2\vert< \epsilon^{-1}\vert\xi_1\vert\}
\end{equation}
for some $\epsilon>0$. Then set 
\begin{equation}\label{eq:Gb-relation-not-melrose-bdy-relation}
 \Lambda = N^*V \cap \cC=\{ (t,\lambda, p_0, 0,0,\xi_1,\xi_2) \ ; \  t\in\RR_+,\ \lambda\in \RR_+^*,\ \epsilon \vert\xi_1\vert< \vert\xi_2\vert<  \epsilon^{-1}\vert\xi_1\vert\}.
\end{equation}
We are going to check that $\Lambda$ is a $\cG_b$-relation which does not satisfy \eqref{eq:first-half-2.13}. Firstly, $\Lambda$ is obviously a conic Lagrangian submanifold. A straight computation shows that $A^*\cG_b$ is the subset of $T^*_{\cG_b^{(0)}}\cG_b$ consisting in elements of the form 
\begin{equation}\label{eq:b-lagrangian-2}
 (x,1,y,y,0,\nu,-\xi,\xi)\in T^*_{\cG_b^{(0)}}\cG_b. 
\end{equation}
and that 
\begin{equation}\label{eq:b-lagrangian-3}
 s_\Gamma : (x,\lambda,y,z,\tau,\nu, \xi,\eta) \longmapsto (x\lambda,1,z,z,0,\lambda\nu, -\eta,\eta) 
\end{equation}
\begin{equation}\label{eq:b-lagrangian-4}
 r_\Gamma : (x,\lambda,y,z,\tau,\nu, \xi,\eta) \longmapsto (x,1,y,y,0,\lambda\nu-x\tau, \xi,-\xi) 
\end{equation}
Applying these formulae to points in \eqref{eq:Gb-relation-not-melrose-bdy-relation} shows that $\Lambda$ is a $\cG_b$-relation. 
Next, let us check Condition \eqref{eq:transver-cond-G-rel-1}  of Proposition \ref{prop:charact-family-G-rel}, that is, the transversality between $p : \Lambda\to \cG_b$ and the foliation $\cF_{\cG_b}$. The condition is empty at any interior point so we focus on boundary points, that is on points in $\Lambda\cap p^{-1}(\mathrm{bf})$. We have
\begin{equation}\label{eq:b-lagrangian-5}
 \gamma=(0,\lambda,y)\in\mathrm{bf}=\{0\}\times\RR_+^*\times Y^2,\quad T_\gamma\cF_{\cG_b}= T_\gamma\cG_{b\, s(\gamma)}+T_\gamma\cG_{b}^{r(\gamma)}=\{0\}\times T\RR_+^*\times Y^2 \subset T_{\mathrm{bf}} \cG_b.
\end{equation}
Using the expression \eqref{eq:Gb-relation-not-melrose-bdy-relation} we obtain immediately
\begin{equation}\label{eq:b-lagrangian-6}
 dp(T_{\gamma,\xi}\Lambda) +  T_\gamma\cF_{\cG_b}= T_\gamma\cG_{b}, \quad \forall (\gamma,\xi)\in\Lambda\cap p^{-1}(\mathrm{bf}),
\end{equation}
therefore $\Lambda$ is a family $\cG_b$-relation. Now consider 
\begin{equation}
\Lambda_1=i^*\Lambda = \{ (\lambda, p_0, 0,\xi_1,\xi_2) \ ; \ \lambda\in \RR_+^*,\ \epsilon \vert\xi_1\vert< \vert\xi_2\vert<  \epsilon^{-1}\vert\xi_1\vert\}\subset T^*\mathrm{bf}\setminus 0.
\end{equation}
Thus 
\begin{equation}
 T\Lambda_1 = \{ (\lambda, p_0, 0,\xi; u,0,0,\zeta) \ ; \ \lambda\in \RR_+^*,\ \epsilon \vert\xi_1\vert< \vert\xi_2\vert<  \epsilon^{-1}\vert\xi_1\vert,\ \zeta\in T_{p_0}^*Y^2.\}
\end{equation}
Since 
\begin{equation}
TF =\{ (\lambda,p,0,\xi ; u,v,0,\zeta)\ ; \ (\lambda,u)\in T\RR_+^*,\ (p,v)\in TY^2,\ \xi,\zeta\in T_{p}^*Y^2\}.
\end{equation}
We now see that $T\Lambda_1+TF\not = T(T^*\mathrm{bf})$, therefore $\Lambda$ is not a boundary canonical relation in the sense of  \cite[chap III, Definition 2.19]{Melrose1981}.

 \bibliographystyle{plain}  
\bibliography{biblio_fougro.bib}

\end{document}